\documentclass[12pt]{amsart}

\usepackage[top=3cm,bottom=3cm,outer=2cm,inner=2cm,marginpar=1.5cm]{geometry}
\usepackage[all]{xy}
\usepackage[utf8]{inputenc}
\usepackage{hyperref}
\usepackage{amssymb}
\usepackage{xcolor}

\hypersetup{
	colorlinks = true,
	linkcolor = blue,
	anchorcolor = black,
  	citecolor = blue,
	filecolor = red,
	urlcolor = magenta,
	pdfauthor=author
	menucolor = red
}

\def\skewind#1#2#3#4%
{\tiny{\vtop{\hbox{$#1,\ldots,#2,\ldots#3$}\vskip-.8mm
\hbox{$\phantom{#1,\ldots,\,}{|\atop #4}$}}}}
\newtheorem{theorem}{Theorem}[section]
\newtheorem{maintheorem}{Main Theorem}
\newtheorem{lemma}[theorem]{Lemma}
\newtheorem{corollary}[theorem]{Corollary}
\newtheorem*{corollary*}{Corollary}
\newtheorem{proposition}[theorem]{Proposition}

\theoremstyle{definition}
\newtheorem{definition}[theorem]{Definition}
\newtheorem{remark}[theorem]{Remark}

\numberwithin{equation}{section}

\newcommand{\C}{{\mathbb C} }

\newcommand{\cA}{{\mathcal A} }

\newcommand{\cE}{{\mathcal E} }

\newcommand{\cL}{{\mathcal L} }

\newcommand{\cO}{{\mathcal O} }

\newcommand{\cT}{{\mathcal T} }

\newcommand{\cX}{{\mathcal X} }

\newcommand{\cK}{{\mathcal K} }
\newcommand{\wh}{\widehat}

\newcommand{\pt}{\partial}

\newcommand{\paren}[1]{\left(#1\right)}
\newcommand{\bparen}[1]{\left[#1\right]}
\newcommand{\pd}[2]{\frac{\partial#1}{\partial#2}}

\newcommand{\norm}[1]{\left\|#1\right\|}

\newcommand{\ov}[1]{\overline{#1}}
\newcommand{\inner}[1]{\left\langle{#1}\right\rangle}
\newcommand{\ind}[3]{_{#1\phantom{#2}#3}^{\phantom{#1}#2}}
\newcommand{\indd}[4]{_{#1\phantom{#2}#3}^{\phantom{#1}#2\phantom{#3}#4}}

\newcommand{\set}[1]{\left\{#1\right\}}

\newcommand{\ip}[1]{\mathrm{Im}\;s}

\newcommand{\lowerindA}[6]
{
	\ind{}{#1}{
		\tiny\vtop{
			\hbox{$#2,\ldots,#3,\ldots,#4,#5$}\vskip-.8mm
			\hbox{$\phantom{\alpha_1,\ldots,}{|\atop#6} $}
		}
	}
}
\newcommand{\lowerindB}[6]
{
	\ind{}{#1}{
		\tiny\vtop{
				\hbox{$#2,#3,\ldots,#4,\ldots,#5$}
				\vskip-.8mm
				\hbox{$\phantom{#2,#3,\ldots,}{|\atop#6}$}
		}
	}
}

\newcommand\vartextvisiblespace[1][.5em]{%
  \makebox[#1]{%
    \kern.07em
    \vrule height.3ex
    \hrulefill
    \vrule height.3ex
    \kern.07em
  }
}

\def\ol#1{{\overline{#1}}}

\def\ke{K{\"a}h\-ler-Ein\-stein}
\def\ks{Ko\-dai\-ra-Spen\-cer}
\def\ka{K{\"a}h\-ler}
\def\wp{Weil-Pe\-ters\-son}

\def\ii{\sqrt{-1}}
\def\ddb{\sqrt{-1}\partial\overline{\partial}}
\def\C{\mathbb{C}}
\def\cinf{C^\infty}

\def\we{\wedge}
\def\re{\mathrm{Re}}
\def\im{\mathrm{Im}}
\def\id{\mathrm{id}}
\def\dbar{\overline\partial}
\def\he{Her\-mite-Ein\-stein}

\def\CC{\mathbb{C}} 
\def\re{\mathrm{Re}} 
\def\dbar{\ol\partial}

\def\ke{K\"ahler-Einstein }

\begin{document}

\title[Curvature of higher direct images]{Curvature of higher direct images of sheaves of twisted holomorphic forms}

\author{Young-Jun Choi}
\address{Department of Mathematics, Pusan National University, 2, Busandaehak-ro 63beon-gil, Geumjeong-gu, Busan, 46241, Republic of Korea}
\email{youngjun.choi@pusan.ac.kr}

\author{Georg Schumacher}
\address{Fachbereich Mathematik und Informatik,+++
Philipps-Universit\"at Marburg, Lahnberge, Hans-Meerwein-Straße, D-35032
Marburg,Germany}
\email{schumac@mathematik.uni-marburg.de}
\date{\today}

\keywords{higher direct images, curvature formula, families of compact K\"ahler manifolds, deformation of compact K\"ahler manifolds, deformation of vector bundles}

\subjclass[2010]{32L10, 14D20, 32G05}

\begin{abstract}
	This paper investigates the curvature properties of higher direct images $ R^qf_*\Omega_{\mathcal{X}/S}^p(\mathcal{E})$, where $f: \mathcal{X} \rightarrow S$ is a family of compact Kähler manifolds equipped with a hermitian vector bundle $\mathcal{E} \rightarrow \mathcal{X}$. We derive a general curvature formula and explore several special cases, including those where $p + q = n$, $q = 0$, and $p = n$, with $\mathcal{E}$ being a line bundle. Furthermore, the paper examines the curvature in the context of fiberwise hermitian flat cases, families of Hermite-Einstein vector bundles, and applications to moduli spaces and Weil-Petersson metrics, providing some insight into their geometric and analytical properties.
\end{abstract}

\maketitle

\section{Introduction}
A strong tool for the study of holomorphic families $f:\cX \to S$ of compact \ka\ manifolds is the differential geometric investigation of twisted Hodge sheaves $R^qf_*\Omega^p_{\cX/S}(\cE)$, where $\cE$ denotes a holomorphic line bundle or vector bundle equipped with a hermitian metric. The aim is to compute the curvature in terms of intrinsic quantities namely holomorphic/harmonic sections and distinguished/harmonic \ks\ tensors together with tensors that are derived from these.

In this sense the total space $\cX$ is equipped with a $d$-closed real $(1,1)$-form that restricts to \ka\ forms on the fibers, and assumptions on the hermitian vector bundle $\cE$ only refer to properties of the restrictions to the fibers of $f$.

In some cases, where vanishing theorems hold, these direct image sheaves are always locally free due to the Grauert-Grothendieck comparison theorem, otherwise under given assumptions the pertinent direct image sheaves are shown to be locally free or the study is restricted to the locally free locus. 

The case, where $p$ is equal to the fiber dimension $n$ and $q=0$ was treated by Mourougane in \cite{Mourougane1997}: If $\cX$ is projective algebraic, and $\cL$ an ample line bundle on $\cX$, then $f_*({\cK_{\cX/S}}\otimes\cL)$ was shown to be locally free and ample, and according to the result by Mourougane and Takayama \cite{Mourougane_Takayama2007} it carries a smooth hermitian metric of Griffiths positive curvature (also an analogous result for compact \ka\ manifolds was shown in \cite{Mourougane_Takayama2007}).
Berndtsson proved in \cite{Berndtsson2009} that the bundle $f_*({\cK_{\cX/S}}\otimes\cL)$ is Nakano-(semi-)positive, if $\cL$ is a (semi-)positive line bundle, where again $f:\cX\to S$ is a proper, smooth map of not necessarily compact \ka\ manifolds.
He also computed the explicit curvature formula in \cite{Berndtsson2011}.

Concerning higher cohomology and applications to moduli and higher order \ks\ maps, in particular hyperbolicity, we mention \cite{Schumacher2012}. The case $R^{n-p}f_*\Omega^p_{\cX/S}(\cK_{\cX/S}^{\otimes m})$, $m>0$ for $\cK_{\cX/S}$ (relatively) positive is critical with respect to the Kodaira-Nakano vanishing theorem. Naumann observed that the curvature of the direct image $R^{n-p}f_*\Omega^p_{\cX/S}(\cL)$ is computable, when $\cL$ denotes a holomorphic line bundle on $\cX$ that is fiberwise positive \cite{Naumann2021}. Naumann replaces the geodesic curvature related to the relative \ke\ form by the geodesic curvature related to the Chern form of $\cL$. His result contains and generalizes \cite{Berndtsson2009,Berndtsson2011}. Using somewhat different methods this situation was also treated by Berndtsson, P\u{a}un, and Wang in \cite{Berndtsson_Paun_Wang2022}.

Allowing singular fibers of $f$, singular hermitian metrics on line bundles, and Bergman metrics, further important results exist; to mention some we refer to Berndtsson-P\u{a}un \cite{Berndtsson_Paun2008} and  P\u{a}un-Takayama \cite{Paun_Takayama2018}.

For holomorphic vector bundles $\cE$, the sheaves $R^qf_*\Omega^p_{\cX/S}(\cE)$ were treated by Mourougane in \cite{Mourougane_Takayama2008}:  In the case of smooth proper families with $\cX$ not necessarily compact, and Nakano-semi-positive vector bundles $\cE$ they proved that local freeness and Nakano-semi-positivity in \cite{Mourougane_Takayama2008}. Liu and Yang also computed several curvature formulas of $f_*(\cK_{\cX/S}\otimes\cE)$ in \cite{Liu_Yang2014} after Liu, Sun, and Yang studied the cohomology $H^{n,q}(\cX, S^kE\otimes \det E)$ on a compact Kähler manifold $X$ equipped with an ample vector bundle $E$ in \cite{Liu_Sun_Yang2013}. They showed that, if $E$ is Griffiths-positive, then $S^kE\otimes \det E$ is both Nakano-positive and dual Nakano-positive for any $k\geq0 $. The positivity result of Mourougane and Takayama was generalized by Wang in \cite{Wang2016Ar} for $q$-semipositive vector bundles on the total space. The curvature problem where $\cE$ is a Higgs bundle was solved by Hu and Huang in \cite{Hu_Huang2021}.

Here, the aim is to treat the most general situation when computing the curvature of $R^qf_*\Omega^p_{\cX/S}(\cE)$.

Our result  gives an answer in terms of intrinsic quantities: Namely in terms of harmonic $(p,q)$-forms with values in $\cE_s$ on the respective fibers $\cX_s$,  a tangent vector of $S$ is represented by the induced \ks\ tensor, and a further ingredient is the geodesic curvature for tangent vectors on the base. The result is compatible with \cite{Schumacher2012,Geiger_Schumacher2017}, and it comprises the situation of \cite{Liu_Yang2014} for $p=n$.

Our result also contains the computation of the curvature of $R^qf_*\cE$ for families of \he\ bundles. It yields a simpler formula for families of stable bundles when applied to $End(\cE)$, a case that is relevant for the \wp\ metric on moduli spaces (cf.\ \cite{Schumacher_Toma1992}). It also applies to higher order \ks\ maps.

We emphasize the special cases of $f_*\Omega^p_{\cX/S}(\cL)$ and $R^qf_*\cK_{\cX/S}(\cL)$, where again we obtain a formula depending on intrinsic terms.

Under the additional assumption of fiberwise flat hermitian metrics more can be shown. For moduli spaces of such structures the curvature of the \wp\ metric is semi-positive.
\bigskip
x
\noindent\textit{Acknowledgements.} The authors would like to thank Philipp Naumann for several discussions.
The work was supported by the National Research Foundation of Korea (NRF) grant funded by the Korea government(MSIT) (No. 2023R1A2C1007227) and by Deutsche Forschungsgemeinschaft DFG (Schu 771/5-2).
\bigskip

This paper is organized as follows.

{\hypersetup{linkcolor=black} \tableofcontents}

\section{Notation and statement of the result}
Let $f:\cX\rightarrow S$ be a smooth proper holomorphic map between complex manifolds, and let $\omega:=\omega_\cX$ be a $d$-closed real $(1,1)$-form on $\cX$ whose restriction toi fibers are positive. Let $\cE\rightarrow\cX$ be a holomorphic vector bundle equipped with a hermitian metric $h$.

In this paper, we consider the curvature of the higher direct image sheaves
\begin{equation*}
R^qf_*\Omega^p_{\cX/S}(\cE).
\end{equation*}
Here, we have to assume that these sheaves are locally free.
For $s\in S$ let $X=\cX_s= f^{-1}(s)$ be the fiber of $s$, and let $\omega_s:=\omega_{\cX_s} = \omega_\cX\vert_{\cX_s}$.

At this point, we will state the result, and indicate the meaning of the necessary notation.

Let $s=(s^1,\ldots,s^k)$ be local holomorphic coordinates on $S$, and pick local coordinates $(z,s)=(z^1,\ldots,z^n,s^1,\ldots,s^k)$ on $\cX$ such that $f(z,s)=s$. We denote the coordinate functions by $z^\alpha$ and $s^i$ resp. We set $\pt_\alpha=\pt/\pt z^\alpha$, and $\pt_i = \pt/\pt s^i$. Let $\rho_s: T_sS \to H^1(\cX_s, T_{\cX_s})$ denote the \ks\ map.

The form $\omega_\cX$ determines a differentiable splitting of the sequence
\begin{equation}\label{eq:dolks}
0 \to T_{\cX/S} \to T_{\cX} \to f^*T_S \to  0,
\end{equation}
assigning to a tangent vector $\pt/\pt s^i$ of $S$ the {\em horizontal lift} $v_i$ to the total space, which consists of tangent vectors that are orthogonal to the respective fiber. Now the restriction
$$
A_i:=(\ol\pt v^i)|\cX_s=A\ind{i}{\alpha}{\ol\beta}(z,s)\pd{}{z^\alpha}\otimes dz^{\ol\beta}
$$
determines a {\em distinguished representative} of the \ks\ class $\rho_s(\pt/\pt s^i)$ in terms of Dolbeault cohomology. If $f:\cX\to S$ is a family of {\em canonically polarized} manifolds such that the restrictions of $\omega_\cX$ to the fibers are K\"ah\-ler-Einstein, then the above distinguished representatives are actually {\em harmonic} \cite{Schumacher2012}.

For  $s \in S$ the cup product together with the
contraction defines mappings
\begin{eqnarray*}
A\ind{i}{\alpha}{\ol\beta}(z,s)\pd{}{z^\alpha}\otimes dz^{\ol\beta}
\cup\vartextvisiblespace:
\cA^{0,0}(\cX_s,\Omega^p_{\cX_s}(\cE)) &\to&
\cA^{0,1}(\cX_s,\Omega^{p-1}_{\cX_s}(\cE))
\;\;\;\text{for}\;\;\;
1\le p\le n.
\end{eqnarray*}

We will apply the above products to harmonic sections. In general the results are not harmonic. We denote the global $L^2$-inner product by angled brackets, and the pointwise $L^2$ inner product by a dot. By abuse of notation we denote a tangent vector of $S$ at a point $s$ by $\pt/\pt s$, or just $\pt_s$ and the distinguished representative of $\rho_s(\pt/\pt s)$ by $A_s$.
The pointwise squared norm $c(\omega)$ of the {\em horizontal lift} $v$ with respect to $\omega$ is also called {\em geodesic curvature}.

We denote the Chern connection of $(\cE,h)$ by $\nabla^\cE$.
The connection form will be denoted by $\theta$.
As usual $\Theta:=\Theta_h(\cE)=\ol\pt\theta$ denotes the (purely imaginary) curvature form of $\cE$, whereas $\Theta_\cE(-,-)$ is the corresponding hermitian form.

We denote by $\Box_\dbar$ (resp. $\Box_\pt$) the $\dbar$-Laplacian (resp. $\partial$-Laplacian) on $\cX_s$ with non-negative eigenvalues on $\cA^{(p,q)}(\cX_s,\cE\vert_{\cX_s})$

For any $p$ and $q$, the locally free sheaf $R^qf_*\Omega^p_{\cX/S}(\cE)$ carries a canonical hermitian metric which is defined by the $L^2$ product of harmonic representatives.

Infinitesimal deformations of a pair $(X,E)$ can be characterized in terms of the Atiyah bundle $\Sigma_X(E)$ as elements of its first cohomology (see Section~\ref{S:pairs} for details). A distinguished representative of the \ks\ class is a pair $(\zeta_s, A_s)$, where
$$
\zeta_s= \eta_s + A\ind{s}{\alpha}{\ol\beta}\theta_\alpha dz^{\ol\beta},
$$
namely the tensor $\eta_s$ is the exterior derivative of the contraction with the horizontal lift $v$ of the respective tangent vector with the curvature form of $\cE$, restricted to the fiber, i.e.
$$
\eta_s = - (v \cup \Theta)|_{\cX_s}= -(\Theta_{s\ol\beta} + a\ind{s}{\alpha}{}\Theta_{\alpha\ol\beta})dz^{\ol\beta} .
$$
Now the curvature of the natural metric on the higher direct image sheaf can be expressed in terms of intrinsic quantities.
\begin{maintheorem}\label{T:main_theorem}
For any  family $\cE$ of hermitian holomorphic vector bundles over $\cX\to S$ the curvature tensor on $R^qf_*\Omega^p_{\cX/S}(\cE)$ is given by
\begin{align*}
R\paren{\pt_s,\pt_{\ol s},\psi,\ol{\psi}}
&=
\inner{L_{[v,\ol{v}]}\psi,\psi}
+
\inner{\Theta(v,\ol{v})\psi,\psi}
-
\inner{G_\dbar(w_s),w_s}
+
\inner{G_\dbar(w_{\ol s}),w_{\ol s}}
\\
&\hspace{.5cm}
+
\inner{A_s\cup\psi,A_s\cup\psi}
-
\inner{A_{\ol s}\cup\psi,A_{\ol s}\cup\psi},
\end{align*}
where $G_\dbar$ denotes the Green's operator for $\cE$-valued forms on $\cX_s$, and $w_s$, $w_{\ol s}$ are given as follows.
\begin{equation*}
		w_s
		=
		\pt(A_s\cup\psi)
		+
		A_s\cup\pt\psi
		+
		\eta_s
			\we\psi
\end{equation*}
and
\begin{equation*}
		w_{\ol s}
		=
	(-1)^p\pt^*(A_{\ol s}\cup\psi)
	+
	(-1)^pA_{\ol s}\cup\pt^*\psi
	+
	[\Lambda,\ii\eta_{\ol s}]\psi.
\end{equation*}
\end{maintheorem}
We note that the general formula follows by polarization.
It is remarkable to mention that $L_{[v,\ol v]}$ contains derivatives only along the fiber direction.
More precisely, Proposition~\ref{P:Lie_derivative_commutator} states that it satisfies
\begin{equation*}
	\inner{L_{[v,\ol v]} \chi,\psi}
	=
	\inner{c(\omega)\,\Box_\pt\chi,\psi}
	-
	\inner{c(\omega)\,\partial\chi,\partial\psi}
	-
	\inner{c(\omega)\,\pt^*\chi,\pt^*\psi}
	+
	\inner{\chi,\pt\ol{c(\omega)}\we\pt^*\psi}
	+
	\inner{\pt c(\omega)\we\pt^*\chi,\psi}.
\end{equation*}
Note that $\eta_s$ is closely related to the variation of the family of vector bundles $\cE\rightarrow \cX\rightarrow S$. (See Section~\ref{S:pairs} and \cite{Schumacher_Toma1992} for the case of $\cX= X\times S \to S$).
The theorem can be generalized to smooth maps of reduced complex spaces.
\medskip

Next, we assume that $(\cE,h)$ is a hermitian line bundle $(\cL,h)$.
If the line bundle $(\cL,h)$ is trivial, then the formula specializes to Griffiths' Theorem on the curvature of Hodge bundles \cite{Griffiths1970,Griffiths_Tu1984}.
In the case that $(\cL,h)$ is fiberwise positive or negative and $p+q=n$, the curvature formula was obtained by \cite{Naumann2021,Berndtsson_Paun_Wang2022}.
The main theorems deal with the curvature formula when $p+q\neq n$.
More precisely, we have the theorems stated below.

\begin{maintheorem}[Theorem~\ref{T:positive_p0}]
\label{T:canpol}
Let $f:\cX\rightarrow S$ be a family of compact Kähler manifolds with a hermitian line bundle $(\cL,h)$ such that $\ii\Theta_h(\cL)$ is positive on fibers.
Then the curvature of $f_*\Omega^p_{\cX/S}(\cL)$ is given by
\begin{eqnarray*}
R(\pt_s,\pt_{\ol s},\psi,\ol\psi)
&=&
(n-p+1)
\inner{c(\omega)\psi,\psi}
-
\inner{c(\omega)\pt\psi,\pt\psi}
\\
&&
+
(n-p+1)
\inner{
	(\Box_{\dbar}+1)^{-1}
	(A_s\cup\psi)
	,
	A_s\cup\psi
}
-
\inner{
	\Box_{\dbar}^{-1}
	(A_s\cup\pt\psi)
	,A_s\cup\pt\psi
}.
\end{eqnarray*}
\end{maintheorem}

For $p=n$ again $\pt\chi=0$, and the formula yields the corresponding formula from \cite{Schumacher2012,Berndtsson2011}.

\begin{maintheorem}
[Theorem~\ref{T:curvature_formula_(n,q)}]
Let $f:\cX\rightarrow S$ be a family of compact Kähler manifolds with a hermitian line bundle $(\cL,h)$ such thati $\ii\Theta_h(\cL)$ is negative on fibers.
The curvature of $R^qf_*\Omega^n_{\cX/S}(\cL)$ is given by
\begin{eqnarray*}
	R(\pt_s,\pt_{\ol s},\psi,\psi)
	&=&
	(q-1)
	\inner{c(\omega)\psi,\psi}
	+
	\inner{c(\omega)\pt^*\psi,\pt^*\psi}
	\\
	&&-
	(q+1)
	\inner{
		(\Box_{\dbar}-1)^{-1}
		(A_s\cup\psi)
		,
		A_s\cup\psi
	}
	+
	\inner{
		\Box_\dbar^{-1}
		\paren{A_{\ol s}\cup\pt^*\psi}
		,
		\paren{A_{\ol s}\cup\pt^*\psi}
	}.
\end{eqnarray*}
\end{maintheorem}
\medskip

General fiberwise flat vector bundles are studied in Section~\ref{S:flat}. Under the assumption that the vector bundle $\cE$ is Nakano-semipositive on the total space Theorem~\ref{T:flat} states that
\begin{equation*}
R(\pt_s,\pt_{\ol s},\psi,\ol\psi)
=
\inner{\Theta_\cE(v,\ol v)\psi,\psi}
+
\norm{H(A_s\cup\psi)}^2
-
\norm{H(A_{\ol s}\cup\psi)}^2,
\end{equation*}
where $\inner{\Theta_\cE(v,\ol v)\psi,\psi}=\Theta_{s\ol s}\norm{\psi}^2$ (For the line bundle case see also \cite{Berndtsson_Paun_Wang2022}).

Section~\ref{S:HE} deals with families of Hermite-Einstein vector bundles over a compact Kähler manifold (Theorem~\ref{curvHE}) containing the case $p=0$, from \cite{Geiger_Schumacher2017}. When applying the formula to the curvature of the \wp\ metric on the moduli space of stable vector bundles, the term containing $H(\Theta_{s\ol s})$ is not present implying the result from \cite{Schumacher_Toma1992}. In the preceding section families of flat hermitian bundles had been studied (Theorem~\ref{T:flat}).
\begin{maintheorem}[Theorem~\ref{curvflat}]
	Given a family of fiberwise flat hermitian bundles for $p=1,2$ the curvature of $R^pf_*End(\cE)$ satisfies
\begin{eqnarray*}
   R(\pt_s,\pt_{\ol s},\eta^p_s,\eta^p_{\ol s})&=&
\inner{G_\dbar(\ii\Lambda[\eta_s,\eta_{\ol s}]) \eta^p_s  , \eta^p_s}\\
& & \qquad +
\inner{G_\dbar(\ii \Lambda[\eta_{\ol s}, \eta^p_s]),\ii \Lambda[\eta_{\ol s}, \eta^p_s]}
\end{eqnarray*}
\end{maintheorem}

This theorem implies the following for families of fiberwise flat hermitian bundles.
\begin{corollary*}
	On the moduli space of simple, flat hermitian bundles the holomorphic sectional curvature of the Weil-Petersson metric is semi-positive.
\end{corollary*}

In this context fiber integrals of differential forms are of interest. The fiber integral of a form of type $(n+k,n+\ell)$ is a differential form of type $(k,\ell)$, when $n$ is the dimension of the fibers, and fiber integration commutes with exterior differentiation. However, when representing holomorphic sections of $R^qf_*\Omega^p_{\cX/S}(\cE)$ by fiberwise harmonic differential forms, it is possible to represent these by global $\ol\pt$-closed $(0,q)$ forms with values in $\Omega^p_{\cX/S}(\cE)$ (which restrict to harmonic $(p,q)$-forms), but in general not by global such $(p,q)$-forms. Therefore, when dealing with fiber integrals, it is natural to apply Lie-derivatives with respect to lifts of tangent vectors of the base space to sections of $\cE$, $\cT_{\cX/S}$, $\Omega_{\cX/S}$, and related bundles, instead of using exterior derivatives.

\section{Preliminaries}\label{S:pre}

\subsection{Hermitian vector bundle}\label{se:hermbdl}
Let $(X,\omega)$ be a compact Kähler manifold and let $(E,h)$ be an hermitian vector bundle of rank $r$ on $X$.
In terms of local holomorphic coordinates $(z^1,\dots,z^n)$, we write the Kähler form $\omega$ by
\begin{equation*}
\omega=\ii g_{\alpha\ol\beta}dz^\alpha\wedge dz^{\ol\beta}.
\end{equation*}
The curvature tensor is
\begin{equation*}
	R_{\alpha\ol\beta\gamma\ol\delta}
	=
	-
	\pt_\alpha\pt_{\ol\beta}g_{\gamma\ol\delta}
	+
	g^{\ol\tau\sigma}
	\pt_\alpha g_{\gamma\ol\tau}\pt_{\ol\beta}g_{\sigma\ol\delta},
\end{equation*}
and its Ricci curvature tensor is
\begin{equation*}
	R_{\alpha\ol\beta}
	=
	g^{\ol\delta\gamma}R_{\alpha\ol\beta\gamma\ol\delta}.
\end{equation*}
The Chern connection $\nabla^E$ on $(E,h)$ is defined to be the connection which is compatible with both the hermitian metric $h$ and the complex structure.
With respect to a local holomorphic frame $\{e_1,\dots,e_r\}$ of $E$, the connection $\nabla^E$ is written as
\begin{equation*}
	\nabla^E e_i=\theta\ind{i}{j}{}e_j
\end{equation*}
where $\theta:=\{\theta\ind{i}{j}{}\}$ is the connection 1-form, which is defined by
$\theta\ind{i}{j}{}=\pt h_{i\ol k}\cdot h^{\ol kj}=:\pt h\cdot h^{-1}$.
The curvature $\Theta_h(E)$ is defined by the $\mathrm{End}(E)$-valued $(1,1)$-form $\Theta_h(E)=\dbar\paren{\pt h\cdot h^{-1}}$ and
\begin{equation*}
	\ii\Theta_h(E)
	=
	-\ii\Theta\ind{i}{j}{\alpha\ol\beta}e^i\otimes e_j\otimes dz^\alpha\we dz^{\ol\beta}.
\end{equation*}
Let $A^{p,q}(E)$ be the space of $E$-valued smooth $(p,q)$-forms (in the sense of class $\cinf$), and let $\chi$ and $\psi$ be local sections of $A^{p,q}(E)$ which are written as
\begin{equation*}
\chi
=
\frac{1}{p!q!}\chi\ind{}{i}{A_p\ol B_q}e_i\otimes dz^{A_p}\wedge dz^{\ov B_q}
\;\;\;
\text{and}
\;\;\;
\psi
=
\frac{1}{p!q!}\psi\ind{}{j}{C_p\ol D_q}e_j\otimes dz^{C_p}\wedge dz^{\ov D_q}
\end{equation*}
where $dz^{A_p} = dz^{\alpha_1} \we \ldots\we dz^{\alpha_p}$, $dz^{\ol B_q}= dz^{{\ol\beta}_1}\we \ldots dz^{{\ol\beta}_q}$ and
\begin{equation*}
\chi\ind{}{i}{A_p\ol B_q}=\chi\ind{}{i}{\alpha_1,\ldots,\alpha_p,\ol\beta_1,\ldots,\ol\beta_q}
\end{equation*}
are skew-symmetric in $\alpha_1,\ldots,\alpha_p$ and in $\ol\beta_1,\ldots,\ol\beta_q$.
The same applies to $\psi$.
The pointwise hermitian form $(\cdot,\cdot)_{g,h}$ is defined by
\begin{align*}
	(\chi,\psi)_{g,h}
	&=
	\frac{1}{(p!)^2(q!)^2}\chi\ind{}{i}{A_p\ol B_q} \ol{\psi\ind{}{j}{C_p \ol D_q}}
	\cdot h_{i\ol\jmath}\cdot g^{\ol B_q D_q}\cdot g^{\ol C_p A_p} \\
	&=
	\frac{1}{p!q!}
	\chi\ind{}{i}{\alpha_1,\ldots,\alpha_p,\ol\beta_1,\ldots,\ol\beta_q}
	\ol{\psi\ind{}{j}{\gamma_1,\ldots,\gamma_p,\ol\delta_1,\ldots,\ol\delta_q}}
	\cdot h_{i\ol\jmath}\cdot g^{\ol\gamma_1\alpha_1}\cdots g^{\ol\gamma_p\alpha_p}
	g^{\ol\beta_1\delta_1}\cdots g^{\ol\beta_q\delta_q}
\end{align*}
where $g^{\ol{C}_p A_p} = \det (g^{\ol\gamma_i \alpha_j})_{i,j = 1\ldots p}$ and $h_{i\ol\jmath}=h(e_i,e_j)$.

The Chern connection on the hermitian vector bundle $(E,h)$ defines the operator $\nabla^E: A^k(E)\rightarrow A^{k+1}(E)$ which is decomposed as $\nabla^E=\partial_h+\dbar$ where
\begin{equation*}
\partial_h:A^{p,q}(E)\rightarrow A^{p+1,q}(E)
\;\;\;\text{and}\;\;\;
\dbar: A^{p,q}(E)\rightarrow A^{p,q+1}(E).
\end{equation*}
The $(1,0)$-part $\pt_h$ of $\nabla^E$ depends on the hermitian metric $h$. If there is no confusion, we will denote $\pt_h$ by $\pt$.
The local computations of $\partial_h$ and $\ol\partial$ are well-known. For details we refer to \cite{Siu1982,Varolin2010}.
For the reader's convenience, we give the local expressions for the skew-symmetric coefficients.
In local coordinates, $\ol\partial$ is given by
\begin{align*}
	\paren{
		\dbar\psi
	}\ind{}{i}{\alpha_1,\ldots,\alpha_p,\ol\beta_0,\ldots,\ol\beta_q}
	&=
	(-1)^p\sum_{\nu=0}^q(-1)^\nu
	\psi\ind{}{i}{\alpha_1,\ldots,\alpha_p,\ol\beta_0,\ldots,\widehat{\ol\beta}_\nu,\ldots,\ol\beta_q;\ol\beta_\nu}
	\\
	&=
	(-1)^p
	\paren{
		\psi\ind{}{i}{\alpha_1,\ldots,\alpha_p,\ol\beta_1,\ldots,\ol\beta_q;\ol\beta_0}
		-
		\sum_{\nu=1}^q
		\psi\ind{}{i}{\tiny{\vtop{\hbox{$\alpha_1,\ldots,\alpha_p,\ol\beta_1,\ldots,\ol\beta_0,\ldots,\ol\beta_q;\ol\beta_\nu$}\vskip-.8mm	\hbox{$\phantom{\alpha_1,\ldots,\alpha_p,\ol\beta_1,\ldots,}{|\atop \nu } $}}}}
	},
\end{align*}
where the semi-colon notation is used for a covariant derivative.
Here either the covariant derivative with respect to $\nabla^E$ or the covariant derivative with respect to $\nabla^E$ and the Kähler (Levi-Civita) connection $\nabla^X$ on $(X,\omega)$ can be taken, because the Kähler connection is torsion-free so that the Christoffel symbols are symmetric.

The formal adjoint operator $\dbar_h^*$ of $\dbar$ is written as follows.
\begin{equation*}
	\paren{\dbar_h^*\psi}\ind{}{i}{A_p,\ol\beta_1,\ldots,\ol\beta_{q-1}}
	=
	(-1)^{p+1}g^{\ol\beta\alpha}
	\psi\ind{}{i}{A_p,\ol\beta,\ol\beta_1,\ldots,\ol\beta_{q-1};\alpha}.
\end{equation*}
The operator $\partial_h$ is computed as
\begin{equation}\label{E:partial}
\begin{aligned}
	\paren{\partial_h\chi}_{\alpha_0,\alpha_1,\ldots,\alpha_p,\ol B_q}
	&=
	\sum_{\mu=0}^p (-1)^\mu
	\chi\ind{}{i}{\alpha_0,\ldots,\wh\alpha_\mu,\ldots,\alpha_p,\ol B_q;\alpha_\mu}
	\\
	&=
	\chi\ind{}{i}{\alpha_1,\ldots,\alpha_p,\ol B_q;\alpha_0}
	-
	\sum_{\mu=1}^p
	\chi\ind{}{i}{
	{\tiny\vtop{
	\hbox{$\alpha_1,\ldots,\alpha_0,\ldots,\alpha_p,\ol B_q;\alpha_\mu$}\vskip-.8mm
	\hbox{$\phantom{\alpha_1,\ldots,}{|\atop\mu} $}}}}.
\end{aligned}
\end{equation}
And the formal adjoint $\partial_h^*$ of $\partial_h$ is computed as follows.
\begin{equation*}
		\paren{\partial_h^*\chi}\ind{}{i}{\alpha_1,\ldots,\alpha_{p-1},\ol B_q}
		=
		-g^{\ol\beta\alpha}
		\psi\ind{}{i}{\alpha,\alpha_1\ldots,\alpha_{p-1},\ol B_q;\ol\beta}.
\end{equation*}

We will need the Lefschetz operator $L$, which is defined by
\begin{equation*}
L\chi = \omega\wedge\chi ,
\end{equation*}
and its adjoint operator $\Lambda$ is defined by
\begin{equation*}
\inner{L\chi,\psi}=\inner{\chi,\Lambda\psi}.
\end{equation*}
In local coordinates, for a $(p,q)$-form $\psi$ we have
\begin{equation*}
(\Lambda\psi)\ind{}{i}{\alpha_1,\ldots,\alpha_{p-1},\ol\beta_1,\ldots,\ol\beta_{q-1}}
=
(-1)^p\ii g^{\ol\beta\alpha}\psi\ind{}{i}{\alpha,\alpha_1,\ldots,\alpha_{p-1},\ol\beta,\ol\beta_1,\ldots,\ol\beta_{q-1}}
\end{equation*}
for $\psi\in A^{p,q}(E)$.

The $\dbar$-Laplacian and $\pt_h$-Laplacian on the space of $E$-valued $(p,q)$-forms are defined by
\begin{equation*}
\Box_{\dbar}=\dbar\,\dbar_h^*+{\dbar}_h^*\dbar
\;\;\;\text{and}\;\;\;
\Box_{\pt_h}=\pt_h\pt_h^*+\pt_h^*\pt_h.
\end{equation*}
Then the Bochner-Kodaira-Nakano formula (see, e.g., \cite{Demailly(Book)}) states that
\begin{equation}\label{E:BKN}
\Box_{\dbar}-\Box_{\pt_h}=[\ii\Theta_h(E),\Lambda],
\end{equation}
which is computed in local holomorphic coordinates as
\begin{equation}
	\label{E:local_expression_curvature_operator}
	\begin{aligned}
		\paren{[\ii\Theta_h(E),\Lambda]\psi}\ind{}{i}{A_p,\ol B_q}
		=
		-
		g^{\ol\beta\alpha}
		\Bigg(
			\Theta\ind{j}{i}{\alpha\ol\beta}
			\psi\ind{}{j}{A_p,\ol B_q}
			&-
			\sum_{\mu=1}^p\Theta\ind{j}{i}{\alpha_\mu\ol\beta}
			\psi\lowerindA{j}{\alpha_1}{\alpha}{\alpha_p}{\ol\beta_1,\ldots,\ol\beta_q}{\mu}
			\\
			&-
			\sum_{\nu=1}^q\Theta\ind{j}{i}{\alpha\ol\beta_\nu}
			\psi\lowerindB{j}{\alpha_1,\ldots,\alpha_p}{\ol\beta_1}{\ol\beta}{\ol\beta_q}{\nu}
		\Bigg)
	\end{aligned}
\end{equation}
for a $E$-valued $(p,q)$-form $\psi$.
If $E$ is a line bundle on $X$ and $\ii\Theta_h(E)=\omega$, then we have
\begin{equation*}
\Box_{\dbar}-\Box_{\pt_h}=[\omega,\Lambda]=(p+q-n)\id
\end{equation*}
on the space of $E$-valued $(p,q)$-forms.
\medskip

The following proposition about the commutator of covariant derivatives is well-known.
\begin{proposition}\label{P:commutator_covariant_derivatives}
Let $\psi$ be a $E$-valued $(p,q)$ form on $X$.
Denote by $\nabla_\alpha$ and $\nabla_{\ol\beta}$ covariant derivatives with respect to $\nabla^E$ and $\nabla^X$. Then we have
\begin{align*}
[\nabla_{\ol\beta},\nabla_\alpha]\psi\ind{}{i}{A_p,\ol B_q}
&=
\psi\ind{}{i}{{A_p,\ol B_q};\alpha\ol\beta}
-
\psi\ind{}{i}{{A_p,\ol B_q};\ol\beta\alpha} \\
&=
-\Theta\ind{j}{i}{\alpha\ol\beta}\psi\ind{}{j}{{A_p,\ol B_q}}
+
\sum_{\mu=1}^p
R\ind{\alpha_\mu}{\gamma}{\alpha\ol\beta}
\psi\ind{}{i}{
{\tiny\vtop{
\hbox{$\alpha_1,\ldots,\gamma,\ldots,\alpha_p,\ol B_q$}\vskip-.8mm
\hbox{$\phantom{\alpha_1,\ldots,}{|\atop\mu } $}}}}
+
\sum_{\nu=1}^q
R\ind{\ol\beta_\nu}{\ol\delta}{\alpha\ol\beta}
\psi\ind{}{i}{
{\tiny\vtop{
\hbox{$A_p,\ol\beta_1,\ldots,\ol\delta,\ldots,\ol\beta_q$}\vskip-.8mm
\hbox{$\phantom{A_p,\ol\beta_1,\ldots,}{|\atop\nu } $}}}}
\end{align*}
\end{proposition}
We end this subsection with recalling the definition of the cup product of a differential form with values in the holomorphic tangent bundle and an (vector bundle valued)
differential form now in terms of local coordinates.
\begin{definition}\label{de:cup}
Let
\begin{align}
\mu
& =
\frac{1}{p!q!}
\mu\ind{}{\sigma}{\alpha_1,\ldots,\alpha_p,\ol\beta_1,\ldots, \ol\beta_q}
\pd{}{z^\sigma}\otimes
dz^{\alpha_1}\we\ldots\we dz^{\alpha_p}\we dz^{\ol\beta_1}\we\ldots\we dz^{\ol\beta_q}, \nonumber
\\
\text{and \hspace{2cm}}\nonumber &\\
\nu
& =
\frac{1}{a!b!}
\nu _{\gamma_1,\ldots,\gamma_a,\ol\delta_1,\ldots,\ol\delta_b}
dz^{\gamma_1}\we\ldots\we dz^{\gamma_a}\we dz^{\ol\delta_1}
\we\ldots\we dz^\ol{\delta_b} \, .\nonumber \\
\text{Then}
\hspace{1.8cm}\nonumber &\\
\label{eq:cup}
	\begin{split}\mu\cup\nu
	&:=
	\frac{1}{p!q!(a-1)!b!}
	\mu\ind{}{\sigma}{\alpha_1,\ldots,\alpha_p,\ol\beta_1,\ldots, \ol\beta_q}
	\nu_{\sigma\gamma_2,\ldots,\gamma_a,\ol\delta_1,\ldots,\ol\delta_b}
	dz^{\alpha_1} \we\ldots\we dz^{\alpha_p}
	\\
	&
	\hspace{1cm}
	\we dz^{\ol\beta_1}\we\ldots\we
	dz^{\ol\beta_q}\we dz^{\gamma_2}
	\we\ldots\we dz^{\gamma_a}\we dz^{\ol\delta_1} \we\ldots\we dz^{\ol\delta_b}\, .
	\end{split}
\end{align}
\end{definition}

\subsection{Lie derivatives of $E$-valued smooth forms}
\label{SS:Lie_derivative}
In this subsection, we introduce the Lie derivative of vector bundle valued differentiable forms.

Let $X$ be a complex manifold and $(E,h)$ be a complex vector bundle equipped with an hermitian metric $h$ and a metric compatible connection $D$.
Then for $v\in TM$, the Lie derivative $L_v: A^k(E)\rightarrow A^k(E)$ along $v$  is defined by Cartan's formula. More precisely, for $\psi\in A^k(E)$ we define
\begin{equation*}
	L_v\psi=\paren{
		\delta_v\circ\nabla^E
		+
		\nabla^E\circ\delta_v
	}\psi,
\end{equation*}
where $\delta_v$ is the contraction of a $E$-valued form with a vector $v$.
It is easy to see that this definition coincides with the original Lie derivative on differentiable forms.
We introduce several properties of this Lie derivative.
For the details, see Appendix B in \cite{Naumann2021}.

\begin{itemize}
\item[(1)] For a smooth section $e$ of $E$ and a smooth form $\alpha$,
\begin{equation*}
	L_v(e\otimes\alpha)
	=
	L_v(e)\otimes\alpha
	+
	e\otimes L_v(\alpha).
\end{equation*}
\item[(2)] For smooth sections $e_1, e_2$ of $E$ and smooth forms $\alpha_1, \alpha_2$,
\begin{equation*}
	L_v
	\paren{
		(e_1\otimes\alpha_1)\we_h\ol{(e_2\otimes\alpha_2)}
	}
	=
	L_v(e_1\otimes\alpha_1)\we_h\ol{(e_2\otimes\alpha_2)}
	+
	(e_1\otimes\alpha_1)\we_h(L_v\ol{(e_2\otimes\alpha_2)}),
\end{equation*}
where $(e_1\otimes\alpha_1)\we_h\ol{(e_2\otimes\alpha_2)}:=\alpha_1\we\ol{\alpha_2}\cdot h(e_1,e_2)$.
\end{itemize}
Due to the above properties, one can easily deduce the following corollary (Corollary 5 in Appendix B in \cite{Naumann2021}).
\begin{corollary}
Let $\psi$ be a $E$-valued $k$-form written as
\begin{equation*}
	\psi
	=
	\frac{1}{k!}\psi\ind{}{i}{\alpha_1,\ldots,\alpha_k}
	dx^{\alpha_1}\we\ldots\we dx^{\alpha_k}.
\end{equation*}
for a local coordinate system $(x^1,\ldots,x^n)$.
Let $v=v^\alpha\pt_\alpha$ where $\pt_\alpha=\pt/\pt_{x^\alpha}$.
Then $L_v\psi$ is computed as
\begin{equation}\label{E:Lie_derivative_coordinate_expression}
	\paren{L_v\psi}_{\alpha_1,\ldots,\alpha_k}
	=
	v^\alpha\psi\ind{}{i}{\alpha_1,\ldots,\alpha_k;\alpha}
	+
	\sum_{\mu=1}^k
	v\ind{}{\alpha}{\vert\alpha_\mu}
	\psi\ind{}{i}{
	{\tiny\vtop{
	\hbox{$\alpha_1,\ldots,\alpha,\ldots,\alpha_k$}\vskip-.8mm
	\hbox{$\phantom{\alpha_1,\ldots,}{|\atop\mu} $}}}},
\end{equation}
where the semi-colon $;$ stands for the covariant derivative with respect to $\nabla^E$ and the $\vert$-symbol stands for the ordinary derivative.
\end{corollary}
\begin{remark}\label{R:covariant_derivative}
If $(X,g)$ be a Riemannian manifold, then \eqref{E:Lie_derivative_coordinate_expression} can be considered as
\begin{equation*}\label{E:Lie_derivative_coordinate_expression'}
	L_v\psi
	=
	v^\alpha\psi\ind{}{i}{\alpha_1,\ldots,\alpha_k;\alpha}
	+
	\sum_{\mu=1}^k
	v\ind{}{\alpha}{;\alpha_\mu}
	\psi\ind{}{i}{
	{\tiny\vtop{
	\hbox{$\alpha_1,\ldots,\alpha,\ldots,\alpha_k$}\vskip-.8mm
	\hbox{$\phantom{\alpha_1,\ldots,}{|\atop\mu} $}}}},
\end{equation*}
where $;$ in the first term is the covariant derivative with respect to $D$ and the Levi-Civita connection $\nabla^E$ and $;$ in the second term is the covariant derivative with respect to $\nabla^X$ since $\nabla^X$ is torsion-free.
\end{remark}

Now suppose that $(X,\omega)$ is a Kähler manifold equipped with the Kähler connection and $(E,h)$ is a hermitian vector bundle with the Chern connection $\nabla^E$.
If $v\in T^{1,0}X$ and $\psi\in A^{p,q}(E)$, then $L_v\psi$ is decomposed into two components as follows.
\begin{equation*}
	L_v'\psi=
	\paren{
		\delta_v\circ \pt_h
		+
		\pt_h\circ\delta_v
	}\psi
	\in A^{p,q}(E)
	\;\;\;\text{and}\;\;\;
	L_v''\psi=
	\paren{
		\delta_v\circ\dbar
		+
		\dbar\circ\delta_v
	}\psi
	\in A^{p-1,q+1}(E)
\end{equation*}
Likewise, for $\ol v\in T^{0,1}X$, we have
\begin{equation*}
	L_{\ol v}'\psi=
	\paren{
		\delta_v\circ\dbar
		+
		\dbar\circ\delta_v
	}\psi
	\in A^{p,q}(E)
	\;\;\;\text{and}\;\;\;
	L_{\ol v}''\psi=
	\paren{
		\delta_v\circ\pt_h
		+
		\pt_h\circ\delta_v
	}\psi
	\in A^{p+1,q-1}(E)
\end{equation*}

\subsection{Fiber integrals -- basic properties}\label{sb:fibint} We denote by $\{\cX_s\}_{s\in S}$ a holomorphic family of compact complex manifolds $\cX_s$ of dimension $n>0$ parameterized by a reduced complex space $S$. By definition, it is given
by a proper holomorphic submersion $f:\cX \to S$, such that the $\cX_s$
are the fibers $f^{-1}(s)$ for $s\in S$. In case of a smooth space $S$, if
$\eta$ is a differential form of class $\cinf$ of degree $2n+r$, then the fiber integral
$$
\int_{\cX/S} \eta
$$
is a differential form of degree $r$ on $S$.

We first note that the exterior derivatives $\pt$ and $\ol\pt$ commute with taking fiber integrals (along compact fibers). However we need to take a somewhat different approach.

\begin{lemma}\label{le:intLie}
Let
$$
w_i=\left.\left(\frac{\pt}{\pt s^i} + b\ind{i}{\alpha}{}(z,s)\frac{\pt}{\pt z^\alpha}
+ c\ind{i}{\ol\beta}{}(z,s)\frac{\pt}{\pt z^{\ol\beta}} \right)\right|_{\cX_s}
$$
be differentiable vector fields, whose projections to $S$ equal
$\frac{\pt}{\pt s^i}$ for all points of the fiber $\cX_s$. Then
\begin{equation}\label{E:derint}
\frac{\pt}{\pt s^i} \int_{\cX_s} \eta =
\int_{\cX_s} L_{w_i}(\eta),
\end{equation}
where $L_{w_i}$ denotes the Lie derivative along $w_i$.
\end{lemma}

Concerning singular base spaces, it is obviously sufficient that the above
equation is given on the first infinitesimal neighborhood of $s$ in $S$.

\begin{proof}[Proof of {Lemma~\ref{le:intLie}}]
Because of linearity, one may consider the real and imaginary parts of
$\pt/\pt s^i$ and $w_i$ resp.\ separately.

Let $\pt/\pt t$ denote $\re(\pt/\pt s^i)$ for some $i$, and let $\Phi_t: X \to \cX_t$
be the one parameter family of diffeomorphisms generated by $\re(w_i)$.
Then
\begin{equation}\label{E:diffint}
L_{\re(w^i)}\psi = \frac{d}{dt} \Phi^*_t \psi
\quad\text{and}\quad
\frac{d}{d t} \int_{\cX_s} \eta = \int_X \frac{d}{dt} \Phi^*_t \eta = \int_X L_{\re(w^i)}(\eta).
\end{equation}
The same argument is applied to the imaginary part $\im(w^i)$, and finally
$$
L_{w^i}:= L_{\re(w^i)} +\ii L_{\im(w^i)}
$$
so that the claim follows.

In general the vector fields $\re(w_i)$ and $\mathrm{Im}(w_i)$ need not commute so that a simultaneous trivialisation for these does not exist, in particular a trivialisation for the vector field $v$ (see the discussion in \cite{Siu1986}).
\end{proof}

The notion of the Lie derivative of {\em relative} differential forms with respect to differentiable lifts $v$ of vector fields on the total space $\cX$ will be needed. Again these are defined in the same way by \eqref{E:diffint} following the above construction. It follows easily that the Leibniz rule holds for these and again yields differential forms and tensors resp. If differential forms with values in a hermitian bundle are being considered, then covariant differentiation is to be applied.

\begin{lemma}
Let $\eta$ be a differential form of a certain type $(p,q)$ on the total space, and $\eta_{\cX/S}$ the induced relative form. Then
$$
	L_v (\eta)|_{\cX/S} = L_v(\eta_{\cX/S}),
$$
in particular if $\eta$ is of type $(n,n)$, then
\begin{equation}\label{E:relLie}
\pt_s\int_{\cX/S} \eta = \int_{\cX/S} L_v(\eta_{\cX/S}).
\end{equation}
Conversely, any relative form is the restriction of a form on the total space to fibers so that \eqref{E:relLie} is applicable.
\end{lemma}
The {\em proof} follows from the specific form of any differentiable lift of a tangent vector $\pt_s$. The first statement is not valid for arbitrary vector fields $v$. The extension of a relative differential form follows by means of a partition of unity. The choice does not affect the result of the fiber integration.

Now by the above Lemma it is possible to restrict calculations of Lie derivatives of differential forms on fibers that depend on parameters.

In our applications, the form $\eta$ will typically involve both inner
products of differential forms with values in hermitian vector bundles,
whose factors need to be treated separately.

As mentioned above in this way the Lie derivative can be reduced to sections of bundles  $\cE$, $\cT_{\cX/S}$, $\Omega_{\cX/S}$ and related bundles.

Let $\omega$ be a closed $(1,1)$-form, whose restrictions to fibers $\omega_s=\omega|\cX_s$ are \ka\ forms, and let $g\, dV$ be the induced family of volume forms on fibers or ''relative \ka\ form''. Then the above lemma is applied
to forms of the type
$$
\eta = k(z,s) \, g\, dV,
$$
where $k$ is any differentiable function.

Now let $v$ be a lift such that $L_v\omega=0$. This property will be shown for horizontal lifts below (cf.\ Lemma~\ref{L:gpar}). Then $L_v\omega|\cX_s$=0 so that
$$
\pt_s\int_{\cX/S}  k \, g\, dV = \pt_s\int_{\cX/S} L_v(k) g\, dV.
$$
In Section~\ref{se:compcurv} this fact will be applied to pointwise inner products $k=(\psi,\chi)$ of $\cE$-valued relative forms of certain degrees $(p,q)$ and related quantities.

\begin{remark}
	When computing Lie derivatives of relative $(p,q)$-forms the process of \eqref{E:diffint} will be applied to the underlying $k$-forms, $k=p+q$. In general, the process is not type-preserving.
\end{remark}

\subsection{Higher direct images and relative differential forms}
\label{SS:higher_direct_images}

Let $(\cE, h)$ be a hermitian, holomorphic vector bundle on $\cX$ -- in order to compute the curvature, we assume that the direct image $R^q f_*\cE$ is {\em locally free}. Unless stated differently we will assume that the resp.\ base space $S$ is reduced. As in particular $\cE$ is $S$-flat, the Grauert-Grothendieck comparison theorem is applicable. For reduced base spaces the direct image is locally free, if $H^q(\cX_s, \cE_s)$ is of constant dimension, where $\cE_s=   \cE \otimes_{\cO_{\cX}}\cO_{\cX_s}|\cX_s$. In this case the cohomology groups can be identified with $R^q f_*\cE\otimes_{\cO_S} \C(s)$. For not necessarily reduced base spaces $S$ the latter properties hold, if for all $s\in S$ the cohomology $H^{q+1}(\cX_s, \cE \otimes \cO_{\cX_s})$  vanishes.

For a Stein open subset $V\subset S$ the space of sections $R^qf_*\cE(V)$ can be identified with the cohomology group $H^q\paren{f^{-1}(V),\cE\vert_{f^{-1}(V)}}$. Given such a cohomology class,
restricted to fibers, we have harmonic representatives of cohomology
classes with respect to the \ka\ forms and hermitian metrics on the fibers.
The following fact from \cite{Schumacher2012} will be essential. Under the above assumptions we have:

\begin{lemma}\label{L:representative}
Let $u \in R^q f_*\cE(S)$ be a section. Let $\psi_s \in
\cA^{0,q}(\cX_s,\cE_s)$ be the harmonic representatives of the cohomology
classes $u\vert_{\cX_s}\in H^q(\cX_s,\cE_s)$.

Then locally with respect to $S$ there exists a $\dbar_\cX$-closed form
$\psi\in \cA^{0,q}(\cX,\cE)$, which represents $u$, and whose
restrictions to the fibers $\cX_s$ equal $\psi_s$.
\end{lemma}

%

\begin{proof}
For the sake of completeness we give the simple argument. The harmonic
representatives $H(\psi_s)$ define a relative differential form, which we denote by
$\psi_{\cX/S}$. Let $\Phi \in \cA^{0,q}(\cX,\cE)$ represent $u$.
Denote by $\Phi_{\cX/S}$ the induced relative form. Then there exists a
relative $(0,q-1)$-form $\chi_{\cX/S}$ on $\cX$, whose exterior derivative
in fiber direction $\ol\pt_{\cX/S}(\chi_{\cX/S})$ satisfies
$$
\psi_{\cX/S}= \Phi_{\cX/S} +\ol\pt_{\cX/S}(\chi_{\cX/S}).
$$
Let $\{U_i\}$ be a covering of $\cX$, which possesses a partition of unity
$\{\rho_i\}$ such that all $\chi_{\cX/S}|U_i$ can be extended to
$(0,q-1)$-forms $\chi_i$ on $U_i$. Then with $\chi=\sum\rho_i\chi_i$ we
set $\psi =\Phi + \ol\pt \chi$.
\end{proof}

We are concerned with higher direct images $R^qf_*\Omega^p_{\cX/S}(\cE)$.
As we mentioned above, a local holomorphic section is of $R^qf_*\Omega^p_{\cX/S}(\cE)$ is a $\dbar$-closed $(0,q)$-form of values in $\Omega^p_{\cX/S}(\cE)$.
Thanks to Lemma~\ref{L:representative}, locally with respect to $S$ there exists a $\dbar_\cX$-closed $\psi\in\cA^{0,q}(\cX,\Omega^p_{\cX/S}(\cE))$ such that $\psi\vert_{\cX_s}$ is harmonic on each fiber $\cX_s$.

We emphasize here that in general the sheaf $\Omega^p_{\cX/S}$ cannot be replaced by $\Omega^p_{\cX}$ unless we consider (relative) {\em canonical} bundles. We have
$$
\Omega^{n+k}_{\cX}\simeq \Omega^{n}_{\cX/S} \otimes_{\cO_\cX} f^*(\Omega^{k}_S),\quad  i.e. \quad \cK_\cX\simeq \cK_{\cK/S}\otimes_{\cO_\cX} f^*\cK_S.
$$

Another case is $\dim S=1$. Let $s$ be a local coordinate. Then, again locally with respect to $S$, for $p>0$ there exists a sequence
$$
0 \to \Omega^{p}_{\cX/S} \to \Omega^{p+1}_\cX \to \Omega^{p+1}_{\cX/S} \to 0
$$
where the injection is defined by the wedge product with $ds$.
\medskip

In this sense the fact that taking fiber integrals of (global)  differential forms commutes with applying exterior derivatives cannot always be used, and taking Lie derivatives is necessary.

\section{Deformation of pairs $(Y,F)$}\label{S:pairs}
As is generally known the existence theorem for a semi-universal deformation of objects like holomorphic vector bundles over compact manifolds (and pairs of such) can be reduced to the existence of a semi-universal deformation of compact complex analytic spaces that need not be reduced, including the theory of infinitesimal deformations. Explicitly, infinitesimal deformations and obstructions for pairs consisting of a manifold and a holomorphic vector bundle had been studied by Griffiths in \cite{Griffiths1969}. We summarize some of the facts, and apply these to our situation. For reasons of notation, in this section we will write $(Y,F)$ for such a pair.

The Atiyah sequence of a vector bundle $F$ on $Y$ is a short exact sequence of holomorphic vector bundles. It consists of the spaces of infinitesimal automorphism of $Y$ and $F$, and the Atiyah bundle $\Sigma_Y(F)$ is the vector bundle of infinitesimal automorphisms of the pair $(Y,F)$:
\begin{equation}\label{eq:atiyah}
\xymatrix{
  0 \ar[r] & End(F)  \ar[r]_i & \Sigma_Y(F) \ar@/_1pc/[l]_{\sigma_Y}  \ar[r]_\rho & \cT_Y \ar[r] \ar@/_1pc/[l]_{\tau_Y} & 0 \, .
}
\end{equation}

\begin{proposition}\label{pr:ks}
	The groups $H^j(Y,\Sigma_Y(F))$ can be identified with the groups of isomorphism classes of infinitesimal automorphisms and deformations resp.\ of the pair $(Y,F)$ for $j=0,1$ resp.
\end{proposition}
We will characterize infinitesimal automorphisms of the pair $(Y,F)$ explicitly. First automorphisms near the identity are characterized (with values in given open sets, but defined on certain smaller domains). Such an automorphism over $V$  consist of a pair $(\Phi,\Psi)$, where $\Phi$ is an automorphism of $Y$, and $\Psi : F\to \Phi^*F$ an automorphism of holomorphic vector bundles, with an obvious notion for $Y$ being replaced by an open subset.

Let $\{U_j\}$ be an open covering of $Y$ together with holomorphic trivializations of  $F$ with values in $U_j\times \C^r$ and transition functions $\gamma_{jk}$. Let $\Phi\in Aut(Y)$ be close to the identity and $\Psi$ as above. In terms of the trivializations it is defined by a set of automorphisms $\Psi_j$ of the restriction of $F$ to $U_j$. The compatibility of locally given automorphisms $(\Phi,\Psi_j)$ with the transition functions means
$$
(\Phi, \gamma_{ij}\circ \Psi_j )= (\Phi, \Psi_i \circ \gamma_{ij}).
$$
Now we set formally $\Phi= id + \epsilon\phi$ and $\Psi_i= id + \epsilon \psi_i$ with $\epsilon^2=0$ and $\psi_j$ and $\phi$ contained in the respective Lie algebras. Then
\begin{equation}\label{eq:atiyahtrans}
  \psi_j= \gamma^{-1}_{ij} \psi_i \gamma_{ij} + \gamma^{-1}_{ij} d\gamma_{ij} \cup \phi
\end{equation}
where the cup product stands for the application of a differential form to a vector field.

One observes that given two sets $e_j$, and $e_j=e_i\gamma_{ij}$ of holomorphic frames the transition law for a connection form $\theta$ is analogous. We have
\begin{gather*}
  e_j \theta_j =D(e_j)= D(e_i\gamma_{ij}) = D(e_i) \gamma_{ij} + e_i d\gamma_{ij} = e_i \theta_i\gamma_{ij} + e_i d\gamma_{ij} = e_j\gamma^{-1}_{ij} \theta_i \gamma_{ij} + e_j\gamma^{-1}_{ij} d\gamma_{ij}
\end{gather*}
i.e.
\begin{equation}\label{eq:conn}
  \theta_j = \gamma^{-1}_{ij} \theta_i \gamma_{ij} + \gamma^{-1}_{ij} d\gamma_{ij}
\end{equation}
This equation applies to flat, holomorphic and differentiable connections.
\begin{lemma}\label{le:atiyah}
  A differentiable splitting $\tau_Y$ of \eqref{eq:atiyah} is defined by applying the connection to a differentiable vector field. In particular, for
  $$
  w = w^\alpha(\pt/\pt y^\alpha)
  $$
  the section $\tau_Y$ is given by
  $$
  \tau_Y(w)=(w^\alpha \theta_{F\alpha} ,w),
  $$
  where $\theta_F$ is the connection of the hermitian bundle $F$ on $Y$.

	The section $\sigma_Y:\Sigma_Y(F) \to End(F)$ of $i$ is given by
	$$
	id_{\Sigma_Y(F)} - \tau_Y\circ\rho,
	$$ which has values in $i(End(F))$.

	Altogether there is a diffeomorphism (which in general is not compatible with taking exterior derivatives of differential forms)
	$$
	\Sigma_Y(F)\simeq End(F) \oplus \cT_Y.
	$$
\end{lemma}
Note that the above section $\tau_Y$ is subject to the transition equation \eqref{eq:conn}.

The purpose of the lemma is to allow for a natural $L^2$-structure on $\Sigma_Y(F)$.
The {\em proof} follows from \eqref{eq:atiyahtrans} together with \eqref{eq:conn}. \qed

\medskip

Similar considerations hold for the relative Atiyah bundle (see below).

An immediate consequence is the following:
\begin{lemma}\label{le:appatiyah}
  \begin{itemize}
	\item[(i)] The forms $(\gamma^{-1}_{ij} d\gamma_{ij})$ on $\{U_{ij}\}$ define a $1$-cocycle, whose class, the Atiyah class $A(F)$ in $H^1(Y, \Omega^1_Y\otimes End(F))=Ext^1_{\cO_Y}(Y;\cT_Y, End(F))$ represents the equivalence class of the extension \eqref{eq:atiyah}.
	\item[(ii)]  Holomorphic, and  differentiable sections of \eqref{eq:atiyah} correspond to holomorphic, and  differentiable connections of $F$ resp.
	\item[(iii)] In terms of Dolbeault cohomology the Atiyah class is represented by the curvature form $\Theta_{\alpha\ol\beta}(F)dz^\alpha\we dz^{\ol\beta}$ of any hermitian metric on $F$.
	\item[(iv)] The connecting morphisms $\delta^j$ of the long exact cohomology sequence
\begin{gather*}
  0 \to H^0(Y,End(F)) \to H^0(Y, \Sigma_Y(F))\to H^0(Y,\cT_Y) \stackrel{\delta^0}{\longrightarrow} \hspace{5cm} \strut \\
  \to H^1(Y,End(F)) \to H^1(Y, \Sigma_Y(F))\to H^1(Y,\cT_Y) \stackrel{\delta^1}{\longrightarrow} H^2(Y,End(F)) \to \ldots
\end{gather*}
are given by the contraction with $\Theta(F)$:  For a holomorphic vector field $C=C^\alpha\pt_\alpha$ the image $\delta^0(C)$ is represented by $C^\alpha\Theta_{\alpha\ol\beta}dz^{\ol\beta}$ and for a \ks\ form $A=A\ind{}{\alpha}{\ol\beta}\pt_\alpha dz^{\ol\beta}$ the image equals $\delta^1(A)= A\ind{}{\alpha}{\ol\beta}\Theta_{\alpha\ol\delta} dz^{\ol\beta}\we dz^{\ol\delta}$. \par
	\item[(v)] The cohomology class of the image $\delta^0(C)= C \cup \Theta$ is the obstruction against extending $C$ as an infinitesimal automorphism of $(Y,F)$, and the class of $\delta^1(A)= A \cup \Theta$ is the obstruction against extending $A$ as an infinitesimal deformation of the pair $(Y,F)$.
  \end{itemize}
\end{lemma}
\begin{proof}
  {(i):} Equation \eqref{eq:atiyahtrans} implies the closedness and the property claimed. (ii): The claim follows from \eqref{eq:atiyahtrans} together with \eqref{eq:conn}. Item (iii) follows from (ii), and (iv) follows from the preceding claims, namely: Let $C$ and $A$ be given as above. Then the connecting homomorphisms are defined by applying the splitting $\tau_Y$ first. The results $C^\alpha \theta_\alpha$ and $A^\alpha_{\phantom{\alpha}\ol\beta} \theta_\alpha dz^{\ol\beta}$ resp.\ have to be interpreted as $\Sigma_Y(F)$-valued with transition functions induced by \eqref{eq:conn}. The respective images under the exterior derivative $\ol\pt$ are ordinary differential sections and forms resp.\ and can be interpreted as $End(F)$ valued. The results are those claimed in (iv).
\end{proof}

Now we return to the situation, where $\cE \to \cX \stackrel{f}{\longrightarrow} S $ is a holomorphic family of holomorphic vector bundles and compact complex manifolds equipped with a family $h$ of hermitian metrics and a \ka\ form $\omega_\cX$ resp.\ with induced relative \ka\ form $\omega_{\cX/S}$. A distinguished fiber $\cX_{s_0}$, $s_0\in S$ will also simply be denoted by $X$, and $E:= \cE\vert_X$.

Our aim is to characterize the \ks\ mapping
\begin{equation}\label{eq:kssig}
\rho : T_{S,s_0} \to H^1(X,\Sigma_X(E))
\end{equation}
in terms of certain $(0,1)$-forms, which are related to the {\em differentiable splitting} of the relative Atiyah sequence of rather the Atiyah sequence on the distinguished fiber $X$.

Consider the following diagram:
\begin{equation*}
\xymatrix{
		 &  &  0 \ar[d] & \ar[d] 0    &
		 \\
		0 \ar[r] & End(\cE)\ar@{=}[d] \ar[r]
		& \Sigma_{\cX/S}(\cE) \ar[r]\ar[d]
		& \cT_{\cX/S}  \ar[r]\ar[d] & 0
		\\
		0 \ar[r] & End(\cE)  \ar[r]_i
		& \Sigma_\cX(\cE) \ar@/_1pc/[l]_\sigma \ar[d]  \ar[r]_\rho
		& \cT_\cX \ar[r] \ar@/_1pc/[l]_\tau \ar[d] & 0
		\\
		& & f^*\cT_S \ar@{=}[r]\ar[d] &f^*\cT_S\ar[d] &
		\\
		& & 0&0 &
}
\end{equation*}
(Here the relative space $\Sigma_{\cX/S}$ is defined as the kernel of the natural map $\Sigma_\cX \to f^*T_S$.)
Let again (by abuse of notation) $s$ be a local holomorphic coordinate on $S$ and $\pt_s$ a  coordinate vector field near the origin.

We consider the {\em last column}. The last column is being used to describe the \ks\ map in terms of Dolbeault cohomology by means of differentiable  lifts of tangent vectors from $S$ to $\cX$. As above $X$ denotes the distinguished fiber of $\cX\to S$.

\begin{lemma}\label{le:athorli}
  The \ks\ map \eqref{eq:kssig} for deformations of the pair $(X,E)$ in terms of the \ka\ and hermitian metrics resp.\ for the given family can be characterized as follows:

  Let $\pt_s$ be a tangent vector of the base $S$ at the distinguished point, and $v=\pt_s + a\ind{s}{\alpha}{} \pt_\alpha$ its horizontal lift. Then $\tau(v)$ is a differentiable lift of the tangent vector to the Atiyah bundle, and
  $$
  \ol\pt(\tau(v))\vert_X
  $$
  is a $\ol\pt$-closed $\Sigma_X(E)$-valued form that represents the \ks\ class $\rho(\pt/\pt s)$.
\end{lemma}

In a more explicit way the result is given as follows:

\begin{lemma}\label{le:atiyah2}
  A differentiable splitting $\tau$ of \eqref{eq:atiyah} is defined by applying the connection to a differentiable vector field, in particular
  $$
  \tau(v) = \tau(\pt_s + a\ind{s}{\alpha}{}\pt_a):= (\theta_s + a\ind{s}{\alpha}{}\theta_\alpha, v)
  $$
  with $\rho(\tau (v))=v$. In particular
  \begin{eqnarray*}
  \ol\pt(\tau(v))|X & =& (\ol\pt(\theta\cup v) , \ol\pt v )|X\\
	& = &    (\ol\pt(\theta_s + a\ind{s}{\alpha}{} \theta_\alpha)|X, A_s)\\
	 &=& ((\Theta_{s \ol\beta} + a\ind{s}{\alpha}{} \Theta_{\alpha\ol\beta} + A\ind{s}{\alpha}{\ol\beta} \theta_\alpha)dz^{\ol\beta}  , A\ind{s}{\alpha}{\ol\beta}\pt_\alpha dz^{\ol\beta}).
  \end{eqnarray*}

  Note that the transition equations are induced by those of the connection forms in the sense of \eqref{eq:atiyahtrans}. In general the first components $v\cup \theta$ as well as its $\ol\pt$-derivatives are no $End(E)$-valued forms.
  \end{lemma}

The above description of $\ol\pt(\tau(v))\vert_X$ does not facilitate the notion of an $L^2$-norm as it stands, unless the differentiable decomposition of the Atiyah bundle from Lemma~\ref{le:atiyah} is applied. This is done at this point.

\begin{proposition}\label{P:atiyah}
  With respect to the differentiable trivialization of the Atiyah bundle the \ks\ class is represented by distinguished forms
  \begin{eqnarray*}
	(-\eta_s,A_s) &:=& (v\cup \Theta, A_s)\\
	&=&(\Theta_{s\ol\beta} + a\ind{s}{\alpha}{}\Theta_{\alpha\ol\beta}, A_s)\\
	 &\in& \cA^{0,1}(X,End(E))\oplus \cA^{0,1}(X,\cT_X).
  \end{eqnarray*}
\end{proposition}

The latter spaces carry natural $L^2$-inner products. However these are introduced at the expense that in general $\eta_s$ is no longer $\ol\pt$-closed.

\begin{proof}[Proof of the proposition]
  The decomposition from Lemma~\ref{le:atiyah} is applied to the form  $\ol\pt(\tau(v))\vert_X$  from Lemma~\ref{le:athorli}.

  Note first that
  $$
  \rho(\ol\pt \tau(v)) = A_s = A\ind{s}{\alpha}{\ol\beta} \pt_\alpha dz^{\ol\beta}.
  $$
Now
\begin{eqnarray*}
  \Big((1 - \tau\rho)(\ol\pt \tau(v)), \rho(\tau(\ol\pt \tau( v )))\Big) &=& \Big(\ol\pt(v\cup \theta) - \tau(\rho(\ol\pt \tau(v) )), \rho(\ol\pt(\tau(v)) )\Big) \\
  &=&(\ol\pt( v\cup\theta) - \tau(A_s)   , A_s   )\\
  &=& ( \ol\pt(v) \cup \theta + v\cup \ol\pt \theta  - A_s\cup \theta  )     , A_s )\\
  &=& ( v\cup \Theta, A_s  ) \\
  &\in& \cA^{0,1}(X,End(E))\oplus \cA^{0,1}(X,T_X).
\end{eqnarray*}
\end{proof}
In principle, Lemma~\ref{le:appatiyah} together with Lemma~\ref{le:atiyah} can be used to determine the subspace of the above direct sum that actually represents infinitesimal deformations of $(X,E)$.

\begin{remark}
  In the given situation it is known that $\delta^1(A_s)=A_s\cup \Theta$ is a $\ol\pt$-coboundary from Lemma~\ref{le:appatiyah}(v), in fact it is equal to $\ol\pt(\eta_s)$. However, later the stronger condition $\ol\pt(\eta_s)=A_s\cup \Theta=0$ will be of interest. The form $\eta_s$ determines a class in $H^1(X,End(\cE))$, and the geometric meaning of the condition is that the induced infinitesimal deformation of the pair $(X,E)$ is isomorphic to a deformation that splits into a deformation of $X$ and a deformation of $E$.
\end{remark}

The $End(\cE)$-valued $(0,1)$-form $\eta_s=-v\cup \Theta$ on $X$ will be crucial, when computing curvatures of direct images.
We finish this section with the following two propositions.

\begin{proposition}\label{P:d-closed}
If $A_s\cup\Theta=0$, then $\eta_s$ is $\dbar$-closed.
\end{proposition}

\begin{proof}
Recall that
\begin{align*}
		-\eta_s=v\cup\Theta\vert_X
		&=
		\paren{
				\Theta_{s \ol\beta}
				+
				a\ind{s}{\alpha}{}\Theta_{\alpha\ol\beta}
	}dz^{\ol\beta}
\end{align*}
A direct computation gives that
\begin{align*}
	-\paren{\dbar\eta_s}_{\ol\beta\ol\delta}
	&=
	\paren{
				\Theta_{s \ol\beta}
				+
				a\ind{s}{\alpha}{}\Theta_{\alpha\ol\beta}
	}_{;\ol\delta}
	-
	\paren{
				\Theta_{s \ol\delta}
				+
				a\ind{s}{\alpha}{}\Theta_{\alpha\ol\delta}
	}_{;\ol\beta}
	\\
	&=
	a\ind{s}{\alpha}{;\ol\delta}\Theta_{\alpha\ol\beta}
	-
	a\ind{s}{\alpha}{;\ol\beta}\Theta_{\alpha\ol\delta}
	=
	A\ind{s}{\alpha}{\delta}\Theta_{\alpha\ol\beta}
	-
	A\ind{s}{\alpha}{\beta}\Theta_{\alpha\ol\delta}
	=
	(A_s\cup\Theta)_{\ol\beta\ol\delta}=0.
\end{align*}
This completes the proof.
\end{proof}

\begin{proposition}\label{P:dbar*-closed}
If $h\vert_E$ is Hermite-Einstein on the distinguished fiber $X$, then $\eta_s$ is $\dbar^*$-closed on $X$.
\end{proposition}

\begin{proof}
The connection form $\theta$ of the hermitian metric $h$ of $\cE$  satisfies
\begin{equation*}
\theta_{s}=h_{|s}\cdot h^{-1}
\;\;\;\text{and}\;\;\;
\theta_\alpha=h_{\vert\alpha}\cdot h^{-1}.
\end{equation*}
It follows that
\begin{align*}
	\theta_{s\vert\alpha}-\theta_{s\vert\alpha}
	&=
	\paren{h_{\vert s}\cdot h^{-1}}_{\vert\alpha}
	-
	\paren{h_{\vert\alpha}\cdot h^{-1}}_{\vert s}
	\\
	&=
	h_{\vert s\alpha}\cdot h^{-1}
	-
	h_{\vert s}h^{-1}h_{\vert\alpha} h^{-1}h
	-
	h_{\vert\alpha s}\cdot h^{-1}
	+
	h_{\vert\alpha} h^{-1}h_{\vert s} h^{-1}h
	\\
	&=
	[\theta_\alpha,\theta_s].
\end{align*}
Now $\dbar^*\eta_s$ is computed as follows.
\begin{align*}
	-\paren{\dbar^*\eta_s}
	&=
	g^{\ol\beta\gamma}
	\paren{
				\Theta_{s \ol\beta}
				+
				a\ind{s}{\alpha}{}\Theta_{\alpha\ol\beta}
	}_{;\gamma}
	\\
	&=
	g^{\ol\beta\gamma}
	\paren{
				\Theta_{s\ol\beta;\gamma}
				+
				a\ind{s}{\alpha}{;\gamma}\Theta_{\alpha\ol\beta}
				+
				a\ind{s}{\alpha}{}\Theta_{\alpha\ol\beta;\gamma}
	}
	\\
	&=
	g^{\ol\beta\gamma}
	\paren{
		\Theta_{s\ol\beta\vert\gamma}
		-
		[\theta_\gamma,\Theta_{s\ol\beta}]
		+
		a\ind{s}{\alpha}{;\gamma}\Theta_{\alpha\ol\beta}
		+
		a\ind{s}{\alpha}{}\Theta_{\alpha\ol\beta;\gamma}
	}
	\\
	&=
	g^{\ol\beta\gamma}
	\paren{
		\paren{\theta_{s\vert\ol\beta}}_{\vert\gamma}
		-
		[\theta_\gamma,\Theta_{s\ol\beta}]
		+
		a\ind{s}{\alpha}{;\gamma}\Theta_{\alpha\ol\beta}
		+
		a\ind{s}{\alpha}{}\Theta_{\alpha\ol\beta;\gamma}
	}
	\\
	&=
	g^{\ol\beta\gamma}
	\paren{
		\theta_{\gamma\vert\ol\beta s}
		+
		[\theta_\gamma,\theta_s]_{\vert\ol\beta}
		-
		[\theta_\gamma,\Theta_{s\ol\beta}]
		+
		a\ind{s}{\alpha}{;\gamma}\Theta_{\alpha\ol\beta}
		+
		a\ind{s}{\alpha}{}\Theta_{\alpha\ol\beta;\gamma}
	}
	\\
	&=
	\paren{g^{\ol\beta\gamma}\Theta_{\gamma\ol\beta}}_{\vert s}
	-
	g\ind{}{\ol\beta\gamma}{\vert s}\Theta_{\gamma\ol\beta}
	+
	g^{\ol\beta\gamma}
	[\Theta_{\gamma\ol\beta},\theta_s]
	+
	g^{\ol\beta\gamma}a\ind{s}{\alpha}{;\gamma}\Theta_{\alpha\ol\beta}
	+
	a\ind{s}{\alpha}{}
	\paren{g^{\ol\beta\gamma}\Theta_{\alpha\ol\beta}}_{;\gamma}
	\\
	&=
	g^{\ol\beta\sigma}
	g_{\sigma\ol\tau\vert s}
	g^{\ol\tau\gamma}
	\Theta_{\gamma\ol\beta}
	+
	a\ind{s}{\alpha}{;\gamma}\Theta\ind{\alpha}{\gamma}{}
	=
	-
		a\ind{s}{\gamma}{;\sigma}
	\Theta\ind{\gamma}{\sigma}{}
	+
	a\ind{s}{\alpha}{;\gamma}\Theta\ind{\alpha}{\gamma}{}=0
\end{align*}
This completes the proof.
\end{proof}

Hence we have
\begin{theorem}
If $A_s\cup\Theta=0$ and $h\vert_E$ is Hermite-Einstein on the distinguished fiber $X$, then $\eta_s$ is harmonic.
Furthermore, if $\omega\vert_X$ is K\"ahler-Einstein, then $(-\eta_s,A_s)=(v\cup \Theta, A_s)$ is the harmonic representative of the Kodaira-Spencer class of $(X,E)$.
\end{theorem}

\begin{proof}
By \cite{Schumacher2012}, $A_s$ is harmonic if $\omega\vert_X$ is K\"ahler-Einstein.
Then the conclusion follows from Proposition \ref{P:dbar*-closed} and Proposition \ref{P:d-closed}.
\end{proof}

\section{Computation of the curvature}\label{se:compcurv}
The metric tensors for $R^qf_*\Omega^p_{\cX/S}(\cE)$ on the base space
$S$ will be given by integrals of inner products of distinguished representatives of cohomology classes. We know from Lemma 2 in \cite{Schumacher2012} that these are the restrictions of certain $\ol\pt$-closed differential forms on the total
space. When we compute derivatives with respect to the base of these fiber integrals, we will apply Lie derivatives with respect to lifts of tangent vectors in the sense of Section~\ref{sb:fibint}. Taking \emph{horizontal lifts} simplifies the computations. The Lie derivatives of these pointwise inner products can be broken up into sums of products of tensors, and Lie derivatives of these differential forms with values in the given hermitian vector bundle have to be taken. Furthermore the Lie derivatives restricted to fibers of $\omega_\cX$ and $\omega_{\cX/S}$ with respect to horizontal lifts of tangent vectors on the base vanish, in particular such Lie derivatives of relative volume forms.



\subsection{Setup}\label{se:setup}
As discussed in Section~\ref{S:pairs}, we denote by $f:\cX \to S$ a holomorphic family of compact Kähler manifolds, which is a proper holomorphic submersion from $\cX$ to a complex manifold $S$.
We assume that $\cX$ is equipped with a fiberwise Kähler form $\omega$, which is a $d$-closed real $(1,1)$-form on $\cX$ whose restriction is always positive-definite on every fiber $\cX_s=f^{-1}(s)$ for $s\in S$.
We denote a local holomorphic coordinate in $S$ by $s=(s^1,\ldots,s^m)$.
Then one can take a local holomorphic coordinate $z=(z^1,\ldots,z^n)$ on a fixed fiber such that
\begin{itemize}
\item $(z^1,\dots,z^n,s^1,\ldots,s^m)$ forms a local coordinate of $\cX$,
\item $f(z^1,\dots,z^n,s^1,\ldots,s^m)=(s^1,\ldots,s^m)$ in the coordinate $(z,s)$.
\end{itemize}
We call this an \emph{admissible coordinate system of $f$}.
In the rest of this paper, admissible coordinates are always assumed.
Under an admissible coordinate, the fiber coordinates will always be denoted by $z^\alpha$ for $\alpha=1,\ldots,n$ and the base coordinates by $s^i$ for $i=1,\ldots,m$.
We also set $\pt_i=\pt/\pt s^i$, $\pt_\alpha=\pt/\pt z^\alpha$.
Throughout this paper, small Greek letters $\alpha,\beta,\dots=1,\dots,n$ stand for indices on $z=(z^1,\dots,z^n)$ unless otherwise specified.
If there is no confusion, we always use the Einstein summation convention.

For any differentiable vector field $V$ of type $(1,0)$ on $S$ its {\em horizontal lift} $v=v_\omega$ with respect to $\omega$ is by definition a differentiable vector field on $\cX$ of the same type satisfying
\begin{itemize}
\item [(\romannumeral1)]$df(v)=V$, and
\item [(\romannumeral2)]$\omega(v,\ol w)=0$ for all $w\in{T^{1,0}\cX_s}$.
\end{itemize}
If we denote by $\rho:T_{S,s}\rightarrow H^1(\cX_s,T_{\cX_s})$ the Kodaira-Spencer map at a given point $s\in S$, then it is well-known that
\begin{equation*}
	\rho(V_s)
	=
	\left[\paren{\dbar v_\omega}\vert_{\cX_s}\right].
\end{equation*}
Namely, $\paren{\dbar v_\omega}\vert_{\cX_s}$ is the corresponding distinguished representative of the Kodaira-Spencer class $\rho(V_s)$.
In terms of admissible coordinates $(z,s)$, the horizontal lift $v_i$ of $\pt_i$ with respect to $\omega$ is given by
\begin{equation*}
v_i= \pt_i + a\ind{i}{\alpha}{}  \pt_\alpha,
\;\;\;
a\ind{i}{\alpha}{} = - g_{i \ol\beta}g^{\ol\beta\alpha}
\end{equation*}
where $(g^{\ol\beta\alpha})= (g_{\alpha\ol\beta})^{-1}$.
And the corresponding distinguished representative of the \ks\ class $\rho(\pt_i|_{s_0})$ is equal to
\begin{equation*}
	\dbar_{\cX}(v_i)\vert_{\cX_s}
	=
	A_i
	=
	A\ind{i}{\alpha}{\ol\beta}\pt_\alpha\otimes dz^{\ol \beta}\, .
\end{equation*}
The equations
$$
A_{i\ol\beta\ol\delta}=-g_{i\ol\beta;\ol\delta}=-g_{i\ol\delta;\ol\beta}= A_{i\ol\delta\ol\beta}
$$
and
$$
A_{i\ol\beta\ol\delta;\ol\tau}=A_{i\ol\beta\ol\tau;\ol\delta}
$$
follow immediately from the definition of the covariant derivative and the $\dbar$-closedness of $A_i$.
For the relative \ke\ case we also have harmonicity, which reads
$$
0=\ol\pt^*\!A_i=-g^{\ol\beta\gamma} A^{\;\;\alpha}_{i\;\;\ol\beta;\gamma}\pt_\alpha
$$
(\cite[Prop.~2]{Schumacher2012}).
For later use, we denote by
\begin{equation*}
	c(\omega)_{j\ol k}
	=
	\omega(v_j,\ol{v_k}).
\end{equation*}

Let a section of $R^qf_*\Omega^p_{\cX/S}(\cE)$ be given. For Stein spaces $S$, in particular locally with respect to the base, any such section can be identified with an element of $H^q(\cX, \Omega^p_{\cX/S}(\cE))$, which under the Dolbeault isomorphism in turn can be represented by a $\ol\pt_\cX$-closed $(0,q)$-form $\psi$ with values in the locally free sheaf $\Omega^p_{\cX/S}(\cE)$. By Lemma~\ref{L:representative} the form $\psi$ can be chosen in a way that the restriction of $\psi$ to an arbitrary fiber $\cX_s$ is {\em harmonic}. Primarily harmonicity on a fiber $\cX_s$ refers to harmonicity with respect to $\omega_{\cX_s}$ of $(0,q)$-forms on $\cX_s$ with values in the hermitian vector bundle $\Omega^p_{\cX_s}(\cE|_{\cX_s})$, which  is equivalent to harmonicity with respect to $\omega_{\cX_s}$ of $(p,q)$-forms with values in the hermitian vector bundle $\cE|_{\cX_s}$.
Obviously these restrictions are unique, the remaining components of $\psi$ are not.

Let
\begin{equation*}
	\psi_{A_p\ol B_q}
	=
	\psi^i_{\;A_p\ol B_q} e_i
	\text{ and }
	\psi_{A_p\ol B_{q-1}\ol\jmath}
	=
	\psi^i_{\; A_p\ol B_{q-1}\ol\jmath} e_i \text{ resp.\ }
\end{equation*}
be the coefficients of
\begin{equation*}
dz^{\alpha_1}\we\ldots dz^{\alpha_p}\we dz^{\ol \beta_1}\we \ldots\we dz^{\ol \beta_q} \text{ and }dz^{\alpha_1}\we\ldots dz^{\alpha_p}\we dz^{\ol \beta_1}\we \ldots\we dz^{\ol \beta_{q-1}}\wedge ds^\ol{\jmath} \text{ resp. }
\end{equation*}
Here $A_p=(\alpha_1,\ldots, \alpha_p)$, and $\ol B_q=(\ol\beta_1,\ldots,\ol\beta_q)$.
Note that the differentials $dz^{\alpha_i}$ only occur as {\em relative differentials} originating from the locally free sheaf $\Omega^p_{\cX/S}$.

Then the closedness $\ol\pt_\cX\psi =0$ implies
\begin{equation*}
\psi\ind{}{i}{A_p,\ol B_q;\ol\jmath} = \sum_{\nu=1}^{q}(-1)^{q-\nu} \psi\ind{}{i}{A_p,\ol\beta_1,\ldots,\widehat{\ol \beta_\nu},\ldots,\ol\beta_q,\ol\jmath;\ol\beta_\nu} \; ,
\end{equation*}
where the hat symbol denotes omission. In case $\dim S = m=1$ only the above terms occur in the sum decomposition of $\psi$.

By polarization the full curvature formula can be derived from this case.

The notation can be simplified now, and in most part of this paper, we will assume that $\dim S=1$. We set $s=s_1$ and $v:=v_s=v_1$ etc. In this case we write $s$ and $\ol s$ for the indices $1$ and $\ol 1$ so that
\begin{equation*}
	v_s
	=
	\pt_s + a\ind{s}{\alpha}{} \pt_\alpha
	\;\;\;\text{and}\;\;\;
	A_s
	=
	a\ind{s}{\alpha}{;\ol\beta}\pt_\alpha\otimes dz^{\ol\beta}
	=
	A\ind{s}{\alpha}{\ol\beta}\pt_\alpha\otimes dz^{\ol\beta}.
\end{equation*}
Moreover $c(\omega)=\omega(v,\ol v)$ is called the {\em geodesic curvature} of $\omega$ and well-known to satisfy the following identity.
\begin{equation}\label{E:Semmes}
\frac{\omega^{n+1}}{(n+1)!}
=
c(\omega)\cdot\frac{\omega^n}{n!}\wedge \ii ds\wedge{d}\ol{s}.
\end{equation}
This implies that $c(\omega)$ is positive (resp.\ semi-positive) if and only if $\omega$ is positive (resp.\ semi-positive) definite also in horizontal directions.

\subsection{$L^2$-hermitian metric on $R^qf_*\Omega^p_{\cX/S}(\cE)$}
We use the notation from Section~\ref{se:hermbdl} for a holomorphic family.
Let again the base be smooth, and with no loss of generality $\dim S=1$ with local coordinate $s$.
Then the induced $L^2$ metric $\inner{\cdot,\cdot}^H$ on $R^qf_*\Omega^p_{\cX/S}(\cE)$ is given as follows:
For sections $[\chi]$ and $[\psi]$ of $R^qf_*\Omega^p_{\cX/S}(\cE)$,
\begin{equation*}
\inner{[\chi],[\psi]}^H(s)
=\inner{H(\chi\vert_{\cX_s}),H(\psi\vert_{\cX_s})}
\end{equation*}
where $H(\chi\vert_{\cX_s})$ and $H(\psi\vert_{\cX_s})$ are the harmonic representatives of $[\chi\vert_{\cX_s}],[\psi\vert_{\cX_s}]\in H^{p,q}(\cX_s,\cE\vert_{\cX_s})$, respectively and $\inner{\cdot,\cdot}$ is the $L^2$ product on $\cX_s$ with respect to $\omega\vert_{\cX_s}$ and the hermitian metric $h$.
Hence if $\chi$ and $\psi$ are the representatives given by Lemma~\ref{L:representative}, then we have
\begin{equation*}
\inner{[\chi],[\psi]}^H(s)
=
\inner{\chi\vert_{\cX_s},\psi\vert_{\cX_s}}.
\end{equation*}
In the rest of this paper, we will always use the representatives in Lemma~\ref{L:representative}.
Moreover, we will omit the restriction $\vert_{\cX_s}$ if there is no confusion.
So under this assumption, we will denote the $L^2$ inner product as
\begin{equation*}
\inner{[\chi],[\psi]}^H(s)
=
\inner{\chi,\psi}.
\end{equation*}
Note that the harmonicity of the forms $\chi|_{\cX_s}$ and $\psi|_{\cX_s}$ as $(0,q)$-forms with values in $\Omega^p_{\cX_s}(\cE_s)$ is equivalent to the harmonicity as $\cE_s$-valued $(p,q)$-forms on $\cX_s$.

\subsection{Computation of pertinent Lie derivatives}
\label{SS:Computation_Lie_derivative}
Let $[\psi]$ be a holomorphic section of $R^qf_*\Omega_{\cX/S}^p(\cE)$.
By Lemma \ref{L:representative}, a representative $\psi$ can be considered as a $\Omega^p_{\cX/S}(\cE)$-valued $\dbar_\cX$-closed $(0,q)$-form on $\cX$ whose restriction on each fiber $\cX_s$ is $\dbar$-harmonic, which is also represented by a $\cE$-valued $(p,q)$-form on $\cX$.
In this section, we will compute the Lie derivatives of $\psi$ along the horizontal lift $v$ of $\pt_s$ and its conjugate.

As we have seen in Section~\ref{SS:Lie_derivative}, the Lie derivative is not type-preserving.
More precisely, we have the type decomposition for $L_v\psi$ as
\begin{equation}
L_v\psi = L_v'\psi + L_v''\psi,
\end{equation}
where $L_v'\psi$ is of type $(p,q)$ and $L_v''\psi$ is of type
$(p-1,q+1)$.
Then the local expressions of $L_v'\psi$ and $L_v''\psi$ are as follows.
\begin{eqnarray}
L_v'\psi
&:=&
\bparen{\pt_s + a\ind{s}{\alpha}{} \pt_\alpha,
\frac{1}{p!q!}\psi\ind{}{i}{A_p,\ol B_q}e_i\otimes dz^{A_p}\we dz^{\ol B_q}}_{(p,q)}
\nonumber   \\
&=&
\frac{1}{p!q!}
\paren{
	\psi\ind{}{i}{A_p,\ol B_q;s}
	+
	a\ind{s}{\alpha}{} \psi\ind{}{i}{A_p,\ol B_q;\alpha}
	+
	\sum_{\mu=1}^p a\ind{s}{\alpha}{;\alpha_\mu}
	\psi\ind{}{i}{
	{\tiny\vtop{
	\hbox{$\alpha_1,\ldots,\alpha,\ldots,\alpha_p,\ol B_q$}\vskip-.8mm
	\hbox{$\phantom{\alpha_1,\ldots,}{|\atop\mu} $}}}}
}
	e_i\otimes dz^{A_p}\we dz^{\ol B_q}\label{eq:lvprime}
\\
L_v''\psi
&:=&
\bparen{
	\pt_s + a\ind{s}{\alpha}{} \pt_\alpha,
	\frac{1}{p!q!}\psi\ind{}{i}{A_p,\ol B_q}
	e_i\otimes dz^{A_p }\we dz^{\ol B_q}
}_{(p-1,q+1)}
\nonumber   \\
&=&
\frac{1}{(p-1)!q!}
A\ind{s}{\alpha}{\ol\beta}
\psi\ind{}{i}{\alpha,\alpha_1,\ldots,\alpha_{p-1},\ol B_q}
e_i\otimes dz^{\ol\beta}\wedge dz^{A_{p-1}}\we dz^{\ol B_q}.
\label{eq:lvsecond}
\end{eqnarray}
The semi-colon notation is used for covariant derivatives:
Here it is worth pointing out that in Equation~\eqref{eq:lvprime} the covariant $s$-derivative is taken with respect to $\nabla^\cE$ but the covariant derivatives in fiber direction are covariant derivatives with respect to $\nabla^{\cX_s}$ and $\nabla^\cE$. (See Remark~\ref{R:covariant_derivative}.)
In most part of this paper, $;s$ and $;\ol s$ stand for the covariant derivative along $s$ and $\ol s$ with respect to $\nabla^\cE$ and other covariant derivatives are with respect to $\nabla^\cE$ and the Kähler connection $\nabla^{\cX_s}$ simultaneously.
\medskip

Similarly we have a type decomposistion for the Lie derivative along $\ol v = v_{\ol s}$
$$
L_{\ol v}\psi = L_{\ol v}'\psi + L_{\ol v}''\psi,
$$
where $L_{\ol v}'\psi$ is of type $(p,q)$ and $L_{\ol v}''\psi$ is of type $(p+1,q-1)$. In local coordinates, these are
\begin{eqnarray}
\label{Lvbar'}
	L_{\ol v}'\psi
	=
	\frac{1}{p!q!}
	\left(
		\psi_{A_p\ol B_q;\ol s}
		+
		a\ind{\ol s}{\ol\beta}{}\psi_{A_p\ol B_q;\ol\beta}
		+
		\sum_{\nu=1}^q
		{a\ind{\ol s}{\ol\beta}{;\ol\beta_\nu}
		\psi\ind{}{i}{
		{\tiny\vtop{
		\hbox{$A_p,\ol\beta_1,\ldots,\ol\beta,\ldots\ol\beta_q$}
		\vskip-.8mm
		\hbox{$\phantom{A_p,\ol\beta_1,\ldots,}{|\atop\nu}$}}}}}
	\right)
	\;dz^{A_p}\wedge dz^{\ol B_q}
\end{eqnarray}
\begin{equation}\label{Lvbar''}
	L_{\ol v}''\psi
	=
	\frac{1}{p!(q-1)!}
	\sum_{\nu=1}^q
	A\ind{\ol s}{\ol\beta}{\alpha_{p+1}}
	\psi\ind{}{i}{\tiny\vtop{
	\hbox{$A_p,\ol\beta_1,\ldots,\ol\beta,\ldots,\ol\beta_q\;$}
	\vskip-.8mm
	\hbox{$\phantom{A_p\ol\beta_{p+1}\ldots}{|\atop\nu}$}}}
	\vtop{\hbox{$dz^{\alpha_1}\we\ldots\we dz^{\alpha_p}\we dz^{\ol\beta_1}\we\ldots\we dz^{\alpha_p+1}\we\ldots\we dz^{\ol\beta_q} $}
	\hbox{$\phantom{dz^{\alpha_1}\we\ldots\we dz^{\alpha_p}\we dz^{\ol\beta_1}\we\ldots\we dz}{|\atop\nu}$}}.
\end{equation}

\begin{proposition}\label{P:Lie_derivative}
We have the following identities on fibers.
\begin{eqnarray*}
L_v''\psi&=&A_s \cup \psi,\label{id1}\\
L_{\ol v}''\psi&=&(-1)^pA_{\ol s}\cup \psi, \label{id2}\\
L_{\ol v}'\psi&=&(-1)^p\dbar(\ol v \cup \psi). \label{id3}
\end{eqnarray*}
\end{proposition}

\begin{proof}
The first and second identities are clear by the definition of the cup product.
For the last identity we first recall that
\begin{equation*}
(\ol v\cup\psi)_{A_p,\ol\beta_1,\ldots,\ol\beta_{q-1}}
=
(-1)^{q-1}
\paren{
	\psi\ind{}{i}{A_p,\ol B_{q-1},\ol s}
	+
	a\ind{\ol s}{\ol\beta}{}
	\psi\ind{}{i}{A_p,\ol B_{q-1},\ol\beta}
}.
\end{equation*}
Next we have
\begin{align*}
\paren{\dbar(\ol v\cup\psi)}\ind{}{i}{A_p,\ol B_q}
&=
q\paren{
	\psi\ind{}{i}{A_p,\ol B_{q-1},\ol s}
	+
	a\ind{\ol s}{\ol\beta}{}
	\psi\ind{}{i}{A_p,\ol B_{q-1},\ol\beta}
}_{;\ol\beta_q}
\\
&=
q\paren{
	\psi\ind{}{i}{A_p,\ol B_{q-1},\ol s;\ol\beta_q}
	+
	a\ind{\ol s}{\ol\beta}{;\ol\beta_q}
	\psi\ind{}{i}{A_p,\ol B_{q-1},\ol\beta}
	+
	a\ind{\ol s}{\ol\beta}{}
	\psi\ind{}{i}{A_p,\ol B_{q-1},\ol\beta;\ol\beta_q}
}
\\
&=:I_1+I_2+I_3.
\end{align*}
Lemma~\ref{L:representative} implies that
\begin{equation*}
I_1=
\psi\ind{}{i}{A_p,\ol B_{q-1},\ol\beta_q;\ol s} \, .
\end{equation*}
After skew-symmetrizing $\ol\beta_1,\ldots,\ol\beta_q$,
\begin{align*}
I_2+I_3
&=
a\ind{\ol s}{\ol\beta}{;\ol\beta_q}
\psi\ind{}{i}{A_p,\ol B_{q-1},\ol\beta}
-
\sum_{\nu=1}^{q-1}
a\ind{\ol s}{\ol\beta}{;\ol\beta_\nu}
\psi\ind{}{i}{{
{\tiny\vtop{
\hbox{$A_p,\ol\beta_1,\ldots,\ol\beta_q,\ldots,\ol\beta_{q-1},\ol\beta$}\vskip-.8mm
\hbox{$\phantom{A_p,\ol\beta_1,\ldots,}{|\atop\nu} $}}}}
}
+
a\ind{\ol s}{\ol\beta}{}
\psi\ind{}{i}{A_p,\ol B_{q-1},\ol\beta;\ol\beta_q}
\\
&\hspace{.5cm}
-
\sum_{\nu=1}^{q-1}
a\ind{\ol s}{\ol\beta}{}
\psi\ind{}{i}{{
{\tiny\vtop{
\hbox{$A_p,\ol\beta_1,\ldots,\ol\beta_q,\ldots,\ol\beta_{q-1},\ol\beta;\ol\beta_\nu$}\vskip-.8mm
\hbox{$\phantom{A_p,\ol\beta_1,\ldots,}{|\atop\nu} $}}}}
}
\\
&=
\sum_{\nu=1}^q
a\ind{\ol s}{\ol\beta}{;\ol\beta_\nu}
\psi\ind{}{i}{{
{\tiny\vtop{
\hbox{$A_p,\ol\beta_1,\ldots,\ol\beta,\ldots,\ol\beta_q$}\vskip-.8mm
\hbox{$\phantom{A_p,\ol\beta_1,\ldots,}{|\atop\nu} $}}}}
}
+
a\ind{\ol s}{\ol\beta}{}
\psi\ind{}{i}{A_p,\ol B_q;\ol\beta}
\end{align*}
\end{proof}

While the following proposition is not necessary for subsequent computations, it could still hold interest.

\begin{proposition}\label{P:primitive}
If $\psi$ is primitive on fibers, then so are $L_v'\psi, L_v''\psi, L_{\ol v}'\psi$, and $L_{\ol v}''\psi$.
\end{proposition}

\begin{proof}
First we consider $L_v''\psi$. By Proposition~\ref{P:Lie_derivative}, we have $L_v''\psi=A_s\cup\psi$, namely
\begin{equation*}
	(A\cup \psi)_{\alpha_2,\ldots,\alpha_p,\ol\beta_1,\ldots,\ol\beta_{q+1}}
	=
	A\ind{s}{\alpha}{\ol\beta_1}
	\psi_{\alpha,\alpha_2,\ldots,\alpha_p,\ol\beta_2,\ldots,\ol\beta_{q+1}}
	-
	\sum^{q+1}_{\nu=2}
	(-1)^\nu A\ind{s}{\alpha}{\ol\beta_\nu}
	\psi_{\alpha,\alpha_2,\ldots,\alpha_p,
	\ol\beta_1,\ldots,\wh{\ol\beta}_\nu,\ldots,\ol\beta_{q+1}}
\end{equation*}
so that for
\begin{equation*}
	\paren{\Lambda(A_s\cup\psi)}
	_{\alpha_3,\ldots,\alpha_p,\ol\beta_2,\ldots,\ol\beta_{q+1}}.
\end{equation*}
The first corresponding term vanishes because of the symmetry
$$
A\ind{s}{\alpha\alpha_2}{}=A\ind{s}{\alpha_2\alpha}{},
$$
whereas the remaining sum vanishes because of $\psi$ being primitive by assumption.

Next we consider $L_v'\psi$.
Recall that
\begin{equation*}
(L_v'\psi)\ind{}{i}{A_p,\ol B_q}
=
\psi\ind{}{i}{A_p,\ol B_q;s}
+
a\ind{s}{\alpha}{} \psi\ind{}{i}{A_p,\ol B_q;\alpha}
+
\sum_{\mu=1}^p a\ind{s}{\alpha}{;\alpha_\mu}
\psi\ind{}{i}{
{\tiny\vtop{
\hbox{$\alpha_1,\ldots,\alpha,\ldots,\alpha_p,\ol B_q$}\vskip-.8mm
\hbox{$\phantom{\alpha_1,\ldots,}{|\atop\mu} $}}}}
\end{equation*}
Since $\psi$ is primitive on all fibers, we have
\begin{eqnarray*}
(\Lambda L_v'\psi)\ind{}{i}{A_{p-1},\ol B_{q-1}}
&=&
g^{\ol\beta_1\alpha_1}
\paren{
	\psi\ind{}{i}{A_p,\ol B_q;s}
	+
	a\ind{s}{\alpha}{} \psi\ind{}{i}{A_p,\ol B_q;\alpha}
	+
	\sum_{\mu=1}^p a\ind{s}{\alpha}{;\alpha_\mu}
	\psi\ind{}{i}{
	{\tiny\vtop{
	\hbox{$\alpha_1,\ldots,\alpha,\ldots,\alpha_p,\ol B_q$}\vskip-.8mm
	\hbox{$\phantom{\alpha_1,\ldots,}{|\atop\mu} $}}}}
}
\\
&=&
g^{\ol\beta_1\alpha_1}
a\ind{s}{\alpha}{;\alpha_1}
\psi\ind{}{i}{\alpha,\alpha_2,\ldots,\alpha_p,\ol B_q}.
\end{eqnarray*}
Taking into account that $g^{\ol\beta_1\alpha_1}a\ind{s}{\alpha}{;\alpha_1}=\pt_sg^{\ol\beta_1\alpha}$, we have
\begin{eqnarray*}
g^{\ol\beta_1\alpha_1}
a\ind{s}{\alpha}{;\alpha_1}
\psi\ind{}{i}{\alpha,\alpha_2,\ldots,\alpha_p,\ol B_q}
&=&
g\ind{}{\ol\beta_1\alpha}{\vert s}
\psi\ind{}{i}{\alpha,\alpha_2,\ldots,\alpha_p,\ol B_q}
\\
&=&
\paren{
	g^{\ol\beta_1\alpha}
	\psi\ind{}{i}{\alpha,\alpha_2,\ldots,\alpha_p,\ol B_q}
}_{\vert s}
-
g^{\ol\beta_1\alpha}
\psi\ind{}{i}{\alpha,\alpha_2,\ldots,\alpha_p,\ol B_q\vert s}
\\
&=&
-g^{\ol\beta_1\alpha}
\psi\ind{}{i}{\alpha,\alpha_2,\ldots,\alpha_p,\ol B_q;s}
-
g^{\ol\beta_1\alpha}
\psi\ind{}{j}{\alpha,\alpha_2,\ldots,\alpha_p,\ol B_q;s}
\theta\ind{j}{i}{}(\pt_s).
\end{eqnarray*}
Note that non-conjugate indices $\alpha_\nu$ refer to the sheaf $\Omega^p_{\cX/S}$ so that the covariant derivative with respect to $s$ only concerns the hermitian metric on the bundle $\cE$.
\medskip

Similar computations show that the other two terms are also primitive.
\end{proof}

The following lemma is well-known. (See \cite{Schumacher2012,Naumann2021}.)
\begin{lemma}\label{L:gpar}
The Lie derivative $L_v\omega$ of $\omega$ along the horizontal lift $v$ vanishes.
In particular, so does $L_v dV_s$ where $dV_s=\omega^n/n!\vert_{\cX_s}$.
\end{lemma}

We begin computing the curvature by studying the first order variation of the metric tensor. Using Lie derivatives, the pointwise inner products can be broken up:

\begin{proposition}\label{P:first_derivative}
We have the following first and second variation formulas.
\begin{align*}
	\pd{}{s}
		\inner{\chi,\psi}
		&=
		\inner{L_v\chi,\psi} = \inner{L_v'\chi,\psi}
		\\
	-\pd{^2}{\ol s\partial s}
	\inner{\chi,\psi}
	&=
	\inner{[L_v,L_{\ol v}]\chi,\psi}
		-
		\inner{L_v'\chi,L_v'\psi}
		+
		\inner{L_v''\chi,L_v''\psi}
		+
		\inner{L_{\ol v}'\chi,L_{\ol v}'\psi}
		-
		\inner{L_{\ol v}''\chi,L_{\ol v}''\psi}
\end{align*}
\end{proposition}

This fact has been used in various situations (see e.g.\ \cite{Siu1986,Schumacher2012,Berndtsson2009,Naumann2021}), see also Wang's approach in \cite{Wang2016Ar}.
\begin{proof}
First we recall the first variation formula.
For any local smooth section of $\chi$ and $\psi$ of $R^qf_*\Omega_{\cX/S}^p(\cE)$, one has
\begin{equation*}
	\pd{}{s}\inner{\chi,\psi}
	=
	\inner{L_v\chi,\psi}+ \inner{\chi, L_{\ol v} \psi},
\end{equation*}
and
\begin{equation*}
	\pd{}{\ol s}\inner{\chi,\psi}
	=
	\inner{L_{\ol v}\chi,\psi}+ \inner{\chi, L_v\psi}.
\end{equation*}
We refer to the proof of Proposition 4.4 in \cite{Naumann2021}, where the author establishes the case for $R^qf_*\Omega_{\cX/S}^{n-q}(\mathcal{L})$ with a holomorphic line bundle $\mathcal{L}$. However, the proof can be applied to $R^qf_*\Omega_{\cX/S}^p(\cE)$ with general $(p,q)$ and a holomorphic vector bundle $\mathcal{E}$.

Now we assume that $\chi$ and $\psi$ are holomorphic section of $R^qf_*\Omega_{\cX/S}^p(\cE)$.
Then it follows from the bidegree reason and Proposition~\ref{P:Lie_derivative}  that
\begin{equation*}
	\pd{}{s}\inner{\chi,\psi}
	=
	\inner{L_v'\chi,\psi}+ \inner{\chi, L_{\ol v}'\psi}
	=
	\inner{L_v'\chi,\psi}+ \inner{\chi,\dbar(\ol v\cup\psi)}
	=
	\inner{L_v'\chi,\psi}.
\end{equation*}
For the second assertion, we first note that the first assertion yields that
\begin{equation*}
	0=
	\pd{}{s}\inner{L_{\ol v}'\chi,\psi}
	=
	\inner{L_v'L_{\ol v}'\chi,\psi}
	+
	\inner{L_{\ol v}'\chi,L_{\ol v}'\psi}
	=
	\inner{L_vL_{\ol v}\chi,\psi}
	-
	\inner{L_v''L_{\ol v}''\chi,\psi}
	+
	\inner{L_{\ol v}'\chi,L_{\ol v}'\psi}.
\end{equation*}
On the other hand, one has
\begin{align*}
	-\pd{^2}{\ol s\partial s}
	\inner{\chi,\psi}
	=
	-\pd{}{\ol s}\inner{L_v'\chi,\psi}
	&=
	-
	\inner{L_{\ol v}'L_v'\chi,\psi}
	-
	\inner{L_v'\chi,L_v'\psi}
	\\
	&=
	-
	\inner{L_{\ol v}L_v\chi,\psi}
	+
	\inner{L_{\ol v}''L_v''\chi,\psi}
	-
	\inner{L_v'\chi,L_v'\psi}.
\end{align*}
Summing up above two equations,
\begin{equation*}
	-\pd{^2}{\ol s\partial s}
	\inner{\chi,\psi}
	=
	\inner{[L_v,L_{\ol v}]\chi,\psi}
		-
		\inner{L_v'\chi,L_v'\psi}
		+
		\inner{L_{\ol v}''L_v''\chi,\psi}
		+
		\inner{L_{\ol v}'\chi,L_{\ol v}'\psi}
		-
		\inner{L_v''L_{\ol v}''\chi,\psi}
\end{equation*}
Now we need the following lemma.
\begin{lemma}
\label{L:formal_adjoint_A}
For any smooth section $\chi$ of $R^{q-1}f_*\Omega_{\cX/S}^{p+1}(\cE)$ and $\psi$ of $R^qf_*\Omega_{\cX/S}^p(\cE)$, the following holds.
\begin{equation*}
	\inner{A_s\cup\chi,\psi}
	=
	\inner{\chi,A_{\ol s}\cup\psi}
\end{equation*}
In particular, $A_{\ol s}\cup\vartextvisiblespace$ is the formal adjoint operator of $A_s\cup\vartextvisiblespace$.
\end{lemma}
\begin{proof}
Since $\psi$ is $\cE$-valued $(p+1,q-1)$ form on fibers, Equation~\eqref{eq:lvsecond} states that
\begin{equation*}
	\frac{(-1)^p}{p!(q-1)!}
	A\ind{s}{\alpha}{\ol\beta}
	\chi\ind{}{i}{\alpha,A_p,\ol B_{q-1}}
	e_i\otimes dz^{A_{p-1}}\we dz^{\ol\beta}\we dz^{\ol B_{q-1}},
\end{equation*}
it follows that
\begin{eqnarray*}
	\inner{A_s\cup\chi,\psi}
	&=&
	\frac{(-1)^p}{(p!)^2(q-1)!q!}
	\int_X
	A\ind{s}{\alpha}{\ol\beta}
	\chi\ind{}{i}{\alpha,A_p,\ol B_{q-1}}
	\ol{
		\psi\ind{}{j}{C_p,\ol D_q}
	}
	h_{i\ol\jmath}
	g^{A_p\ol C_p}
	g^{\ol B_q D_q}dV_\omega
	\\
	&=&
	\frac{(-1)^p}{(p!)^2((q-1)!)^2}
	\int_X
	\chi\ind{}{i}{\alpha,A_p,\ol B_{q-1}}
	\ol{
		A\ind{\ol s}{\ol\delta}{\gamma}
		\psi\ind{}{j}{C_p,\ol\delta,\ol D_{q-1}}
	}
	h_{i\ol\jmath}
	g^{\alpha\ol\gamma}
	g^{A_p\ol C_p}
	g^{\ol B_{q-1} D_{q-1}}dV_\omega
	\\
	&=&
	\inner{\chi,A_{\ol s}\cup\psi}.
\end{eqnarray*}
\end{proof}
Proposition~\ref{P:Lie_derivative} and Lemma~\ref{L:formal_adjoint_A} together imply that
\begin{equation*}
		\inner{L_{\ol v}''L_v''\chi,\psi}
		=
		\inner{A_{\ol s}\cup L_v''\chi,\psi}
		=
		\inner{L_v''\chi,A_s\cup\psi}
		=
		\inner{L_v''\chi,L_v''\psi}.
\end{equation*}
Similarly, we obtain $\inner{L_v''L_{\ol v}''\chi,\psi}=\inner{L_{\ol v}''\chi,L_{\ol v}''\psi}$, which completes the proof.
\end{proof}

We recall the following identity which is proved in \cite{Naumann2021,Wang2016Ar}.
\begin{equation}\label{eq:vvb}
[L_\ol v, L_v]
=
L_{[\ol v,v]}-\Theta_h(\cE)(v,\ol v).
\end{equation}
Note that $[\ol v,v]$ is tangential to fibers.
In particular, $L_{[\ol v,v]}$ does not contain any differentiation along the horizontal direction. (See also Lemma~\ref{L:commutator_vbar_v}.)
So the second derivative of the metric tensor for $R^qf_*\Omega_{\cX/S}^p(\cE)$ turns into
\begin{align*}
-\pd{^2}{\ol s\partial s}
\inner{\chi,\psi}
&=
\inner{L_{[v,\ol v]}\chi,\psi}
+
\inner{\Theta_h(\cE)(v,\ol v)\chi,\psi}
-
\inner{L_v'\chi,L_v'\psi}
+
\inner{A_s\cup\chi,A_s\cup\psi}
\\
&\hspace{.5cm}
+
\inner{L_{\ol v}'\chi,L_{\ol v}'\psi}
-
\inner{A_{\ol s}\cup\chi,A_{\ol s}\cup\psi}.
\end{align*}

\section{Auxiliary formulas}
In the previous section, we computed the second derivative of the natural $L^2$ metric $\inner{\cdot,\cdot}^H$ on the higher direct image sheaf $R^qf_*\Omega_{\cX/S}^p(E)$.
For a further computation of the Lie derivative terms, we need several auxiliary formulas, which are discussed in this section.
Although some of the formulas are already computed in \cite{Schumacher2012,Naumann2021}, we will give a detailed proof for the sake of the reader's convenience.

\begin{lemma}\label{L:derivative_christoffel_symbol}
The following identity holds.
\begin{equation*}
\pt_s\Gamma\ind{\beta}{\alpha}{\gamma}
=
-a\ind{s}{\alpha}{;\beta\gamma}
\quad\text{and}\quad
\pt_{\ol s}\Gamma\ind{\beta}{\alpha}{\gamma}
=
-g^{\ol\tau\alpha}a_{\ol s\gamma;\ol\tau\beta}
\end{equation*}
\end{lemma}

\begin{proof}
See \cite{Schumacher2012,Naumann2021}.
\end{proof}

\begin{lemma}
We have the following.
\begin{align}
	\label{eq:aux1}
	\chi\ind{}{i}{A_p,\ol B_q;s\ol\beta}
	&=
	\chi\ind{}{i}{A_p,\ol B_q;\ol\beta s}
	-
	\Theta\ind{j}{i}{s\ol\beta}\chi\ind{}{j}{A_p,\ol B_q}
	+
	\sum_{\nu=1}^{q} g^{\ol\delta \gamma} a_{\ol s \gamma;\ol\beta\ol\beta_\nu}
 	\chi\lowerindB{i}{A_p}{\ol\beta_1}{\ol\delta}{\ol\beta_q}{\nu}.
	\\
	\label{eq:aux2}
	\chi\ind{}{i}{A_p,\ol B_q;\ol s\gamma}
	&=
	\chi\ind{}{i}{A_p,\ol B_q;\gamma\ol s}
	-
	\Theta\ind{j}{i}{s\ol\beta}\chi\ind{}{j}{A_p,\ol B_q}
	-
	\sum^p_{\mu =1}
	g^{\ol\tau\gamma}
	\paren{
		A_{\ol s\alpha_\mu\gamma;\ol\tau}
		-
		a_{\ol s\sigma}R\ind{\alpha_\mu}{\sigma}{\ol\tau\gamma}
	}
	\chi\lowerindA{i}{\alpha_1}{\alpha}{\alpha_p}{\ol B_q}{\nu}.
	\\
	\label{eq:aux3}
	\chi\ind{}{i}{A_p,\ol B_q;s\gamma}
	&=
	\chi\ind{}{i}{A_p,\ol B_q;\gamma s}
	+
	\sum_{\mu=1}^p
	a\ind{s}{\alpha}{;\alpha_\mu\gamma}
	\chi\lowerindA{i}{\alpha_1}{\alpha}{\alpha_p}{\ol B_q}{\mu}.
\end{align}
\end{lemma}
Note that in \eqref{eq:aux2}, the covariant derivative along $s$ denotes the covariant derivative with respect to the Chern connection $\nabla^\cE$ on the holomorphic vector bundle $\cE$ and the covariant derivative along $z^\gamma$ is the one with respect to $\nabla^\cE$ and $\nabla^{\cX_s}$.
\begin{proof}
Here we will  prove \eqref{eq:aux3}.
The two other identities are computed in the same way.
The definition of the covariant derivatives implies
\begin{equation*}
	\chi\ind{}{i}{A_p,\ol B_q;s}
	=
	\chi\ind{}{i}{A_p,\ol B_q\vert s}
	+
	\theta\ind{j}{i}{s}
	\chi\ind{}{j}{A_p,\ol B_q}
\end{equation*}
and
\begin{equation*}
	\chi\ind{}{i}{A_p,\ol B_q;\gamma}
	=
	\chi\ind{}{i}{A_p,\ol B_q\vert\gamma}
	+
	\theta\ind{j}{i}{\gamma}
	\chi\ind{}{j}{A_p,\ol B_q}
	-
	\sum_{\mu=1}^p
	\Gamma\ind{\alpha_\mu}{\alpha}{\gamma}
	\chi\ind{}{i}{
		{\tiny\vtop{
		\hbox{$\alpha_1,\ldots,\alpha,\ldots,\alpha_p,\ol B_q$}
		\vskip-.8mm
		\hbox{$\phantom{\alpha_1,\ldots,}{|\atop\mu} $}}}}.
\end{equation*}
Taking one more derivative, it follows that
\begin{align*}
	\chi\ind{}{i}{A_p,\ol B_q;s\gamma}
	=
	&
	\paren{
		\chi\ind{}{i}{A_p,\ol B_q\vert s}
		+
		\theta\ind{j}{i}{s}
		\chi\ind{}{j}{A_p,\ol B_q}
	}_{\vert\gamma}
	+
	\theta\ind{k}{i}{\gamma}\chi\ind{}{k}{A_p,\ol B_q\vert s}
	+
	\theta\ind{k}{j}{\gamma}
	\theta\ind{j}{i}{s}
	\chi\ind{}{k}{A_p,\ol B_q}
	\\
	&-
	\sum_{\mu=1}^p
	\paren{
		\chi\ind{}{i}{
		{\tiny\vtop{
		\hbox{$\alpha_1,\ldots,\alpha,\ldots,\alpha_p\ol B_q\vert s$}
		\vskip-.8mm
		\hbox{$\phantom{\alpha_1,\ldots,}{|\atop\mu} $}}}
		}
		+
		\theta\ind{j}{i}{s}
		\chi\ind{}{j}{
		{\tiny\vtop{
		\hbox{$\alpha_1,\ldots,\alpha,\ldots,\alpha_p\ol B_q$}
		\vskip-.8mm
		\hbox{$\phantom{\alpha_1,\ldots,}{|\atop\mu} $}}}
		}
	}
	\Gamma\ind{\alpha_\mu}{\alpha}{\gamma}.
\end{align*}
Likewise, we also have
\begin{align*}
	\chi\ind{}{i}{A_p,\ol B_q;\gamma s}
	=
	&
	\paren{
		\chi\ind{}{i}{A_p,\ol B_q\vert\gamma}
		+
		\theta\ind{j}{i}{\gamma}
		\chi\ind{}{j}{A_p,\ol B_q}
		-
		\sum_{\mu=1}^p
		\Gamma\ind{\alpha_\mu}{\alpha}{\gamma}
		\chi\ind{}{i}{
			{\tiny\vtop{
			\hbox{$\alpha_1,\ldots,\alpha,\ldots,\alpha_p\ol B_q$}
			\vskip-.8mm
			\hbox{$\phantom{\alpha_1,\ldots,}{|\atop\mu} $}}}}
	}_{\vert s}
	+
	\theta\ind{k}{i}{s}
	\chi\ind{}{k}{A_p,\ol B_q\vert\gamma}
	\\
	&
	+
	\theta\ind{k}{j}{s}
	\theta\ind{j}{i}{\gamma}
	\chi\ind{}{k}{A_p,\ol B_q}
	-
	\theta\ind{k}{i}{s}
	\sum_{\mu=1}^p
	\Gamma\ind{\alpha_\mu}{\alpha}{\gamma}
	\chi\ind{}{k}{
		{\tiny\vtop{
		\hbox{$\alpha_1,\ldots,\alpha,\ldots,\alpha_p\ol B_q$}
		\vskip-.8mm
		\hbox{$\phantom{\alpha_1,\ldots,}{|\atop\mu} $}}}}.
\end{align*}
Then Lemma \ref{L:derivative_christoffel_symbol} implies that
\begin{equation*}
\chi\ind{}{i}{A_p,\ol B_q;s\gamma}
=
\chi\ind{}{i}{A_p,\ol B_q;\gamma s}
+
\sum_{\mu=1}^p
a\ind{s}{\alpha}{;\alpha_\mu\gamma}
\chi\ind{}{i}{
	{\tiny\vtop{
	\hbox{$\alpha_1,\ldots,\alpha,\ldots,\alpha_p,\ol B_q$}
	\vskip-.8mm
	\hbox{$\phantom{\alpha_1,\ldots,}{|\atop\mu} $}}}
	},
\end{equation*}
which gives \eqref{eq:aux3}.
\end{proof}

Now we compute exterior derivatives of $L_v'\psi$ and $L_{\ol v}''\psi$.

\begin{proposition}\label{P:dbar_Lv}
The following holds on fibers.
\begin{equation*}
	\dbar(L_v'\chi)
	=
	A_s\cup(\pt\chi)
	+
	\pt\big( A_s \cup\chi\big)
	+
	\eta_s
	\we\chi.
\end{equation*}
where $\eta_s$ is defined by
\begin{equation*}
	\eta_s=-\paren{v\cup \Theta}\vert_{\cX_s}.
\end{equation*}
\end{proposition}

\begin{proof}
Note first that after skew-symmetrizing $\ol\beta,\ol\beta_1,\ldots,\ol\beta_q$, Equation~\eqref{eq:aux1} becomes
\begin{equation*}
	\chi\ind{}{i}{A_p,\ol B_q;s\ol\beta}
	=
	\chi\ind{}{i}{A_p,\ol B_q;\ol\beta s}
	-
	\Theta\ind{j}{i}{s\ol\beta}\chi\ind{}{j}{A_p,\ol B_q}
	=
	-
	\Theta\ind{j}{i}{s\ol\beta}\chi\ind{}{j}{A_p,\ol B_q}.
\end{equation*}
The $\dbar$-closedness of $\chi$ is essential.
Now, starting from \eqref{eq:lvprime} we get, using \eqref{eq:aux1} and Proposition~\ref{P:commutator_covariant_derivatives},
\begin{eqnarray*}
\ol\pt L_v'\chi
&=&
\frac{1}{p!q!}
\Bigg(
\chi\ind{}{i}{A_p,\ol B_q;s\ol\beta}
+
A\ind{s}{\alpha}{\ol\beta}\chi\ind{}{i}{A_p,\ol B_q;\alpha}
+
a\ind{s}{\alpha}{}\chi\ind{}{i}{A_p,\ol B_q;\alpha\ol\beta}
+
\sum^p_{\mu =1}a\ind{s}{\alpha}{;\alpha_\mu\ol\beta}
\chi\ind{}{i}{
{\tiny\vtop{
\hbox{$\alpha_1,\ldots,\alpha,\ldots,\alpha_p,\ol B_q$}\vskip-.8mm
\hbox{$\phantom{\alpha_1,\ldots,}{|\atop\mu } $}}}}
\\
&&
\hspace{1cm}+
\sum^p_{\mu =1} a\ind{s}{\alpha}{;\alpha_\mu }
\chi\ind{}{i}{
{\tiny\vtop{
\hbox{$\alpha_1,\ldots,\alpha,\ldots,\alpha_p,\ol B_q;\ol\beta$}\vskip-.8mm
\hbox{$\phantom{\alpha_1,\ldots,}{|\atop\mu } $}}}}
\Bigg)
\cdot e_i\otimes dz^{\ol\beta}\we dz^{A_p}\we dz^{\ol B_q}
\\
&=&
\frac{1}{p!q!}
\paren{
	-\Theta\ind{j}{i}{}(v,\pt_{\ol\beta})\chi\ind{}{j}{A_p,\ol B_q}
	+
	A\ind{s}{\alpha}{\ol\beta}\chi\ind{}{i}{A_p,\ol B_q;\alpha}
	+
	\sum^p_{\mu =1}
	A\ind{s}{\alpha}{\ol\beta;\alpha_\mu }
	\chi\ind{}{i}{
	{\tiny\vtop{
	\hbox{$\alpha_1,\ldots,\alpha,\ldots,\alpha_p,\ol B_q$}\vskip-.8mm
	\hbox{$\phantom{\alpha_1,\ldots,}{|\atop\mu } $}}}}
}
\\
&&\hspace{1cm}\cdot
e_i\otimes dz^{\ol\beta}\we dz^{A_p}\we dz^{\ol B_q}.
\end{eqnarray*}
The first term equals to $\eta_s$, and the second and third terms are
\begin{align*}
\frac{1}{p!}
\Bigg[
A\ind{s}{\alpha}{\ol\beta} &
(\partial\chi)\ind{}{i}{\alpha,\alpha_1,\ldots,\alpha_p,\ol B_q}
+
\sum^p_{\nu =1}
\Big(
A\ind{s}{\alpha}{\ol\beta}
\chi\ind{}{i}{\tiny \vtop{
\hbox{$\alpha_1,\ldots,\alpha,\ldots,\alpha_p,\ol B_q$}
\vskip-1.5mm\hbox{$ \phantom{\alpha_1,\ldots,}{|\atop \nu }$}}}
\Big)_{;\alpha_\nu }
\Bigg]
e^i\otimes
dz^{\ol\beta}\we dz^{A_p}\we dz^{\ol B_q}.
\\
&=
A_s\cup(\pt\chi)
+
\frac{1}{p!}
\sum^p_{\nu =1}
(-1)^{\nu+1}
\big(
A\ind{s}{\alpha}{\ol\beta}
\chi\ind{}{i}{\alpha,\alpha_1,\ldots,\widehat\alpha_\nu,\ldots,\alpha_p,\ol B_q}
\big)_{;\alpha_\nu}e^i\otimes
dz^{A_p}\we dz^{\ol\beta}\we dz^{\ol B_q}
\\
&=
A_s\cup(\pt\chi)
+
\pt\big( A_s \cup \chi\big).
\end{align*}
This completes the proof.
\end{proof}

According to Proposition~\ref{P:atiyah} the forms $(\eta_{s},A_s)$ define distinguished \ks\ forms for the deformation of the pair $(X,E)$.

\begin{proposition}\label{P:dbar*_Lvbar}
We have the following equation on fibers.
\begin{equation*}
	\dbar^*(L_{\ol v}'\psi)
	=
	(-1)^p\pt^*(A_{\ol s}\cup\psi)
	+
	(-1)^pA_{\ol s}\cup\pt^*\psi
	+
	[\Lambda,\ii\eta_{\ol s}]\psi
\end{equation*}
where $\eta_{\ol s}$ is defined by
\begin{equation*}
\eta_{\ol s}=-\paren{\ol v\cup \Theta}\vert_{\cX_s}.
\end{equation*}
\end{proposition}

\begin{proof}
Recall that
\begin{eqnarray*}
L_{\ol v}'\psi
=
\frac{1}{p!q!}
\left(
	\psi\ind{}{i}{A_p,\ol B_q;\ol s}
	+
	a\ind{\ol s}{\ol\beta}{}
	\psi\ind{}{i}{A_p,\ol B_q;\ol\beta}
	+
	\sum_{\nu=1}^q
	a\ind{\ol s}{\ol\beta}{;\ol\beta_\nu}
	\psi\lowerindB{i}{A_p}{\ol\beta_1}{\ol\beta}{\ol\beta_q}{\nu}
\right)
e_i\otimes dz^{A_p}\wedge dz^{\ol B_q}.
\end{eqnarray*}
Taking $\dbar^*$, we have
\begin{eqnarray*}
\paren{\dbar^*L_{\ol v}'\psi}
\ind{}{i}{A_p,\ol\beta_2,\ldots,\ol\beta_q}
&=&
(-1)^{p+1}g^{\ol\beta_1\alpha}
\paren{
	\psi\ind{}{i}{A_p,\ol B_q;\ol s}
	+
	a\ind{\ol s}{\ol\beta}{}
	\psi\ind{}{i}{A_p,\ol B_q;\ol\beta}
	+
	\sum_{\nu=1}^q
	a\ind{\ol s}{\ol\beta}{;\ol\beta_\nu}
	\psi\ind{}{i}{
	{\tiny\vtop{
	\hbox{$A_p,\ol\beta_1,\ldots,\ol\beta,\ldots\ol\beta_q$}
	\vskip-.8mm
	\hbox{$\phantom{A_p,\ol\beta_1,\ldots,}{|\atop\nu}$}}}}
}_{;\alpha}
\\
&=&
(-1)^{p+1}g^{\ol\beta_1\alpha}
\Bigg(
	\psi\ind{}{i}{A_p,\ol B_q;\ol s\alpha}
	+
	a\ind{\ol s}{\ol\beta}{;\alpha}
	\psi\ind{}{i}{A_p,\ol B_q;\ol\beta}
	+
	a\ind{\ol s}{\ol\beta}{}
	\psi\ind{}{i}{A_p,\ol B_q;\ol\beta\alpha}
	\\
	&&\hspace{1cm}+
	\sum_{\nu=1}^q
	a\ind{\ol s}{\ol\beta}{;\ol\beta_\nu\alpha}
	\psi\ind{}{i}{
	{\tiny\vtop{
	\hbox{$A_p,\ol\beta_1,\ldots,\ol\beta,\ldots\ol\beta_q$}
	\vskip-.8mm
	\hbox{$\phantom{A_p,\ol\beta_1,\ldots,}{|\atop\nu}$}}}}
	+
	\sum_{\nu=1}^q
	a\ind{\ol s}{\ol\beta}{;\ol\beta_\nu}
	\psi\ind{}{i}{
	{\tiny\vtop{
	\hbox{$A_p,\ol\beta_1,\ldots,\ol\beta,\ldots\ol\beta_q;\alpha$}
	\vskip-.8mm
	\hbox{$\phantom{A_p,\ol\beta_1,\ldots,}{|\atop\nu}$}}}}
\Bigg)
\end{eqnarray*}
Since $\psi$ is $\dbar$-closed on fibers, it follows that
\begin{eqnarray*}
\paren{\dbar^*L_{\ol v}'\psi}\ind{}{i}{A_p,\ol\beta_2,\ldots,\ol\beta_q}
&=&
(-1)^{p+1}g^{\ol\beta_1\alpha}
\Bigg(
	-
	\Theta\ind{j}{i}{\ol s\alpha}\psi\ind{}{j}{A_p,\ol B_q}
	+
	\sum_{\mu=1}^pA\ind{\ol s}{\ol\beta}{\gamma;\ol\beta_1}
	\chi\ind{}{i}{
	{\tiny\vtop{
	\hbox{$\alpha_1,\ldots,\alpha,\ldots,\alpha_p,\ol B_q$}\vskip-.8mm
	\hbox{$\phantom{\alpha_1,\ldots,}{|\atop\mu } $}}}}
	+
	A\ind{\ol s}{\ol\beta}{\alpha}\psi\ind{}{i}{A_p,\ol B_q;\ol\beta}
	\\
	&&\hspace{1cm}
	-
	a\ind{\ol s}{\ol\beta}{}
	\Theta\ind{j}{i}{\ol\beta\alpha}\psi\ind{}{j}{A_p,\ol B_q}
	+
	A\ind{\ol s}{\ol\beta}{\alpha;\ol\beta_1}
	\psi\ind{}{i}{A_p,\ol\beta,\ol\beta_2,\ldots,\ol\beta_q}
	+
	a\ind{\ol s}{\ol\beta}{;\ol\beta_1}
	\psi\ind{}{i}{A_p,\ol\beta,\ol\beta_2,\ldots,\ol\beta_q;\alpha}
\Bigg)
\end{eqnarray*}
The first and forth terms  become
\begin{equation*}
	(-1)^p g^{\ol\beta_1\alpha}
	\paren{
		\Theta\ind{j}{i}{\ol s\alpha}
		+
		a\ind{\ol s}{\ol\beta}{}
		\Theta\ind{j}{i}{\ol\beta\alpha}
	}
	\psi\ind{}{j}{A_p,\ol B_q}
	=
	[\Lambda,\ii\eta_{\ol s}]\psi.
\end{equation*}
And the last term in the first equation vanishes, since the $\dbar^*$-closedness of $\psi$ implies that
\begin{align*}
	g^{\ol\beta_1\alpha}
	a\ind{\ol s}{\ol\beta}{;\ol\beta_1}
	\psi\ind{}{i}{A_p,\ol\beta,\ol\beta_2,\ldots,\ol\beta_q;\alpha}
	&=
	\paren{\pt_sg^{\ol\beta\alpha}}
	\psi\ind{}{i}{A_p,\ol\beta,\ol\beta_2,\ldots,\ol\beta_q;\alpha}
	\\
	&=
	\paren{\pt_sg^{\ol\beta\alpha}
	\psi\ind{}{i}{A_p,\ol\beta,\ol\beta_2,\ldots,\ol\beta_q;\alpha}}
	-
	g^{\ol\beta\alpha}
	\psi\ind{}{i}{A_p,\ol\beta,\ol\beta_2,\ldots,\ol\beta_q;\alpha\vert s}
	\\
	&=
	g^{\ol\beta\alpha}
	\psi\ind{}{i}{A_p,\ol\beta,\ol\beta_2,\ldots,\ol\beta_q;\alpha s}
	+
	g^{\ol\beta\alpha}
	\theta\ind{j}{i}{}(\pt_s)
	\psi\ind{}{j}{A_p,\ol\beta,\ol\beta_2,\ldots,\ol\beta_q;\alpha}
	=0.
\end{align*}
The remaining terms are computed by means of the following lemma.
\begin{lemma}
	The following holds.
	\begin{align*}
		&\paren{
			\pt^*\paren{A_{\ol s}\cup\psi}
			+
			A_{\ol s}\cup\paren{\pt^*\psi}
		}
		\ind{}{i}{A_p,\ol B_{q-1}}
		\\
		&\hspace{1.5cm}=
		-g^{\ol\delta\gamma}
		\paren{
			A\ind{\ol s}{\ol\beta_1}{\gamma;\ol\delta}
			\psi\ind{}{i}{A_p,\ol\beta_1,\ldots,\ol\beta_q}
			+
			A\ind{\ol s}{\ol\beta_1}{\gamma}
			\psi\ind{}{i}{A_p,\ol\beta_1,\ldots,\ol\beta_q;\ol\delta}
			+
			\sum_{\mu=1}^p
			(-1)^\mu A\ind{\ol s}{\ol\beta_1}{\alpha_\mu;\ol\delta}
			\psi\ind{}{i}{\alpha_0,\ldots,\wh\alpha_\mu,\ldots,\alpha_p,\ol B_{q-1}}	
		}.
	\end{align*}
\end{lemma}
\begin{proof}
Note that
\begin{align*}
	\paren{A_{\ol s}\cup\psi}\ind{}{i}{A_{p+1}\ol B_{q-1}}
	=
	A\ind{\ol s}{\ol\beta_1}{\alpha_0}\psi^i_{A_p,\ol B_{q-1}}
	+
	\sum_{\mu=1}^p
	(-1)^\mu A\ind{\ol s}{\ol\beta_1}{\alpha_\mu}
	\psi\ind{}{i}{\alpha_0,\ldots,\wh\alpha_\mu,\ldots,\alpha_p,\ol B_{q-1}}.
\end{align*}
Although this expression seems different from (6.6), it is obtained by skew-symmetrizing the coefficients of \eqref{Lvbar''}.
The conclusion can then be easily reached through a straightforward computation.
\end{proof}
This completes the proof.
\end{proof}

\begin{proposition}\label{P:auxiliary_formulas}
We have the following identities on fibers.
\begin{eqnarray*}
\dbar^*(L_v'\chi)&=&0
\\
\dbar(L_{\ol v}'\chi)&=&0
\\
\pt^*(A_s\cup\chi)&=&0
\\
\pt(A_{\ol s}\cup\chi)&=&-A_{\ol s}\cup\pt\chi.
\end{eqnarray*}
\end{proposition}

\begin{proof}
The second identity directly follows from Proposition~\ref{P:Lie_derivative}.
For the third identity, we recall that
\begin{equation*}
\paren{
	A_s\cup\chi
}
_{A_{p-1},\ol\beta,\ol B_q}
=
(-1)^{p-1}
A\ind{s}{\alpha}{\ol\beta}
\chi_{\alpha,\alpha_1,\ldots,\alpha_{p-1},\ol B_q}.
\end{equation*}
Hence we get
\begin{eqnarray*}
\paren{
	\pt^*(A_s\cup\chi)
}
_{A_{p-2},\ol\beta,\ol\beta}
&=&
(-1)^p
g^{\ol\delta\gamma}
\paren{
	A\ind{s}{\alpha}{\ol\beta}
	\chi_{\alpha,\gamma,\alpha_1,\ldots,\alpha_{p-2},\ol B_q}
}_{;\ol\delta}
\\
&=&
(-1)^p
A\indd{s}{\alpha}{\ol\beta}{;\gamma}
\chi_{\alpha,\gamma,\alpha_1,\ldots,\alpha_{p-2}}
+
(-1)^p
g^{\ol\delta\gamma}
A\ind{s}{\alpha}{\ol\beta}
\chi_{\alpha,\gamma,\alpha_1,\ldots,\alpha_{p-2},\ol B_q;\ol\delta}.
\end{eqnarray*}
The first term vanishes, since it is symmetric in the upper indices $\alpha, \gamma$ and skew-symmetric in the lower indices $\alpha, \gamma$.
The second term also vanishes, since $\chi$ is $\pt^*$-closed.
This proves the third identity.
The last identity concerns
\begin{equation*}
\paren{
	A_{\ol s}\cup\chi
}
_{\alpha,A_p,\ol B_{q-1}}
=
A\ind{\ol s}{\ol\beta}{\alpha}
\chi_{A_p,\ol\beta,\ol B_{q-1}}.
\end{equation*}
It follows that
\begin{eqnarray*}
\paren{
	\pt(A_{\ol s}\cup\chi)
}
_{\gamma,\alpha,A_p,\ol B_{q-1}}
&=&
\paren{
	A\ind{\ol s}{\ol\beta}{\alpha}
	\chi_{A_p,\ol\beta,\ol B_{q-1}}
}_{;\gamma}
\\
&=&
A\ind{\ol s}{\ol\beta}{\alpha;\gamma}
\chi_{A_p,\ol\beta,\ol B_{q-1}}
+
A\ind{\ol s}{\ol\beta}{\alpha}
\chi_{A_p,\ol\beta,\ol B_{q-1};\gamma}
\\
&=&
-
\paren{
	A_{\ol s}\cup\pt\chi
}_{\gamma,\alpha,A_p,\ol B_{q-1}}.
\end{eqnarray*}
We will prove the first identity, which is already contained in \cite{Schumacher2012,Naumann2021} when $\cE$ is a holomorphic line bundle and $p+q=n$.
The proof is somewhat similar, but for the sake of the completeness, we will give it here.

We continue the proof of Proposition~\ref{P:auxiliary_formulas}, first identity. Again Lie derivatives are handled as follows. In principle all derivatives are ordinary, except for the hermitian metric on the vector bundle $\cE$. Since ultimately integration along fibers will be needed, and exterior derivatives and their adjoints are taken with respect to the \ka\ structure on the fibers (and hermitian metric on $\cE$).
As before, because of the symmetry of Christoffel symbols taken on fibers, in such a context by definition a covariant derivative with respect to a parameter $s$ is a covariant derivative only with respect to $(\cE,h)$, and a covariant derivative with respect to a fiber coordinate is a covariant derivative with respect to both $h$ and $\omega_{\cX_s}$ (See also the discussion in Section~\ref{S:pre}.)

By using Equation~\eqref{eq:aux3}, furthermore by the $\dbar^*$-closedness of $\chi$, and
$\pt_sg^{\ol\beta\gamma}=g^{\ol\beta\sigma}a\ind{s}{\gamma}{;\sigma}$,
we get
\begin{equation}\label{E:commutation_sgamma}
	g^{\ol\beta\gamma}
	\chi\ind{}{i}{A_p,\ol\beta,\ol B_{q-1};s\gamma}
	=
	-
	\chi\ind{}{i}{A_p,\ol\beta,\ol B_{q-1};\gamma}
	g^{\ol\beta\sigma}
	a\ind{s}{\gamma}{;\sigma}
	-
	\sum_{\mu=1}^p
	g^{\ol\beta\gamma}
	a\ind{s}{\alpha}{;\alpha_\mu\gamma}
	\chi\ind{}{i}{
		{\tiny\vtop{
		\hbox{$\alpha_1,\ldots,\alpha,\ldots,\alpha_p\ol B_q$}
		\vskip-.8mm
		\hbox{$\phantom{\alpha_1,\ldots,}{|\atop\mu} $}}}
		}
\end{equation}
On the other hand, taking $\dbar^*$ on \eqref{eq:lvprime}, one can easily get
\begin{equation*}
\dbar^*(L_v'\chi)_{A_p,\ol B_{q-1}}
=
(-1)^{p+1}
g^{\ol\beta_1\gamma}
\paren{
	\chi_{A_p,\ol B_q;s\gamma}
	+
	a\ind{s}{\alpha}{;\gamma}
	\chi\ind{}{i}{A_p,\ol B_q;\alpha}
	+
	\sum_{\mu=1}^p
	a\ind{s}{\alpha}{;\alpha_\mu\gamma}
	\chi\ind{}{i}{
		{\tiny\vtop{
		\hbox{$\alpha_1,\ldots,\alpha,\ldots,\alpha_p,\ol B_q$}
		\vskip-.8mm
		\hbox{$\phantom{\alpha_1,\ldots,}{|\atop\mu} $}}}
		}
}.
\end{equation*}
This term vanishes by \eqref{E:commutation_sgamma}.
The proof is complete.
\end{proof}

\section{Curvature formula for $R^qf_*\Omega^p_{\cX/S}(\cE)$}\label{S:genfo}

In this section, we discuss the general curvature formula for $R^qf_*\Omega^p_{\cX/S}(\cE)$, invoking the formulas obtained in the previous sections.
\medskip

As we assume that $R^qf_*\Omega^p_{\cX/S}(\cE)$ is locally free, there is a local holomorphic frame $\set{[\psi^{(k)}]}_{k=1}^N$ of $R^qf_*\Omega^p_{\cX/S}(\cE)$, where $\psi^{(k)}$ is given by Lemma~\ref{L:representative}.
Then the hermitian metric $\inner{\cdot,\cdot}^H$ on $R^qf_*\Omega^p_{X/S}(\cE)$ is given as
\begin{equation*}
H^{k\ol l}:=\inner{[\psi^{(k)}],[\psi^{(l)}]}^H
=
\inner{\psi^{(k)},\psi^{(l)}}.
\end{equation*}
Let a point $s_0\in S$ be fixed.
By a unitary change of $\{\psi^{(k)}\}$, we can take a normal coordinate at $s_0$ such that $\pd{}{s}\vert_{s=s_0}H^{k\ol l}=0$, which means that
\begin{equation*}
\pd{}{s}\Big\vert_{s=s_0}\inner{\psi^{(k)},\psi^{(l)}}
=
\inner{\paren{L_v'\psi^{(k)}},\psi^{(l)}}\Big\vert_{s=s_0}
=0
\end{equation*}
for all $k,l=1,\ldots,N$.
This implies that the harmonic part of $\paren{L_v'\psi^{(k)}}$ vanishes, i.e.,
\begin{equation}\label{E:orthogonal_to_harmonic_space}
H\paren{L_v'\psi^{(k)}}=0
\;\;\;\text{for all }k.
\end{equation}
Then the curvature tensor $R\ind{i\ol\jmath}{k\ol l}{}(s_0)$ at $s=s_0$ is given as follows.
\begin{align*}
R\paren{\pt_i,\pt_{\ol\jmath},\psi^{(k)},\ol{\psi^{(l)}}}(s_0)
=
R\ind{i\ol\jmath}{k\ol l}{}(s_0)
=
\left.-\pd{^2}{s^i\pt s^{\ol\jmath}}\right\vert_{s=s_0}H^{k\ol l}
=
\left.-\pd{^2}{s^i\pt s^{\ol\jmath}}\right\vert_{s=s_0}
\inner{\psi^{(k)},\psi^{(l)}}.
\end{align*}
Then it follows from the computations in Subsection~\ref{SS:Lie_derivative} that
\begin{equation}\label{E:curvature_formula}
\begin{aligned}
R\paren{\pt_i,\pt_{\ol\jmath},\psi^{(k)},\ol{\psi^{(l)}}}(s_0)
&=
\inner{L_{[v,\ol v]}\psi^{(k)},\psi^{(l)}}
+
\inner{\Theta_E(v_i,\ol{v_j})\psi^{(k)},\psi^{(l)}}
-
\inner{(L_{v_i}'\psi^{(k)}),(L_{v_j}'\psi^{(l)})}
\\
&\hspace{-.5cm}
+
\inner{A_i\cup\psi^{(k)},A_j\cup\psi^{(l)}}
+
\inner{(L_{\ol{v_j}}'\psi^{(k)}),(L_{\ol{v_i}}'\psi^{(l)})}
-
\inner{A_{\ol\jmath}\cup\psi^{(k)},A_{\ol i}\cup\psi^{(l)}}
\end{aligned}
\end{equation}
A similar formula is proved in \cite{Berndtsson_Paun_Wang2022,Wang2016Ar}.
\begin{remark}
By Proposition~\ref{P:dbar_Lv}, Proposition~\ref{P:auxiliary_formulas} and Equation \eqref{E:orthogonal_to_harmonic_space}, the Lie derivatives $(L_{v_i}'\psi^{(k)})$ define the minimal solution of the system
\begin{equation*}
	\dbar u
	=
	A_i\cup(\pt\psi^{(k)})
	+
	\pt\paren{A_i\cup\psi^{(k)}}
	+
	\eta_i
	\we\psi^{(k)}.
\end{equation*}
Likewise, Proposition~\ref{P:dbar*_Lvbar} reads that $(L_{\ol{v_j}}'\psi^{(l)})$ is a solution of
\begin{equation*}
	\dbar^* u
	=
	(-1)^p\pt^*(A_{\ol\jmath}\cup\psi^{(l)})
	+
	(-1)^pA_{\ol\jmath}\cup\pt^*\psi^{(l)}
	+
	[\Lambda,\ii\eta_{\ol\jmath}]\psi^{(l)}.
\end{equation*}
Moreover, Proposition~\ref{P:Lie_derivative} states that $(L_{\ol{v_j}}'\psi^{(l)})$ is orthogonal to the harmonic space, so $(L_{\ol{v_j}}'\psi^{(l)})$ is the minimal solution by Proposition~\ref{P:auxiliary_formulas}.
\end{remark}

On the other hand, Proposition~\ref{P:auxiliary_formulas} and Proposition~\ref{P:dbar_Lv} yield that
\begin{align*}
	\inner{(L_{v_i}'\psi^{(k)}),(L_{v_j}'\psi^{(l)})}
	&=
	\inner{G_{\dbar}\Box_{\dbar}(L_{v_i}'\psi^{(k)}),L_{v_i}'\psi^{(l)}}
	\\
	&=
	\inner{G_{\dbar}\dbar(L_{v_i}'\psi^{(k)}),\dbar(L_{v_i}'\psi^{(l)})}
	\\
	&=
	\inner{G_{\dbar}(w\ind{i}{k}{}),w\ind{j}{l}{}},
\end{align*}
where
\begin{equation*}
	w\ind{i}{k}{}
	=
	A_i\cup\pt\psi^{(k)}
	+
	\pt(A_i\cup\psi^{(k)})
	+
	\eta_i
	\we\psi^{(k)}.
\end{equation*}
The same argument using Proposition~\ref{P:dbar*_Lvbar} instead gives that
\begin{equation*}
	\inner{(L_{\ol{v_i}}'\psi^{(k)}),(L_{\ol{v_j}}'\psi^{(l)})}
	=
	\inner{G_\dbar(w\ind{\ol i}{k}{}),w\ind{\ol\jmath}{l}{}},
\end{equation*}
where
\begin{equation*}
	w\ind{\ol\jmath}{l}{}
	=
	(-1)^p\pt^*(A_{\ol\jmath}\cup\psi^{(l)})
	+
	(-1)^pA_{\ol\jmath}\cup\pt^*\psi^{(l)}
	+
	[\Lambda,\ii\eta_{\ol\jmath}]\psi^{(l)}.
\end{equation*}
Altogether, we have the following formula for $R^qf_*\Omega^p_{\cX/S}(\cE)$.
\begin{align*}
	R\paren{\pt_i,\pt_{\ol\jmath},\psi^{(k)},\ol{\psi^{(l)}}}(s_0)
	&=
	\inner{L_{[v_i,\ol{v_j}]}\psi^{(k)},\psi^{(l)}}
	+
	\inner{\Theta(v_i,\ol{v_j})\psi^{(k)},\psi^{(l)}}
	-
	\inner{G_\dbar(w\ind{i}{k}{}),w\ind{j}{l}{}}
	\\
	&\hspace{.5cm}
	+
	\inner{A_i\cup\psi^{(k)},A_j\cup\psi^{(l)}}
	+
	\inner{G_\dbar(w\ind{\ol\jmath}{k}{}),w\ind{\ol i}{l}{}}
	-
	\inner{A_{\ol\jmath}\cup\psi^{(k)},A_{\ol i}\cup\psi^{(l)}}
\end{align*}
where $\pt_i,\pt_j\in T_{s_0}S$ and $\psi^{(k)}, \psi^{(l)}$ harmonic representatives of elements from $H^q(\cX_{s_0},\Omega_{\cX_{s_0}}^p(\cE))$, which gives Main Theorem 1 together with Proposition~\ref{P:Lie_derivative_commutator}.

\section{Additional formulas}
In this section we will prove several formulas for further curvature computation.
As in previous section, we always assume that a local holomorphic section $\chi$ given by Lemma~\ref{L:representative} of $R^qf_*\Omega^p_{\cX/S}(\cE)$.
We again assume that $\dim S=1$ for simplicity.

We first consider the first term $\inner{L_{[v,\ol v]}\chi,\psi}$ in \eqref{E:curvature_formula}.
\begin{lemma}\label{L:commutator_vbar_v}
We have the following.
\begin{equation*}
	[v,\ol v]
	=
	c(\omega)^{;\alpha}\partial_\alpha
	-
	c(\omega)^{;\ol\beta}\partial_{\ol\beta}
\end{equation*}
\end{lemma}

\begin{proof}
	By the definition of the Lie bracket, we have
	\begin{align*}
		[v,\ol v]
		&=
		[\pt_{s}+a\ind{s}{\alpha}{}\pt_\alpha,
		\pt_\ol s+a\ind{\ol s}{\ol\beta}{}\pt_\ol\beta] \\
		&=
		-
		\left(\pt_\ol s (a\ind{s}{\alpha}{})
		+
		a\ind{\ol s}{\ol\beta}{}a\ind{s}{\alpha}{|\ol\beta}\right)\pt_\alpha
		+
		\left(\pt_s(a\ind{\ol s}{\ol\beta}{})
		+
		a\ind{s}{\alpha}{}a\ind{\ol s}{\ol\beta}{|\alpha}\right)\pt_\ol\beta.
	\end{align*}
	Now
	\begin{align*}
		\pt_\ol s (a\ind{s}{\alpha}{})
		&=
		-\pt_\ol s (g^{\ol\beta\alpha}g_{s\ol\beta})
		=
		g^{\ol\beta\sigma} g_{\sigma\ol s| \ol \tau}g^{\ol\tau
		\alpha}g_{s\ol\beta} - g^{\ol\beta\alpha}g_{s\ol \beta|\ol s} \\
		&=
		g^{\ol\beta\sigma}a_{\ol s
		\sigma;\ol\tau}g^{\ol\tau\alpha}a_{s\ol\beta} - g^{\ol\beta\alpha} g_{s\ol
		s; \ol\beta}.
		\end{align*}
	Hence we have
	\begin{equation*}
		-\pt_\ol s (a\ind{s}{\alpha}{})
		-
		a\ind{\ol s}{\ol\beta}{} a\ind{s}{\alpha}{|\ol\beta}
		=
		c(\omega)^{;\alpha}.
	\end{equation*}
	Similarly, one can show that the coefficient of $\pt_\ol \beta$ that $-c(\omega)^{;\ol\beta}$.
\end{proof}

\begin{proposition}\label{P:Lie_derivative_commutator}
	The following equation holds.
	\begin{equation}\label{E:Lie_derivative_commutator}
		\inner{L_{[v,\ol v]} \chi,\psi}
		=
		\inner{c(\omega)\,\Box_\pt\chi,\psi}
		-
		\inner{c(\omega)\,\partial\chi,\partial\psi}
		-
		\inner{c(\omega)\,\pt^*\chi,\pt^*\psi}
		+
		\inner{\chi,\pt \ol{c(\omega)}\wedge\pt^*\psi}
		+
		\inner{\pt c(\omega)\wedge\pt^*\chi, \psi}.
	\end{equation}
\end{proposition}
	
\begin{proof}
	First, it follows from the definition of Lie derivatives that
	\begin{eqnarray*}
		\paren{
			[c(\omega)^{;\ol\beta}\pt_{\ol\beta},\chi]'
		}\ind{}{i}{A_p,\ol B_q}
		& = &
		c(\omega)^{;\ol\beta}\chi\ind{}{i}{A_p,\ol B_q;\ol\beta}
		+
		\sum_{\nu=1}^p c(\omega)\ind{}{;\ol\beta}{;\ol\beta_\nu }
		\chi\ind{}{i}{{\tiny\vtop{
		\hbox{$\alpha_1,\ldots,\alpha_p,\ol\beta_1,\ldots,\ol\beta,\ldots,\ol\beta_q$}\vskip-.8mm
		\hbox{$\phantom{\alpha_1,\ldots,\alpha_p,\ol\beta_1,\ldots,}{|\atop\nu} $}}}}\\
		&=&
		c(\omega)^{;\ol\beta}
		\paren{\dbar\chi}\ind{}{i}{\alpha_1,\ldots,\alpha_p,\ol\beta,\ol\beta_1,\ldots,\ol\beta_q}
		+
		\sum_{\kappa =1}^p \big( c(\omega) ^{;\ol\beta}
		\chi\ind{}{i}{ {\tiny\vtop{
		\hbox{$\alpha_1,\ldots,\alpha_p,\ol\beta_1,\ldots,\ol\beta,\ldots,\ol\beta_q$}\vskip-.8mm
		\hbox{$\phantom{\alpha_1,\ldots,\alpha_p,\ol\beta_1,\ldots,}{|\atop\nu} $}}}}\big)_{;\ol\beta_\nu}
		\\
		&=&
		\paren{
			\paren{c(\omega)^{;\ol\beta}\partial_{\ol\beta}\cup\dbar\chi}
			+
			\dbar\big(c(\omega)^{;\ol\beta}\pt_{\ol\beta}\cup \chi\big)
		}\ind{}{i}{A_p,\ol B_q}.
	\end{eqnarray*}
	Since $\chi$ is $\dbar$-harmonic on fibers, the first term vanishes.
	It also implies that
	\begin{equation*}
		\inner{\dbar\big( c(\omega) ^{;\ol\beta}\pt_{\ol\beta}\cup \chi\big),\psi}
		=
		\inner{c(\omega) ^{;\ol\beta}\pt_{\ol\beta}\cup \chi,\dbar^*\psi}
		=0.
	\end{equation*}

	On the other hand, it follows from \eqref{E:partial} that
	\begin{eqnarray*}
		\paren{
			[c(\omega)^{;\alpha}\pt_\alpha,\chi]'
		}\ind{}{i}{A_p,\ol B_q}
		& = &
		c(\omega)^{;\alpha}\chi_{A_p,\ol B_q;\alpha}
		+
		\sum_{\mu=1}^p c(\omega)\ind{}{;\alpha}{;\alpha_\mu}
		\chi\ind{}{i}{ {\tiny\vtop{
		\hbox{$\alpha_1,\ldots,\alpha,\ldots,\alpha_p,\ol B_q$}\vskip-.8mm
		\hbox{$\phantom{\alpha_1,\ldots,}{|\atop\mu} $}}}}\\
		&=&
		c(\omega)^{;\alpha}
		\paren{\partial\chi}\ind{}{i}{\alpha,\alpha_1,\cdots,\alpha_p,\ol B_q}
		+
		\sum_{\mu=1}^p \big( c(\omega) ^{;\alpha}
		\chi\ind{}{i}{ {\tiny\vtop{
		\hbox{$\alpha_1,\ldots,\alpha,\ldots,\alpha_p,\ol B_q$}\vskip-.8mm
		\hbox{$\phantom{\alpha_1,\ldots,}{|\atop\mu} $}}}}\big)_{;\alpha_\mu}
		\\
		&=&
		\paren{
			\paren{c(\omega)^{;\alpha}\partial_\alpha\cup\partial\chi}
			+
			\pt\big(c(\omega)^{;\alpha}\pt _\alpha\cup \chi\big)
		}_{A_p,\ol B_q}.
	\end{eqnarray*}
	Note that $c(\omega)^{;\alpha}\pt_\alpha=\paren{\dbar c(\omega)}^\omega$, which is the dual element of $\dbar c(\omega)$ with respect to $\omega$, and it is well known that the adjoint operator of $\pt\ol{c(\omega)}\wedge\cdot$ is $\paren{\dbar c(\omega)}^\omega\cup\cdot$.
	Hence the second term gives $\inner{\paren{\dbar c(\omega)}^\omega\cup\chi,\pt^*\psi}=\inner{\chi,\pt \ol{c(\omega)}\wedge\pt^*\psi}$ in the conclusion.
	Since
	\begin{equation*}
		\paren{
			c(\omega)^{;\alpha}\partial_\alpha\cup\partial\chi
		}_{A_p,\ol B_q}
		=
		c(\omega)^{;\alpha}
		\paren{
		\chi\ind{}{i}{\alpha_1,\ldots,\alpha_p,\ol B_q;\alpha}
		-
		\sum_{\mu=1}^p
		\chi\ind{}{i}{
		{\tiny\vtop{
		\hbox{$\alpha_1,\ldots,\alpha,\ldots,\alpha_p,\ol B_q;\alpha_\mu$}\vskip-.8mm
		\hbox{$\phantom{\alpha_1,\ldots,}{|\atop\mu} $}}}}
		},
	\end{equation*}
	the first term is computed as (setting $X=\cX_s$)
	\begin{align*}
		\inner{c(\omega)^{;\alpha}\pt_\alpha\cup\pt\chi,\psi}
		&=
		\frac{1}{(p!)^2(q!)^2}
		\int_X
		g^{\ol\beta\alpha}c(\omega)_{;\ol\beta}
		\paren{\partial\chi}\ind{}{i}{\alpha,\alpha_1,\ldots,\alpha_p,\ol B_q}\cdot
		\ol{\psi\ind{}{j}{C_p,\ol D_q}}
		\cdot h_{i\ol\jmath}\cdot g^{A_p\ol C_p}\cdot g^{\ol B_qD_q}
		dV_\omega \\
		&=
		-
		\frac{1}{(p!)^2(q!)^2}
		\int_X
		g^{\ol\beta\alpha}c(\omega)
		\paren{\partial\chi}\ind{}{i}{\alpha,\alpha_1,\ldots,\alpha_p,\ol B_q;\ol\beta}\cdot
		\ol{\psi\ind{}{j}{C_p,\ol D_q}}
		\cdot h_{i\ol\jmath}
		\cdot g^{A_p\ol C_p}\cdot g^{\ol B_qD_q}
		dV_\omega \\
		&\hspace{.5cm}
		-
		\frac{1}{(p!)^2(q!)^2}
		\int_X
		g^{\ol\beta\alpha}c(\omega)
		\paren{\partial\chi}\ind{}{i}{\alpha,\alpha_1,\ldots,\alpha_p,\ol B_q}\cdot
		\ol{\psi\ind{}{j}{C_p,\ol D_q;\beta}}
		\cdot h_{i\ol\jmath}\cdot g^{A_p\ol C_p}\cdot g^{\ol B_qD_q}
		dV_\omega \\
		&:=
		I_1+I_2,
	\end{align*}
	
	First we compute $I_2$.
		\begin{align*}
		I_2
		&=
		-\frac{1}{(p!)^2(q!)^2}
		\int_X
		c(\omega)
		\paren{\partial\chi}\ind{}{i}{\alpha,\alpha_1,\ldots,\alpha_p,\ol B_q}\cdot
		\ol{\psi\ind{}{j}{C_p,\ol D_q;\gamma}}
		\cdot h_{i\ol\jmath}\cdot g^{\ol\gamma\alpha}g^{A_p\ol C_p}\cdot g^{\ol B_qD_q}
		dV_\omega \\
		&=
		-\frac{1}{((p+1)!)^2(q!)^2}
		\int_X
		c(\omega)
		\paren{\partial\chi}\ind{}{i}{\alpha_0,\alpha_1,\ldots,\alpha_p,\ol B_q}\cdot
		\ol{\paren{\partial\psi}\ind{}{j}{\gamma_0,\gamma_1,\ldots\gamma_p,\ol D_q}}
		\cdot h_{i\ol\jmath}\cdot g^{A_{p+1}\ol C_{p+1}}\cdot g^{\ol B_qD_q}
		dV_\omega \\
		&=
		-\inner{c(\omega)\,  \partial\chi,\partial\psi}.
	\end{align*}
	
	Next we compute $I_1$.
	It follows from \eqref{E:partial} that
	\begin{equation*}
		\paren{\partial\chi}\ind{}{i}{\alpha,\alpha_1,\ldots,\alpha_p,\ol B_q;\ol\beta}
		=
		\chi\ind{}{i}{\alpha_1,\ldots,\alpha_p,\ol B_q;\alpha\ol\beta}
		-
		\sum_{\mu=1}^p
		\chi\ind{}{i}{
		{\tiny\vtop{
		\hbox{$\alpha_1,\ldots,\alpha,\ldots,\alpha_p,\ol B_q;\alpha_\mu\ol\beta$}\vskip-.8mm
		\hbox{$\phantom{\alpha_1,\ldots,}{|\atop\mu} $}}}},
	\end{equation*}
	which implies that
	\begin{align*}
		I_1
		&=
		-\frac{1}{(p!)^2(q!)^2}
		\int_X
		g^{\ol\beta\alpha}c(\omega)
		\paren{\partial\chi}\ind{}{i}{\alpha,\alpha_1,\ldots,\alpha_p,\ol B_q;\ol\beta}
		\ol{\psi\ind{}{j}{C_p,\ol D_q}}
		\cdot h_{i\ol\jmath}
		\cdot g^{A_p\ol C_p}\cdot g^{\ol B_qD_q}
		dV_\omega
		\\
		&=
		-\frac{1}{(p!)^2(q!)^2}
		\int_X
		g^{\ol\beta\alpha}c(\omega)
		\paren{\chi\ind{}{i}{\alpha_1,\ldots,\alpha_p,\ol B_q;\alpha\ol\beta}
		-
		\sum_{\mu=1}^p
		\chi\ind{}{i}{
		{\tiny\vtop{
		\hbox{$\alpha_1,\ldots,\alpha,\ldots,\alpha_p,\ol B_q;\alpha_\mu\ol\beta$}\vskip-.8mm
		\hbox{$\phantom{\alpha_1,\ldots,}{|\atop\mu} $}}}}}
		\ol{\psi\ind{}{j}{C_p,\ol D_q}}
		\cdot h_{i\ol\jmath}
		\cdot g^{A_p\ol C_p}\cdot g^{\ol B_qD_q}
		dV_\omega \\
		&=:
		\frac{1}{(p!)^2(q!)^2}
		\int_X
		c(\omega)\paren{I_{11}+I_{12}}
		\ol{\psi\ind{}{j}{C_p,\ol D_q}}
		\cdot h_{i\ol\jmath}
		\cdot g^{A_p\ol C_p}\cdot g^{\ol B_qD_q}
		dV_\omega.
	\end{align*}
	Then invoking Proposition~\ref{P:commutator_covariant_derivatives}, one has
	\begin{align*}
		I_{11}
		&=
		-g^{\ol\beta\alpha}
		\paren{
			\chi\ind{}{i}{A_p,\ol B_q;\ol\beta\alpha}
		-
		\Theta\ind{j}{i}{\alpha\ol\beta}\chi\ind{}{j}{A_p,\ol B_q}
		+
		\sum_{\mu=1}^p
		R\ind{\alpha_\mu}{\gamma}{\alpha\ol\beta}
		\chi\ind{}{i}{
		{\tiny\vtop{
		\hbox{$\alpha_1,\ldots,\gamma,\ldots,\alpha_p,\ol B_q$}\vskip-.8mm
		\hbox{$\phantom{\alpha_1,\ldots,}{|\atop\mu } $}}}}
		+
		\sum_{\nu=1}^q
		R\ind{\ol\beta_\nu}{\ol\delta}{\alpha\ol\beta}
		\chi\ind{}{i}{
		{\tiny\vtop{
		\hbox{$A_p,\ol\beta_1,\ldots,\ol\delta,\ldots,\ol\beta_q$}\vskip-.8mm
		\hbox{$\phantom{A_p,\ol\beta_1,\ldots,}{|\atop\nu} $}}}}
		}
		\\
		&=
		-g^{\ol\beta\alpha}
		\chi\ind{}{i}{A_p,\ol B_q;\ol\beta\alpha}
		+
		\Theta\ind{k}{i}{}
		\chi\ind{}{k}{A_p,\ol B_q}
		-
		\sum_{\mu=1}^p
		R\ind{\alpha_\mu}{\gamma}{}
		\chi\ind{}{i}{
		{\tiny\vtop{
		\hbox{$\alpha_1,\ldots,\gamma,\ldots,\alpha_p,\ol B_q$}\vskip-.8mm
		\hbox{$\phantom{\alpha_1,\ldots,}{|\atop\mu } $}}}}
		-
		\sum_{\nu=1}^q
		R\ind{\ol\beta_\nu}{\ol\delta}{}
		\chi\ind{}{i}{
		{\tiny\vtop{
		\hbox{$A_p,\ol\beta_1,\ldots,\ol\delta,\ldots,\ol\beta_q$}\vskip-.8mm
		\hbox{$\phantom{A_p,\ol\beta_1,\ldots,}{|\atop\nu} $}}}}
	\end{align*}
	where $\Theta\ind{k}{i}{}=\Theta\ind{k}{i}{\alpha\ol\beta}g^{\ol\beta\alpha}$ is the mean curvature.
	In view of the following equality obtained by Proposition~\ref{P:commutator_covariant_derivatives},
	\begin{align*}
		\chi\ind{}{i}{
		{\tiny\vtop{
		\hbox{$\alpha_1,\ldots,\alpha,\ldots,\alpha_p,\ol B_q;\alpha_\mu\ol\beta$}\vskip-.8mm
		\hbox{$\phantom{\alpha_1,\ldots,}{|\atop\mu} $}}}}
		&=
		\chi\ind{}{i}{
		{\tiny\vtop{
		\hbox{$\alpha_1,\ldots,\alpha,\ldots,\alpha_p,\ol B_q;\ol\beta\alpha_\mu$}\vskip-.8mm
		\hbox{$\phantom{\alpha_1,\ldots,}{|\atop\mu} $}}}}
		-
		\Theta\ind{j}{i}{\alpha_\mu\ol\beta}
		\chi\ind{}{j}{
		{\tiny\vtop{
		\hbox{$\alpha_1,\ldots,\alpha,\ldots,\alpha_p,\ol B_q$}\vskip-.8mm
		\hbox{$\phantom{\alpha_1,\ldots,}{|\atop\mu} $}}}}
		\\
		&\hspace{.5cm}+
		\sum_{\tau\neq\mu }
		R\ind{\alpha_\tau}{\gamma}{\alpha_\mu\ol\beta}
		\chi\ind{}{i}{
		{\tiny\vtop{
		\hbox{$\alpha_1,\ldots,\gamma,\ldots,\alpha,\ldots,\alpha_p,\ol B_q$}\vskip-.8mm
		\hbox{$\phantom{\alpha_1,\ldots,}{|\atop\tau}\phantom{,\ldots,}{|\atop\mu} $}}}}
		+
		R\ind{\alpha}{\gamma}{\alpha_\mu \ol\beta}
		\chi\ind{}{i}{
		{\tiny\vtop{
		\hbox{$\alpha_1,\ldots,\gamma,\ldots,\alpha_p,\ol B_q$}\vskip-.8mm
		\hbox{$\phantom{\alpha_1,\ldots,}{|\atop\mu } $}}}}
		\\
		&\hspace{.5cm}+
		\sum_{\nu=1}^q
		R\ind{\ol\beta_\nu}{\ol\delta}{\alpha_\mu\ol\beta}
		\chi\ind{}{j}{
		{\tiny\vtop{
		\hbox{$\alpha_1,\ldots,\alpha,\ldots,\alpha_p,\ol\beta_1,\ldots,\ol\delta,\ldots,\ol\beta_q$}\vskip-.8mm
		\hbox{$\phantom{\alpha_1,\ldots,}{|\atop\mu}\phantom{,\ldots,\alpha_p,\ol\beta_1,\ldots,}{|\atop\nu} $}}}},
	\end{align*}
	one also has
	\begin{align*}
		I_{12}
		&=
		g^{\ol\beta\alpha}
		\sum_{\mu=1}^p
		\chi\ind{}{i}{
		{\tiny\vtop{
		\hbox{$\alpha_1,\ldots,\alpha,\ldots,\alpha_p,\ol B_q;\alpha_\mu\ol\beta$}\vskip-.8mm
		\hbox{$\phantom{\alpha_1,\ldots,}{|\atop\mu} $}}}}
		\\
		&=
		g^{\ol\beta\alpha}
		\sum_{\mu=1}^p
		\Bigg(
		\chi\ind{}{i}{
		{\tiny\vtop{
		\hbox{$\alpha_1,\ldots,\alpha,\ldots,\alpha_p,\ol B_q;\ol\beta\alpha_\mu$}\vskip-.8mm
		\hbox{$\phantom{\alpha_1,\ldots,}{|\atop\mu} $}}}}
		-
		\Theta\ind{j}{i}{\alpha_\mu\ol\beta}
		\chi\ind{}{j}{
		{\tiny\vtop{
		\hbox{$\alpha_1,\ldots,\alpha,\ldots,\alpha_p,\ol B_q$}\vskip-.8mm
		\hbox{$\phantom{\alpha_1,\ldots,}{|\atop\mu} $}}}}
		\\
		&\hspace{2.5cm}
		+
		\sum_{\tau\neq\mu}
		R\ind{\alpha_\tau}{\gamma}{\alpha_\mu\ol\beta}
		\chi\ind{}{i}{
		{\tiny\vtop{
		\hbox{$\alpha_1,\ldots,\gamma,\ldots,\alpha,\ldots,\alpha_p,\ol B_q$}\vskip-.8mm
		\hbox{$\phantom{\alpha_1,\ldots,}{|\atop\tau}\phantom{,\ldots,}{|\atop\mu} $}}}}
		+
		R\ind{\alpha}{\gamma}{\alpha_\mu\ol\beta}
		\chi\ind{}{i}{
		{\tiny\vtop{
		\hbox{$\alpha_1,\ldots,\gamma,\ldots,\alpha_p,\ol B_q$}\vskip-.8mm
		\hbox{$\phantom{\alpha_1,\ldots,}{|\atop\mu} $}}}}
		\\
		&\hspace{2.5cm}
		+
		\sum_{\nu=1}^q
		R\ind{\ol\beta_\nu}{\ol\delta}{\alpha_\mu\ol\beta}
		\chi\ind{}{j}{
		{\tiny\vtop{
		\hbox{$\alpha_1,\ldots,\alpha,\ldots,\alpha_p,\ol\beta_1,\ldots,\ol\delta,\ldots,\ol\beta_q$}\vskip-.8mm
		\hbox{$\phantom{\alpha_1,\ldots,}{|\atop\mu}\phantom{,\ldots,\alpha_p,\ol\beta_1,\ldots,}{|\atop\nu} $}}}}
		\Bigg)
		\\
		&=
		\sum_{\mu=1}^p
		\Bigg(
		g^{\ol\beta\alpha}
		\chi\ind{}{i}{
		{\tiny\vtop{
		\hbox{$\alpha_1,\ldots,\alpha,\ldots,\alpha_p,\ol B_q;\ol\beta\alpha_\mu$}\vskip-.8mm
		\hbox{$\phantom{\alpha_1,\ldots,}{|\atop\mu} $}}}}
		-
		\Theta\indd{j}{i}{\alpha_\mu}{\alpha}
		\chi\ind{}{j}{
		{\tiny\vtop{
		\hbox{$\alpha_1,\ldots,\alpha,\ldots,\alpha_p,\ol B_q$}\vskip-.8mm
		\hbox{$\phantom{\alpha_1,\ldots,}{|\atop\mu} $}}}}
		\\
		&\hspace{2.5cm}
		+
		R\ind{\alpha_\mu}{\gamma}{}
		\chi\ind{}{i}{
		{\tiny\vtop{
		\hbox{$\alpha_1,\ldots,\gamma,\ldots,\alpha_p,\ol B_q$}\vskip-.8mm
		\hbox{$\phantom{\alpha_1,\ldots,}{|\atop\mu} $}}}}
		+
		\sum_{\nu=1}^q
		R\indd{\ol\beta_\nu}{\ol\delta}{\alpha_\mu}{\alpha}
		\chi\ind{}{i}{
		{\tiny\vtop{
		\hbox{$\alpha_1,\ldots,\alpha,\ldots,\alpha_p,\ol\beta_1,\ldots,\ol\delta,\ldots,\ol\beta_q$}\vskip-.8mm
		\hbox{$\phantom{\alpha_1,\ldots,}{|\atop\mu}\phantom{,\ldots,\alpha_p,\ol\beta_1,\ldots,}{|\atop\nu} $}}}}
		\Bigg)
	\end{align*}
	The third equality follows from
	\begin{equation*}
		\sum_{\mu=1}^p
		\sum_{\tau\neq\mu}
		R\indd{\alpha_\tau}{\gamma}{\alpha_\mu}{\alpha}
		\chi\ind{}{i}{
				{\tiny\vtop{
				\hbox{$\alpha_1,\ldots,\gamma,\ldots,\alpha,\ldots\alpha_p,\ol B_q$}
				\vskip-.8mm
				\hbox{$\phantom{\alpha_1,\ldots,}
				{|\atop\tau}\phantom{,\ldots,}{|\atop\mu} $}}}}
		=0,
	\end{equation*}
	which results from the fact that $R\indd{\alpha_\tau}{\gamma}{\alpha_\mu}{\alpha}$ is symmetric on $\alpha$ and $\gamma$, while $\chi\ind{}{i}{
		{\tiny\vtop{
		\hbox{$\alpha_1,\ldots,\gamma,\ldots,\alpha,\ldots\alpha_p,\ol B_q$}
		\vskip-.8mm
		\hbox{$\phantom{\alpha_1,\ldots,}
		{|\atop\tau}\phantom{,\ldots,}{|\atop\mu} $}}}}$ is antisymmetric in $\alpha$ and $\gamma$.
	Hence it follows that
	\begin{align*}
		I_{11}+I_{12}
		&=
		-
		g^{\ol\beta\alpha}
		\chi\ind{}{i}{A_p,\ol B_q;\ol\beta\alpha}
		+
		\Theta\ind{k}{i}{}
		\chi\ind{}{k}{A_p,\ol B_q}
		-
		\sum_{\nu=1}^q
		R\ind{\ol\beta_\nu}{\ol\delta}{}
		\chi\ind{}{i}{
		{\tiny\vtop{
		\hbox{$A_p,\ol\beta_1,\ldots,\ol\delta,\ldots,\ol\beta_q$}\vskip-.8mm
		\hbox{$\phantom{A_p,\ol\beta_1,\ldots,}{|\atop\nu} $}}}}
		\\
		&\hspace{1cm}
		+
		\sum_{\mu=1}^p
		\Bigg(
		g^{\ol\beta\alpha}
		\chi\ind{}{i}{
		{\tiny\vtop{
		\hbox{$\alpha_1,\ldots,\alpha,\ldots,\alpha_p,\ol B_q;\ol\beta\alpha_\mu$}\vskip-.8mm
		\hbox{$\phantom{\alpha_1,\ldots,}{|\atop\mu} $}}}}
		-
		\Theta\indd{j}{i}{\alpha_\mu}{\alpha}
		\chi\ind{}{j}{
		{\tiny\vtop{
		\hbox{$\alpha_1,\ldots,\alpha,\ldots,\alpha_p,\ol B_q$}\vskip-.8mm
		\hbox{$\phantom{\alpha_1,\ldots,}{|\atop\mu} $}}}}
		\\
		&\hspace{4cm}
		+
		\sum_{\nu=1}^q
		R\indd{\ol\beta_\nu}{\ol\delta}{\alpha_\mu}{\alpha}
		\chi\ind{}{j}{
		{\tiny\vtop{
		\hbox{$\alpha_1,\ldots,\alpha,\ldots,\alpha_p,\ol\beta_1,\ldots,\ol\delta,\ldots,\ol\beta_q$}
		\vskip-.8mm
		\hbox{$\phantom{\alpha_1,\ldots,}{|\atop\mu}
		\phantom{,\ldots,\alpha_p,\ol\beta_1,\ldots,}{|\atop\nu} $}}}}
		\Bigg)
	\end{align*}
	Since $\chi$ is $\dbar$-closed and $\dbar^*$-closed on fibers, we have
	\begin{eqnarray*}
		-g^{\ol\beta\alpha}
		\chi\ind{}{i}{A_p,\ol B_q;\ol\beta\alpha}
		&=&
		-g^{\ol\beta\alpha}
		\sum_{\nu=1}^q
		\paren{
			\chi_{\tiny{\vtop{\hbox{$A_p,\ol\beta_1,\ldots,\ol\beta,\ldots,\ol\beta_q;\ol\beta_\nu\alpha$}
			\vskip-.8mm
			\hbox{$\phantom{A_p,\ol\beta_1,\ldots,}{|\atop\nu} $}}}}
		}
		\\
		&=&
		-g^{\ol\beta\alpha}
		\sum_{\nu=1}^q
		\Bigg(
			\chi_{\tiny{\vtop{\hbox{$A_p,\ol\beta_1,\ldots,\ol\beta,\ldots,\ol\beta_q;\alpha\ol\beta_\nu$}
			\vskip-.8mm
			\hbox{$\phantom{A_p,\ol\beta_1,\ldots,}{|\atop\nu} $}}}}
			-
			\Theta\ind{j}{i}{\ol\beta_\nu\alpha}
			\chi_{\tiny{\vtop{\hbox{$A_p,\ol\beta_1,\ldots,\ol\beta,\ldots,\ol\beta_q$}
			\vskip-.8mm
			\hbox{$\phantom{A_p,\ol\beta_1,\ldots,}{|\atop\nu} $}}}}
			\\
			&&
			\hspace{.5cm}
			+
			\sum_{\mu=1}^p
			R\ind{\alpha_\mu}{\gamma}{\ol\beta_\nu\alpha}
			\chi\ind{}{j}{
			{\tiny\vtop{
			\hbox{$\alpha_1,\ldots,\gamma,\ldots,\alpha_p,\ol\beta_1,\ldots,\ol\beta,\ldots,\ol\beta_q$}
			\vskip-.8mm
			\hbox{$\phantom{\alpha_1,\ldots,}{|\atop\mu}
			\phantom{,\ldots,\alpha_p,\ol\beta_1,\ldots,}{|\atop\nu} $}}}}
			\\
			&&
			\hspace{.5cm}
			+
			\sum_{\sigma\neq\nu}
			R\ind{\ol\beta_\sigma}{\ol\delta}{\ol\beta_\nu\alpha}
			\chi\ind{}{i}{
			{\tiny\vtop{
			\hbox{$A_p,\ol\beta_1,\ldots,\ol\delta,\ldots,\ol\beta,\ldots,\ol\beta_q$}\vskip-.8mm
			\hbox{$\phantom{A_p,\ol\beta_1,\ldots,}{|\atop\sigma}\phantom{,\ldots,}{|\atop\nu} $}}}}
			+
			R\ind{\ol\beta}{\ol\delta}{\ol\beta_\nu\alpha}
			\chi\ind{}{i}{
			{\tiny\vtop{
			\hbox{$A_p,\ol\beta_1,\ldots,\ol\delta,\ldots,\ol\beta_q$}\vskip-.8mm
			\hbox{$\phantom{A_p,\ol\beta_1,\ldots,}{|\atop\nu} $}}}}
		\Bigg)
		\\
		&=&
		\sum_{\nu=1}^q
		\Bigg(
			\Theta\indd{j}{i}{\ol\beta_\nu}{\ol\beta}
			\chi\ind{}{i}{\tiny{\vtop{\hbox{$A_p,\ol\beta_1,\ldots,\ol\beta,\ldots,\ol\beta_q$}
			\vskip-.8mm
			\hbox{$\phantom{A_p,\ol\beta_1,\ldots,}{|\atop\nu} $}}}}
			-
			\sum_{\mu=1}^p
			R\indd{\alpha_\mu}{\gamma}{\ol\beta_\nu}{\ol\beta}
			\chi\ind{}{i}{
			{\tiny\vtop{
			\hbox{$\alpha_1,\ldots,\gamma,\ldots,\alpha_p,\ol\beta_1,\ldots,\ol\beta,\ldots,\ol\beta_q$}
			\vskip-.8mm
			\hbox{$\phantom{\alpha_1,\ldots,}{|\atop\mu}
			\phantom{,\ldots,\alpha_p,\ol\beta_1,\ldots,}{|\atop\nu} $}}}}
			\\
			&&
			\hspace{3cm}
			-
			R\ind{}{\ol\delta}{\ol\beta_\nu}
			\chi\ind{}{i}{
			{\tiny\vtop{
			\hbox{$A_p,\ol\beta_1,\ldots,\ol\delta,\ldots,\ol\beta_q$}\vskip-.8mm
			\hbox{$\phantom{A_p,\ol\beta_1,\ldots,}{|\atop\nu} $}}}}
		\Bigg)
	\end{eqnarray*}
	On the other hand,
	\begin{equation*}
		g^{\ol\beta\alpha}
		\sum_{\mu=1}^p
		\chi\ind{}{i}{
		{\tiny\vtop{
		\hbox{$\alpha_1,\ldots,\alpha,\ldots,\alpha_p,\ol B_q;\ol\beta\alpha_\mu$}\vskip-.8mm
		\hbox{$\phantom{\alpha_1,\ldots,}{|\atop\mu} $}}}}
		=
		g^{\ol\beta\alpha}
		\sum_{\mu=1}^p
		\sum_{\nu=1}^q
		\chi\ind{}{i}{
		{\tiny\vtop{
		\hbox{$\alpha_1,\ldots,\alpha,\ldots,\alpha_p,\ol\beta_1,\ldots,\ol\beta,\ldots,\ol\beta_q;\ol\beta_\nu\alpha_\mu$}\vskip-.8mm
		\hbox{$\phantom{\alpha_1,\ldots,}{|\atop\mu}
		\phantom{,\ldots,\alpha_p,\ol\beta_1,\ldots,}{|\atop\nu}$}}}}.
	\end{equation*}
	Altogether, in view of \eqref{E:local_expression_curvature_operator} it follows that
	\begin{eqnarray*}
		I_{11}+I_{12}
		&=&
		\Theta\ind{k}{i}{}
		\chi\ind{}{k}{A_p,\ol B_q}
		-
		\sum_{\mu=1}^p
		\Theta\indd{j}{i}{\alpha_\mu}{\alpha}
		\chi\ind{}{j}{
		{\tiny\vtop{
		\hbox{$\alpha_1,\ldots,\alpha,\ldots,\alpha_p,\ol B_q$}\vskip-.8mm
		\hbox{$\phantom{\alpha_1,\ldots,}{|\atop\mu} $}}}}
		+
		\sum_{\nu=1}^q
		\Theta\indd{j}{i}{\ol\beta_\nu}{\ol\beta}
		\chi\ind{}{j}{\tiny{\vtop{\hbox{$A_p,\ol\beta_1,\ldots,\ol\beta,\ldots,\ol\beta_q$}
		\vskip-.8mm
		\hbox{$\phantom{A_p,\ol\beta_1,\ldots,}{|\atop\nu} $}}}}
		\\
		&&
		+
		g^{\ol\beta\alpha}
		\sum_{\mu=1}^p
		\chi\ind{}{i}{
		{\tiny\vtop{
		\hbox{$\alpha_1,\ldots,\alpha,\ldots,\alpha_p,\ol B_q;\ol\beta\alpha_\mu$}\vskip-.8mm
		\hbox{$\phantom{\alpha_1,\ldots,}{|\atop\mu} $}}}}
		\\
		&=&
		-[\ii\Theta, \Lambda]\chi
		+
		g^{\ol\beta\alpha}
		\sum_{\nu=1}^p
		\chi\ind{}{i}{
		{\tiny\vtop{
		\hbox{$\alpha_1,\ldots,\alpha,\ldots,\alpha_p,\ol B_q;\ol\beta\alpha_\nu$}\vskip-.8mm
		\hbox{$\phantom{\alpha_1,\ldots,}{|\atop\nu} $}}}}
	\end{eqnarray*}
	Finally, we consider the last term
	\begin{equation*}
		g^{\ol\beta\alpha}
		\sum_{\mu=1}^p
		\chi\ind{}{i}{
		{\tiny\vtop{
		\hbox{$\alpha_1,\ldots,\alpha,\ldots,\alpha_p,\ol B_q;\ol\beta\alpha_\mu$}\vskip-.8mm
		\hbox{$\phantom{\alpha_1,\ldots,}{|\atop\mu} $}}}}.
	\end{equation*}
	Integration by parts implies
	\begin{align*}
		\frac{1}{(p!)^2(q!)^2}
		&
		\int_X
		c(\omega)
		g^{\ol\beta\alpha}
		\sum_{\mu=1}^p
		\chi\ind{}{i}{
		{\tiny\vtop{
		\hbox{$\alpha_1,\ldots,\alpha,\ldots,\alpha_p,\ol B_q;\ol\beta\alpha_\mu$}\vskip-.8mm
		\hbox{$\phantom{\alpha_1,\ldots,}{|\atop\mu} $}}}}
		\ol{\psi\ind{}{j}{C_p,\ol D_q}}
		\cdot h_{i\ol j}
		\cdot g^{A_p\ol C_p}
		\cdot g^{\ol B_qD_q}
		\\
		&=
		-\frac{1}{(p!)^2(q!)^2}
		\int_X
		g^{\ol\beta\alpha}
		\sum_{\mu=1}^p
		c(\omega)_{;\alpha_\mu}
		\chi\ind{}{i}{
		{\tiny\vtop{
		\hbox{$\alpha_1,\ldots,\alpha,\ldots,\alpha_p,\ol B_q;\ol\beta$}\vskip-.8mm
		\hbox{$\phantom{\alpha_1,\ldots,}{|\atop\mu} $}}}}
		\ol{\psi\ind{}{j}{C_p,\ol D_q}}
		\cdot h_{i\ol j}
		\cdot g^{A_p\ol C_p}
		\cdot g^{\ol B_qD_q}
		\\
		&
		\hspace{.5cm}
		-\frac{1}{(p!)^2(q!)^2}
		\int_X
		c(\omega)
		g^{\ol\beta\alpha}
		\sum_{\mu=1}^p
		\chi\ind{}{i}{
		{\tiny\vtop{
		\hbox{$\alpha_1,\ldots,\alpha,\ldots,\alpha_p,\ol B_q;\ol\beta$}\vskip-.8mm
		\hbox{$\phantom{\alpha_1,\ldots,}{|\atop\mu} $}}}}
		\ol{\psi\ind{}{j}{C_p,\ol D_q;\ol\alpha_\mu}}
		\cdot h_{i\ol j}
		\cdot g^{A_p\ol C_p}
		\cdot g^{\ol B_qD_q}
	\end{align*}
	Then the first term is computed as
	\begin{align*}
		-\frac{1}{(p!)^2(q!)^2}
		&
		\int_X
		g^{\ol\beta\alpha}
		\sum_{\mu=1}^p
		c(\omega)_{;\alpha_\mu}
		\chi\ind{}{i}{
		{\tiny\vtop{
		\hbox{$\alpha_1,\ldots,\alpha,\ldots,\alpha_p,\ol B_q;\ol\beta$}\vskip-.8mm
		\hbox{$\phantom{\alpha_1,\ldots,}{|\atop\mu} $}}}}
		\ol{\psi\ind{}{j}{C_p,\ol D_q}}
		\cdot h_{i\ol j}
		\cdot g^{A_p\ol C_p}
		\cdot g^{\ol B_qD_q}
		\\
		&=
		\frac{1}{(p!)^2(q!)^2}
		\int_X
		(\pt c(\omega)\wedge \pt^*\chi)\ind{}{i}{\alpha_1,\ldots,\alpha_p,\ol B_q}
		\ol{\psi\ind{}{j}{C_p,\ol D_q}}
		\cdot h_{i\ol j}
		\cdot g^{A_p\ol C_p}
		\cdot g^{\ol B_qD_q}
		\\
		&=
		\inner{\pt c(\omega)\wedge\pt^*\chi, \psi},
	\end{align*}
	and the second term is computed as
	\begin{align*}
			&
			-\frac{1}{(p!)^2(q!)^2}
			\int_X
			c(\omega)
			g^{\ol\beta\alpha}
			\sum_{\nu=1}^p
			\chi\ind{}{i}{
			{\tiny\vtop{
			\hbox{$\alpha_1,\ldots,\alpha,\ldots,\alpha_p,\ol B_q;\ol\beta$}\vskip-.8mm
			\hbox{$\phantom{\alpha_1,\ldots,}{|\atop\nu} $}}}}
			\ol{\psi\ind{}{j}{C_p,\ol D_q;\ol\alpha_\nu}}
			\cdot h_{i\ol j}
			\cdot g^{A_p\ol C_p}
			\cdot g^{\ol B_qD_q}
			\\
			&
			\hspace{.5cm}
			=
			-\frac{1}{p!q!}
			\int_X
			c(\omega)
			\sum_{\mu=1}^p
			(\pt^*\chi)\ind{}{i}{\alpha_1,\ldots,\widehat\alpha_\mu,\ldots,\alpha_p,\ol B_q}
			\ol{(\pt^*\psi)\ind{}{i}{\gamma_1,\ldots,\widehat\gamma_\mu,\ldots,\gamma_p,\ol D_q}}
			\cdot h_{i\ol j}
			\cdot g^{\alpha_1\ol\gamma_1}\cdots \widehat{g^{\alpha_\mu\ol\gamma_\mu}}\cdots g^{\alpha_p\ol\gamma_p}
			\cdot g^{\ol B_qD_q}
			\\
			&
			\hspace{.5cm}
			=
			-\inner{c(\omega)\pt^*\chi,\pt^*\psi}.
	\end{align*}
	Altogether, the conclusion follows from the Bochner-Kodaira-Nakano formula.
\end{proof}

\begin{remark}
	We have the following observations.
	\begin{enumerate}
		\item
		Since $\dim_\CC S=1$, $c(\omega)$ is real, so one has $\ol{c(\omega)}=c(\omega)$.
		If $\dim_\CC S\ge2$, then the same formula holds only except that $c(\omega)$ should be replaced by the hermitian symmetric terms $c(\omega)_{i\ol\jmath}$.
		Hence \eqref{E:Lie_derivative_commutator} becomes
		\begin{align*}
			\inner{L_{[v_i,\ol{v_j}]} \chi,\psi}
			=&
			\inner{c(\omega)_{i\ol\jmath}\,\Box_\pt\chi,\psi}
			-
			\inner{c(\omega)_{i\ol\jmath}\,\partial\chi,\partial\psi}
			-
			\inner{c(\omega)_{i\ol\jmath}\,\pt^*\chi,\pt^*\psi}
			\\
			&+
			\inner{\chi,\pt \ol{c(\omega)_{i\ol\jmath}}\wedge\pt^*\psi}
			+
			\inner{\pt c(\omega)_{i\ol\jmath}\wedge\pt^*\chi, \psi}.
		\end{align*}
		\item
		The twisted Laplacian operator (see, e.g. \cite{Ohsawa_Takegoshi1987,Demailly(LN)}), which is given as
		\begin{equation*}
			[\pt,\eta\pt^*]
			=
			\eta\Box_\pt+\pt\eta\we\pt^*-(\pt\ol\eta)^*\pt
		\end{equation*}
		for any $\CC$-valued smooth function $\eta$, implies that
		\begin{align*}
			\inner{c(\omega)\,\partial\chi,\partial\psi}
			+
			\inner{c(\omega)\,\pt^*\chi,\pt^*\psi}
			=
			\inner{c(\omega)\,\Box_\pt\chi,\psi}
			+
			\inner{\pt c(\omega)\wedge\pt^*\chi,\psi}
			-
			\inner{\pt\chi,\pt \ol{c(\omega)}\wedge\psi}.
			\end{align*}
		Hence, together with Proposition~\ref{P:Lie_derivative_commutator}, one has
		\begin{align*}
			\inner{L_{[v,\ol v]} \chi,\psi}
			=
			\inner{\pt\chi,\pt\ol{c(\omega)}\wedge\psi}
			+
			\inner{\chi,\pt\ol{c(\omega)}\wedge\pt^*\psi}
			=
			\inner{\left[\paren{\pt\ol{c(\omega)}}^*,\pt\right]\chi,\psi}.
		\end{align*}
	\end{enumerate}
	
\end{remark}
We have the following corollary immediately.
\begin{corollary}\label{C:curvature_formula}
	If $\chi$ and $\psi$ are $\pt^*$-closed, then we have the following.
	\begin{equation*}
		\inner{L_{[v,\ol v]} \chi,\psi}
		=
		\inner{c(\omega)\,\Box_\pt\chi,\psi}
		-
		\inner{c(\omega)\,\partial\chi,\partial\psi}.
	\end{equation*}
		If $\chi$ and $\psi$ are $\pt$-closed, then we have the following.
	\begin{equation*}
		\inner{L_{[v,\ol v]} \chi,\psi}
		=
		-\inner{c(\omega)\,\chi,\Box_\pt\psi}
		+
		\inner{c(\omega)\,\pt^*\chi,\pt^*\psi}.
	\end{equation*}
\end{corollary}
\begin{proof}
	The first assertion follows immediately.
	If $\chi$ and $\psi$ are $\pt$-closed on fibers, then one has
	\begin{equation*}
		\inner{\pt c(\omega)\wedge\pt^*\chi,\psi}
		=
		\inner{\pt\paren{c(\omega)\pt^*\chi},\psi}
		-
		\inner{c(\omega)\pt\pt^*\chi,\psi}
		=
		\inner{c(\omega)\pt^*\chi,\pt^*\psi}
		-
		\inner{c(\omega)\Box_\pt\chi,\psi},
	\end{equation*}
	and similarly one also has
	\begin{equation*}
		\inner{\chi,\pt\ol{c(\omega)}\wedge\psi}
		=
		\inner{c(\omega)\pt^*\chi,\pt^*\psi}-\inner{c(\omega)\chi,\Box_\pt\psi}.
	\end{equation*}
	Then the second assertion follows by plugging these into \eqref{E:Lie_derivative_commutator}.
\end{proof}

Next, we prove the following proposition for computing $\inner{L_v'\chi,L_v'\psi}$ in \eqref{E:curvature_formula}.

\begin{proposition}\label{P:d*_Acupd}
	Suppose that $\chi$ is $\pt^*$-closed and $A_s\cup\Theta=0$.
	Then we have
	\begin{equation*}
		\pt^*(A_s\cup\pt\chi)
		=
		[\ii\Theta,\Lambda](A_s\cup\chi)
	\end{equation*}
	on fibers.
\end{proposition}

\begin{proof}
Note that
\begin{equation*}
	A_s\cup\partial\chi
	=
	\frac{1}{p!q!}
	\paren{
		A\ind{s}{\alpha}{\ol\beta}
		\chi\ind{}{i}{\alpha_1,\ldots,\alpha_p,\ol B_q;\alpha}
		-
		\sum^p_{\mu=1}
		A\ind{s}{\alpha}{\ol\beta}
		\chi\ind{}{i}{
		{\tiny\vtop{
		\hbox{$\alpha_1,\ldots,\alpha,\ldots,\alpha_p\ol B_q;\alpha_\mu$}
		\vskip-.8mm
		\hbox{$\phantom{\alpha_1,\ldots,}{|\atop\mu} $}}}}
	}
	dz^{\ol\beta}\wedge dz^{A_p}\we dz^{\ol B_q}.
\end{equation*}
It follows that
\begin{align*}
	\partial^*&(A_s\cup\partial\chi)
	=
	\frac{1}{(p-1)!q!}
	I dz^{\ov\beta}\wedge dz^{A_{p-1}}\we dz^{\ol B_q},
\end{align*}
where
\begin{eqnarray*}
	I
	&=&
	-g^{\ov\delta\alpha_1}\nabla_{\ov\delta}
	\paren{
	A\ind{s}{\alpha}{\ol\beta}\chi\ind{}{i}{\alpha_1,\ldots,\alpha_p,\ol B_q;\alpha}
	-
	\sum^p_{\mu=1}
	A\ind{s}{\alpha}{\ol\beta}\chi\ind{}{i}{
	{\tiny\vtop{
	\hbox{$\alpha_1,\ldots,\alpha,\ldots,\alpha_p\ol B_q;\alpha_\mu$}\vskip-.8mm
	\hbox{$\phantom{\alpha_1,\ldots,}{|\atop\mu} $}}}}
	}\\
	&=&
	-
	A\indd{s}{\alpha}{\ol\beta}{;\alpha_1}
	\paren{
		\chi\ind{}{i}{\alpha_1,\ldots,\alpha_p,\ol B_q;\alpha}
		-
		\sum^p_{\mu=1}
		\chi\ind{}{i}{
		{\tiny\vtop{
		\hbox{$\alpha_1,\alpha_2,\ldots,\alpha,\ldots,\alpha_p,\ol B_q;\alpha_\mu$}
		\vskip-.8mm
		\hbox{$\phantom{\alpha_1,\alpha_2,\ldots,}{|\atop\mu } $}}}}
	} \\
	&&
	-g^{\ov\delta\alpha_1}
	\paren{
		A\ind{s}{\alpha}{\ol\beta}
		\chi\ind{}{i}{\alpha_1,\alpha_2,\ldots,\alpha_p,\ol B_q;\alpha\ov\delta}
		-
		A\ind{s}{\alpha}{\ol\beta}
		\chi\ind{}{i}{\alpha,\alpha_2,\ldots,\alpha_p,\ol B_q;\alpha_1\ov\delta}
		-
		\sum^p_{\mu =2}
		A\ind{s}{\alpha}{\ol\beta}
		\chi\ind{}{i}{
		{\tiny\vtop{
		\hbox{$\alpha_1,\ldots,\alpha,\ldots,\alpha_p,\ol B_q;\alpha_\mu\ov\delta$}	
		\vskip-.8mm
		\hbox{$\phantom{\alpha_1,\ldots,}{|\atop\mu } $}}}}
	}.
\end{eqnarray*}
The first line vanishes since it is both symmetric and anti-symmetric in $\alpha_1$ and $\alpha$.
We denote the second line by $I_1+I_2+I_3$.
We first compute $I_1$.
\begin{eqnarray*}
	I_1
	&=&
	-g^{\ol\delta\alpha_1}A\ind{s}{\alpha}{\ol\beta}
	\chi\ind{}{i}{\alpha_1,\ldots,\alpha_p,\ol B_q;\alpha\ol\delta}
	\\
	&=&
	-A\ind{s}{\alpha}{\ol\beta}
	\Bigg(
		g^{\ol\delta\alpha_1}
		\chi\ind{}{i}{A_p,\ol B_q;\ol\delta\alpha}
		-
		\Theta\indd{j}{i}{\alpha}{\alpha_1}
		\chi\ind{}{j}{A_p,\ol B_q}
		+
		\sum_{\mu=1}^p
		\chi\ind{}{i}{
		{\tiny\vtop{
		\hbox{$\alpha_1,\ldots,\gamma,\ldots,\alpha_p,\ol B_q$}
		\vskip-.8mm
		\hbox{$\phantom{\alpha_1,\ldots,}{|\atop\mu } $}}}}
		R\indd{\alpha_\mu}{\gamma}{\alpha}{\alpha_1}
		\\
		&&
		\hspace{4cm}+
		\sum_{\nu=1}^q
		\chi\ind{}{i}{
		{\tiny\vtop{
		\hbox{$A_p,\ol\beta_1,\ldots,\ol\lambda,\ldots,\ol\beta_q$}
		\vskip-.8mm
		\hbox{$\phantom{A_p,\ol\beta_1,\ldots,}{|\atop\nu } $}}}}
		R\indd{\ol\beta_\nu}{\ol\lambda}{\alpha}{\alpha_1}
	\Bigg)
	\\
	&=&
	-A\ind{s}{\alpha}{\ol\beta}
	\paren{
		-
		\Theta\indd{j}{i}{\alpha}{\alpha_1}
		\chi\ind{}{j}{A_p,\ol B_q}
		+
		\chi\ind{}{i}{\gamma,\alpha_2,\ldots,\alpha_p,\ol B_q}
		R\ind{\alpha}{\gamma}{}
		+
		\sum_{\nu=1}^q
		\chi\ind{}{i}{
		{\tiny\vtop{
		\hbox{$A_p,\ol\beta_1,\ldots,\ol\lambda,\ldots,\ol\beta_q$}
		\vskip-.8mm
		\hbox{$\phantom{A_p,\ol\beta_1,\ldots,}{|\atop\nu } $}}}}
		R\indd{\ol\beta_\nu}{\ol\lambda}{\alpha}{\alpha_1}
	}.
\end{eqnarray*}
Next we compute $I_2$.
%
\begin{eqnarray*}
	I_2
	&=&
	g^{\ol\delta\alpha_1}
	A\ind{s}{\alpha}{\ol\beta}
	\chi\ind{}{i}{\alpha,\alpha_2,\ldots,\alpha_p,\ol B_q;\alpha_1\ol\delta}
	\\
	&=&
	g^{\ov\delta\alpha_1}A\ind{s}{\alpha}{\ol\beta}
	\Bigg(
		\chi\ind{}{i}{\alpha,\alpha_2,\ldots,\alpha_p,\ol B_q;\ol\delta\alpha_1}
		-
		\Theta\ind{j}{i}{\alpha_1\ov\delta}
		\chi\ind{}{j}{\alpha,\alpha_2,\ldots,\alpha_p,\ol B_q}
		+
		\chi\ind{}{i}{\sigma,\alpha_2,\ldots,\alpha_p,\ol B_q}
		R\ind{\alpha}{\sigma}{\alpha_1\ov\delta}
		\\
		&&
		\hspace{2cm}
		+
		\sum_{\mu=2}^p
		\chi\ind{}{i}{
		{\tiny\vtop{
		\hbox{$\alpha,\alpha_2,\ldots,\sigma,\ldots,\alpha_p,\ol B_q$}
		\vskip-.8mm
		\hbox{$\phantom{\alpha,\alpha_2,\ldots,}{|\atop\mu } $}}}}
		R\ind{\alpha_\mu}{\sigma}{\alpha_1\ov\delta}
		+
		\sum_{\nu=1}^q
		\chi\ind{}{i}{
		{\tiny\vtop{
		\hbox{$\alpha,\alpha_2,\ldots,\alpha_p,\ol\beta_1,\ldots,\ol\tau,\ldots,\ol\beta_q$}
		\vskip-.8mm
		\hbox{$\phantom{\alpha,\alpha_2,\ldots,\alpha_p,\ol\beta_1,\ldots,}{|\atop\nu } $}}}}
		R\ind{\ol\beta_\nu}{\ol\tau}{\alpha_1\ov\delta}
	\Bigg)
	\\
	&=&
	A\ind{s}{\alpha}{\ol\beta}
	\Bigg(
		g^{\ol\delta\alpha_1}
		\chi\ind{}{i}{\alpha,\alpha_2,\ldots,\alpha_p,\ol B_q;\ol\delta\alpha_1}
		-
		\Theta\ind{j}{i}{}
		\chi\ind{}{j}{\alpha,\alpha_2,\ldots,\alpha_p,\ol B_q}
		+
		\chi\ind{}{i}{\sigma,\alpha_2,\ldots,\alpha_p,\ol B_q}
		R\ind{\alpha}{\sigma}{}
		\\
		&&
		\hspace{2cm}
		+
		\sum_{\mu=2}^p
		\chi\ind{}{i}{
		{\tiny\vtop{
		\hbox{$\alpha,\alpha_2,\ldots,\sigma,\ldots,\alpha_p,\ol B_q$}
		\vskip-.8mm
		\hbox{$\phantom{\alpha,\alpha_2,\ldots,}{|\atop\mu } $}}}}
		R\ind{\alpha_\mu}{\sigma}{}
		+
		\sum_{\nu=1}^q
		\chi\ind{}{i}{
		{\tiny\vtop{
		\hbox{$\alpha,\alpha_2,\ldots,\alpha_p,\ol\beta_1,\ldots,\ol\tau,\ldots,\ol\beta_q$}
		\vskip-.8mm
		\hbox{$\phantom{\alpha,\alpha_2,\ldots,\alpha_p,\ol\beta_1,\ldots,}{|\atop\nu } $}}}}
		R\ind{\ol\beta_\nu}{\ol\tau}{}
	\Bigg).
\end{eqnarray*}
Here the first term in the bracket is
\begin{align*}
	g^{\ov\delta\alpha_1}
	&
	\chi\ind{}{i}{\alpha,\alpha_2,\ldots,\alpha_p,\ol B_q;\ol\delta\alpha_1}
	=
	g^{\ov\delta\alpha_1}
	\sum_{\nu=1}^q
	\chi\ind{}{i}{
	{\tiny\vtop{
	\hbox{$\alpha,\alpha_2,\ldots,\alpha_p,
	\ol\beta_1,\ldots,\ol\delta,\ldots,\ol\beta_q;\ol\beta_\nu\alpha_1$}
	\vskip-.8mm
	\hbox{$\phantom{\alpha,\alpha_2,\ldots,\alpha_p,\ol\beta_1,\ldots,}{|\atop\nu } $}}}}
	\\
	&=
	g^{\ov\delta\alpha_1}
	\sum_{\nu=1}^q
	\Bigg(
		\chi\ind{}{i}{
		{\tiny\vtop{
		\hbox{$\alpha,\alpha_2,\ldots,\alpha_p,
		\ol\beta_1,\ldots,\ol\delta,\ldots,\ol\beta_q;\alpha_1\ol\beta_\nu$}
		\vskip-.8mm
		\hbox{$\phantom{\alpha,\alpha_2,\ldots,\alpha_p,\ol\beta_1,\ldots,}{|\atop\nu } $}}}}
		-
		\Theta\ind{j}{i}{\ol\beta_\nu\alpha_1}
		\chi\ind{}{j}{
		{\tiny\vtop{
		\hbox{$\alpha,\alpha_2,\ldots,\alpha_p,
		\ol\beta_1,\ldots,\ol\delta,\ldots,\ol\beta_q$}
		\vskip-.8mm
		\hbox{$\phantom{\alpha,\alpha_2,\ldots,\alpha_p,\ol\beta_1,\ldots,}{|\atop\nu } $}}}}
		\\
		&\hspace{2cm}
		+
		\chi\ind{}{i}{
		{\tiny\vtop{
		\hbox{$\sigma,\alpha_2,\ldots,\alpha_p,
		\ol\beta_1,\ldots,\ol\delta,\ldots,\ol\beta_q$}
		\vskip-.8mm
		\hbox{$\phantom{\alpha,\alpha_2,\ldots,\alpha_p,\ol\beta_1,\ldots,}{|\atop\nu } $}}}}
		R\ind{\alpha}{\sigma}{\ol\beta_\nu\alpha_1}
		+
		\sum_{\mu=2}^p
		\chi\ind{}{i}{
		{\tiny\vtop{
		\hbox{$\alpha,\alpha_2,\ldots,\sigma,\ldots,\alpha_p,
		\ol\beta_1,\ldots,\ol\delta,\ldots,\ol\beta_q$}
		\vskip-.8mm
		\hbox{$\phantom{\alpha,\alpha_2,\ldots,}{|\atop\mu }
		\phantom{,\ldots,\alpha_p,\ol\beta_1,\ldots,}{|\atop\nu }$}}}}
		R\ind{\alpha_\mu}{\sigma}{\ol\beta_\nu\alpha_1}
		\\
		&\hspace{2cm}
		+
		\chi\ind{}{i}{
		{\tiny\vtop{
		\hbox{$\alpha,\alpha_2,\ldots,\alpha_p,
		\ol\beta_1,\ldots,\ol\tau,\ldots,\ol\beta_q$}
		\vskip-.8mm
		\hbox{$\phantom{\alpha,\alpha_2,\ldots,\alpha_p,\ol\beta_1,\ldots,}{|\atop\nu } $}}}}
		R\ind{\ol\delta}{\ol\tau}{\ol\beta_\nu\alpha_1}
		+
		\sum_{\kappa\neq\nu}
		\chi\ind{}{i}{
		{\tiny\vtop{
		\hbox{$\sigma,\alpha_2,\ldots,\alpha_p,
		\ol\beta_1,\ldots,\ol\tau,\ldots,\ol\delta,\ldots,\ol\beta_q$}
		\vskip-.8mm
		\hbox{$\phantom{\alpha,\alpha_2,\ldots,\alpha_p,\ol\beta_1,\ldots,}{|\atop\kappa}
		\phantom{,\ldots,}{|\atop\nu}$}}}}
		R\ind{\ol\beta_\kappa}{\ol\tau}{\ol\beta_\nu\alpha_1}
	\Bigg)
	\\
	&=
	\sum_{\nu=1}^q
	\Bigg(
		-
		\Theta\indd{j}{i}{\ol\beta_\nu}{\ol\delta}
		\chi\ind{}{j}{
		{\tiny\vtop{
		\hbox{$\alpha,\alpha_2,\ldots,\alpha_p,
		\ol\beta_1,\ldots,\ol\delta,\ldots,\ol\beta_q$}
		\vskip-.8mm
		\hbox{$\phantom{\alpha,\alpha_2,\ldots,\alpha_p,\ol\beta_1,\ldots,}{|\atop\nu } $}}}}
		+
		\chi\ind{}{i}{
		{\tiny\vtop{
		\hbox{$\sigma,\alpha_2,\ldots,\alpha_p,
		\ol\beta_1,\ldots,\ol\delta,\ldots,\ol\beta_q$}
		\vskip-.8mm
		\hbox{$\phantom{\alpha,\alpha_2,\ldots,\alpha_p,\ol\beta_1,\ldots,}{|\atop\nu } $}}}}
		R\indd{\alpha}{\sigma}{\ol\beta_\nu}{\ol\delta}
		\\
		&\hspace{2cm}
		+
		\sum_{\mu=2}^p
		\chi\ind{}{i}{
		{\tiny\vtop{
		\hbox{$\alpha,\alpha_2,\ldots,\sigma,\ldots,\alpha_p,
		\ol\beta_1,\ldots,\ol\delta,\ldots,\ol\beta_q$}
		\vskip-.8mm
		\hbox{$\phantom{\alpha,\alpha_2,\ldots,}{|\atop\mu }
		\phantom{,\ldots,\alpha_p,\ol\beta_1,\ldots,}{|\atop\nu }$}}}}
		R\indd{\alpha_\mu}{\sigma}{\ol\beta_\nu}{\ol\delta}
		+
		\chi\ind{}{i}{
		{\tiny\vtop{
		\hbox{$\alpha,\alpha_2,\ldots,\alpha_p,
		\ol\beta_1,\ldots,\ol\tau,\ldots,\ol\beta_q$}
		\vskip-.8mm
		\hbox{$\phantom{\alpha,\alpha_2,\ldots,\alpha_p,\ol\beta_1,\ldots,}{|\atop\nu } $}}}}
		R\ind{}{\ol\tau}{\ol\beta_\nu}
	\Bigg).
\end{align*}
Hence $I_2$ is computed as follows.
\begin{eqnarray*}
I_2
&=&
A\ind{s}{\alpha}{\ol\beta}
\Bigg(
	-
	\Theta\ind{j}{i}{}
	\chi\ind{}{j}{\alpha,\alpha_2,\ldots,\alpha_p,\ol B_q}
	-
	\sum_{\nu=1}^q
	\Theta\indd{j}{i}{\ol\beta_\nu}{\ol\delta}
	\chi\ind{}{j}{
	{\tiny\vtop{
	\hbox{$\alpha,\alpha_2,\ldots,\alpha_p,
	\ol\beta_1,\ldots,\ol\delta,\ldots,\ol\beta_q$}
	\vskip-.8mm
	\hbox{$\phantom{\alpha,\alpha_2,\ldots,\alpha_p,\ol\beta_1,\ldots,}{|\atop\nu } $}}}}
	+
		\chi\ind{}{i}{\sigma,\alpha_2,\ldots,\alpha_p,\ol B_q}
		R\ind{\alpha}{\sigma}{}
	\\
	&&
	\hspace{2cm}
	+
	\sum_{\mu=2}^p
		\chi\ind{}{i}{
		{\tiny\vtop{
		\hbox{$\alpha,\alpha_2,\ldots,\sigma,\ldots,\alpha_p,\ol B_q$}
		\vskip-.8mm
		\hbox{$\phantom{\alpha,\alpha_2,\ldots,}{|\atop\mu } $}}}}
		R\ind{\alpha_\mu}{\sigma}{}
	+\sum_{\nu=1}^q
		\chi\ind{}{i}{
		{\tiny\vtop{
		\hbox{$\sigma,\alpha_2,\ldots,\alpha_p,
		\ol\beta_1,\ldots,\ol\delta,\ldots,\ol\beta_q$}
		\vskip-.8mm
		\hbox{$\phantom{\alpha,\alpha_2,\ldots,\alpha_p,\ol\beta_1,\ldots,}{|\atop\nu } $}}}}
		R\indd{\alpha}{\sigma}{\ol\beta_\nu}{\ol\delta}
	\\
	&&
	\hspace{2cm}
	+
	\sum_{\mu=2}^p\sum_{\nu=1}^q
		\chi\ind{}{i}{
		{\tiny\vtop{
		\hbox{$\alpha,\alpha_2,\ldots,\sigma,\ldots,\alpha_p,
		\ol\beta_1,\ldots,\ol\delta,\ldots,\ol\beta_q$}
		\vskip-.8mm
		\hbox{$\phantom{\alpha,\alpha_2,\ldots,}{|\atop\mu }
		\phantom{,\ldots,\alpha_p,\ol\beta_1,\ldots,}{|\atop\nu }$}}}}
		R\indd{\alpha_\mu}{\sigma}{\ol\beta_\nu}{\ol\delta}
\Bigg).
\end{eqnarray*}
Finally,
\begin{eqnarray*}
	I_3
	&=&
	g^{\ov\delta\alpha_1}
	\sum^p_{\mu =2}
	A\ind{s}{\alpha}{\ol\beta}
	\chi\ind{}{i}{
	{\tiny\vtop{
	\hbox{$\alpha_1,\ldots,\alpha,\ldots,\alpha_p,\ol B_q;
	\alpha_\mu\ov\delta$}
	\vskip-.8mm
	\hbox{$\phantom{\alpha_1,\ldots,}{|\atop\mu } $}}}}
	\\
	&=&
	g^{\ov\delta\alpha_1}
	A\ind{s}{\alpha}{\ol\beta}
	\sum^p_{\mu =2}
	\Bigg(
		\chi\ind{}{j}{
		{\tiny\vtop{
		\hbox{$\alpha_1,\ldots,\alpha,\ldots,\alpha_p,\ol B_q;\ol\delta\alpha_\mu$}
		\vskip-.8mm
		\hbox{$\phantom{\alpha_1,\ldots,}{|\atop\mu } $}}}}
		-
		\Theta\ind{j}{i}{\alpha_\mu\ov\delta}
		\chi\ind{}{j}{
		{\tiny\vtop{
		\hbox{$\alpha_1,\ldots,\alpha,\ldots,\alpha_p,\ol B_q$}
		\vskip-.8mm
		\hbox{$\phantom{\alpha_1,\ldots,}{|\atop\mu } $}}}}
		+
		\chi\ind{}{i}{
		{\tiny\vtop{
		\hbox{$\sigma,\alpha_2,\ldots,\alpha,\ldots,\alpha_p,\ol B_q$}
		\vskip-.8mm
		\hbox{$\phantom{\sigma,\alpha_2,\ldots,}{|\atop\mu } $}}}}
		R\ind{\alpha_1}{\sigma}{\alpha_\mu\ov\delta}
		\\
		&&
		\hspace{3.2cm}
		+
		\chi\ind{}{i}{
		{\tiny\vtop{
		\hbox{$\alpha_1,\ldots,\sigma,\ldots,\alpha_p,\ol B_q$}
		\vskip-.8mm
		\hbox{$\phantom{\alpha_1,\ldots,}{|\atop\mu } $}}}}
		R\ind{\alpha}{\sigma}{\alpha_\mu\ov\delta}
		+
		\sum_{\substack{\kappa=2 \\ \kappa\neq\mu}}^p
		\chi\ind{}{i}{
		{\tiny\vtop{
		\hbox{$\alpha_1,\ldots,\sigma,\ldots,\alpha,\ldots\alpha_p,\ol B_q$}
		\vskip-.8mm
		\hbox{$\phantom{\alpha_1,\ldots,}
		{|\atop\kappa}\phantom{,\ldots,}{|\atop\mu} $}}}}
		R\ind{\alpha_\kappa}{\sigma}{\alpha_\mu\ov\delta}
		\\
		&&
		\hspace{3.2cm}
		+
		\sum_{\nu=1}^q
		\chi\ind{}{i}{
		{\tiny\vtop{
		\hbox{$\alpha_1,\ldots,\alpha,\ldots,\alpha_p,
		\ol\beta_1,\ldots,\ol\tau,\ldots,\ol\beta_q$}
		\vskip-.8mm
		\hbox{$\phantom{\alpha_1,\ldots,}{|\atop\mu}
		\phantom{,\ldots,\alpha_p,\ol\beta_1,\ldots,}{|\atop\nu}
		 $}}}}
		R\ind{\ol\beta_\nu}{\ol\tau}{\alpha_\mu\ol\delta}
	\Bigg)
	\\
	&=&
	A\ind{s}{\alpha}{\ol\beta}
	\sum^p_{\mu =2}
	\Bigg(
		-
		\Theta\indd{j}{i}{\alpha_\mu}{\alpha_1}
		\chi\ind{}{j}{
		{\tiny\vtop{
		\hbox{$\alpha_1,\ldots,\alpha,\ldots,\alpha_p,\ol B_q$}
		\vskip-.8mm
		\hbox{$\phantom{\alpha_1,\ldots,}{|\atop\mu } $}}}}
		+
		\chi\ind{}{i}{
		{\tiny\vtop{
		\hbox{$\sigma,\alpha_2,\ldots,\alpha,\ldots,\alpha_p,\ol B_q$}
		\vskip-.8mm
		\hbox{$\phantom{\sigma,\alpha_2,\ldots,}{|\atop\mu } $}}}}
		R\ind{\alpha_\mu}{\sigma}{}{}
		\\
		&&
		\hspace{3.2cm}
		+
		\sum_{\nu=1}^q
		\chi\ind{}{i}{
		{\tiny\vtop{
		\hbox{$\alpha_1,\ldots,\alpha,\ldots,\alpha_p,
		\ol\beta_1,\ldots,\ol\tau,\ldots,\ol\beta_q$}
		\vskip-.8mm
		\hbox{$\phantom{\alpha_1,\ldots,}{|\atop\mu}
		\phantom{,\ldots,\alpha_p,\ol\beta_1,\ldots,}{|\atop\nu}
		 $}}}}
		R\indd{\ol\beta_\nu}{\ol\tau}{\alpha_\mu}{\alpha_1}
	\Bigg).
\end{eqnarray*}
Altogether, we have
\begin{align*}
	I
	&=
	-
	A\ind{s}{\alpha}{\ol\beta}
	\Bigg(
		\Theta\ind{j}{i}{}
		\chi\ind{}{j}{\alpha,\alpha_2,\ldots,\alpha_p,\ol B_q}
		-
		\Theta\indd{j}{i}{\alpha}{\alpha_1}
		\chi\ind{}{j}{A_p,\ol B_q}
		\\
		&\hspace{3cm}
		+
		\sum_{\mu=2}^p
		\Theta\indd{j}{i}{\alpha_\mu}{\alpha_1}
		\chi\ind{}{j}{
		{\tiny\vtop{
		\hbox{$\alpha_1,\ldots,\alpha,\ldots,\alpha_p,\ol B_q$}
		\vskip-.8mm
		\hbox{$\phantom{\alpha_1,\ldots,}{|\atop\mu } $}}}}
		+
		\sum_{\nu=1}^q
		\Theta\indd{j}{i}{\ol\beta_\nu}{\ol\delta}
		\chi\ind{}{i}{
		{\tiny\vtop{
		\hbox{$\alpha,\alpha_2,\ldots,\alpha_p,
		\ol\beta_1,\ldots,\ol\delta,\ldots,\ol\beta_q$}
		\vskip-.8mm
		\hbox{$\phantom{\alpha,\alpha_2,\ldots,\alpha_p,\ol\beta_1,\ldots,}{|\atop\nu } $}}}}
	\Bigg).
\end{align*}
Taking into account the assumption $A_s\cup\Theta=0$, it follows that
\begin{align*}
I
	&=
	-
	\Theta\ind{j}{i}{}
	A\ind{s}{\alpha}{\ol\beta}
		\chi\ind{}{j}{\alpha,\alpha_2,\ldots,\alpha_p,\ol B_q}
	+
	\Theta\ind{j}{i\ol\delta}{\ol\beta}
	A\ind{s}{\alpha}{\ol\delta}
	\chi\ind{}{j}{\alpha,\alpha_2,\ldots,\alpha_p,\ol B_q}
	\\
	&\hspace{1cm}
	+
	\sum_{\mu=2}^p
	\Theta\indd{j}{i}{\alpha_\mu}{\gamma}
	A\ind{s}{\alpha}{\ol\beta}
	\chi\ind{}{j}{
	{\tiny\vtop{
	\hbox{$\alpha,\alpha_2\ldots,\gamma,\ldots,\alpha_p,\ol B_q$}
	\vskip-.8mm
	\hbox{$\phantom{\alpha,\alpha_2\ldots,}{|\atop\mu } $}}}}
	+
	\sum_{\nu=1}^q
	\Theta\indd{j}{i\ol\delta}{\ol\beta_\nu}{}
	\chi\ind{}{i}{
	{\tiny\vtop{
	\hbox{$\alpha,\alpha_2,\ldots,\alpha_p,
	\ol\beta_1,\ldots,\ol\delta,\ldots,\ol\beta_q$}
	\vskip-.8mm
	\hbox{$\phantom{\alpha,\alpha_2,\ldots,\alpha_p,\ol\beta_1,\ldots,}{|\atop\nu } $}}}}.
 \end{align*}
This completes the proof.
\end{proof}

\section{Untwisted case}
In this section, we show that the results from the previous sections, applied to the untwisted case specialize to Griffiths' classical curvature formula for Hodge bundles. We consider the Hodge sheaves $R^qf_*\Omega_{\cX/S}^p(\CC)$.
In this case, the curvature $\Theta_h(\CC)$ vanishes and we have the Kähler identity.
\begin{equation*}
\Box=\Box_\pt=\Box_\dbar.
\end{equation*}
This implies that a local holomorphic section of $\psi$ of $R^qf_*\Omega_{\cX/S}^p(\CC)$ given by a form $\psi$ as in Lemma~\ref{L:representative} satisfies the following.
\begin{itemize}
\item $\psi$ is $\dbar_\cX$-closed,
\item $\psi$ is $\pt$-harmonic and $\dbar$-harmonic on fibers.
\end{itemize}
Then \eqref{E:curvature_formula} and Corollary~\ref{C:curvature_formula} yield that
\begin{equation}\label{E:untwisted_case}
R\paren{\pt_s,\pt_{\ol s},\psi,\ol\psi}
=
-\norm{L_v'\psi}^2
+
\norm{A_s\cup\psi}^2
+
\norm{L_{\ol v}'\psi}^2
-
\norm{A_{\ol s}\cup\psi}^2.
\end{equation}

\begin{theorem}{\cite{Griffiths1969,Griffiths1970,Griffiths_Tu1984}}
The curvature of $R^qf_*\Omega_{\cX/S}^p(\CC)$ is given as follows.
\begin{equation*}
R\paren{\pt_s,\pt_{\ol s},\psi,\ol\psi}
=
\norm{H(A_s\cup\psi)}^2
-
\norm{H(A_{\ol s}\cup\psi)}^2,
\end{equation*}
where $H(A_s\cup\psi)$ and $H(A_{\ol s}\cup\psi)$ are the harmonic parts of $A_s\cup\psi$ and $A_{\ol s}\cup\psi$, respectively.
\end{theorem}

\begin{proof}
It is enough to compute $\norm{L_v'\psi}^2$ and $\norm{L_\ol v'\psi}^2$.
Since in view of \eqref{E:orthogonal_to_harmonic_space}, $L_v\psi$ can be assumed to be orthogonal to the harmonic space, we have
\begin{align*}
	\norm{L_v'\psi}^2
	&=
	\inner{L_v'\psi,L_v'\psi}
	=
	\inner{G_\dbar\Box L_v'\psi,L_v'\psi}
	=
	\inner{G_\dbar\dbar^*\dbar L_v'\psi,L_v'\psi}
	\\
	&=
	\inner{G_\dbar\dbar L_v'\psi,\dbar L_v'\psi}
	=
	\inner{G_\dbar\dbar^*(A_s\cup\psi),\dbar^*(A_s\cup\psi)}
	\\
	&=
	\inner{G_\dbar\Box(A_s\cup\psi),(A_s\cup\psi)}
	=
	\inner{(A_s\cup\psi)^\perp,A_s\cup\psi}
	=
	\norm{(A_s\cup\psi)^\perp}^2.
\end{align*}
Similarly, we have $\norm{L_\ol v'\psi}^2=\norm{(A_\ol s\cup\psi)^\perp}^2$.
Plugging these into \eqref{E:untwisted_case}, the conclusion follows.
\end{proof}

\section{Line bundle case}\label{S:linebundle}
In this section, we discuss the curvature formula of the higher direct image sheaves $R^qf_*\Omega^p_{\cX/S}(\cL)$ where $(\cL,h)$ is a hermitian line bundle equipped with a hermitian metric $h$ which is positively (negatively) curved on fibers.
For the sake of simplicity, we assume that $\dim S=1$ and we also only consider the case when $\psi^{(k)}=\psi^{(l)}=\psi$.
The general formula is easily obtained by polarization.

In this case, the fiberwise K\"ahler form $\omega$ can be taken as $\omega:=\ii\Theta_h(\cE)$
(or $\omega:=-\ii\Theta_h(\cE)$ if $h$ is negatively curved on fibers).
Then it is easy to see the following.
\begin{itemize}
\item[(\romannumeral1)] $\eta_s$ and $\eta_{\ol s}$ vanish.
\item[(\romannumeral2)] $A_s\cup\Theta_h(\cL)$ vanishes.
\end{itemize}
Indeed, (\romannumeral1) comes from the fact that $v$ is a horizontal lift with respect to $\omega$ and
(\romannumeral2) follows by
\begin{equation*}
A_s\cup\Theta
=
A_s\cup\ii\omega
=
\ii A\ind{s}{\alpha}{\ol\beta}g_{\alpha\ol\delta}dz^{\ol\beta}\we dz^{\ol\delta}
=
\ii A\ind{s}{}{\ol\delta\ol\beta}dz^{\ol\beta}\we dz^{\ol\delta}
=0,
\end{equation*}
since $A\ind{s}{}{\ol\delta\ol\beta}=A\ind{s}{}{\ol\beta\ol\delta}$.
Invoking \eqref{E:curvature_formula}, Proposition~\ref{P:dbar_Lv} and Proposition~\ref{P:dbar*_Lvbar}, the curvature formula for $R^qf_*\Omega_{\cX/S}(\cL)$ is given as follows.
\begin{equation*}
		R\paren{\pt_s,\pt_{\ol s},\psi,\ol\psi}
		=
		\inner{L_{[v,\ol v]}\psi,\psi}
		+
		\inner{c(\omega)\psi,\psi}
		-
		\norm{L_v'\psi}^2
		+
		\norm{A_s\cup\psi}^2
		+
		\norm{L_{\ol v}'\psi}^2
		-
		\norm{A_{\ol s}\cup\psi}^2,
\end{equation*}
where
\begin{equation*}
	\dbar L_v'\psi=\pt(A_s\cup\psi)+A_s\cup\pt\psi
	\quad\text{and}\quad
	\dbar^*L_{\ol v}'\psi
	=
	(-1)^p\pt^*(A_{\ol s}\cup\psi)+(-1)^pA_{\ol s}\cup\pt^*\psi.
\end{equation*}
Now we consider three special cases : (1) $p+q=n$, (2) $q=0$, and (3) $p=n$.

\subsection{Case : $p+q=n$}
First we consider $R^qf_*\Omega^q_{\cX/S}(\cL)$ with $p+q=n$ where $\cL$ is a hermitian line bundle equipped with a hermitian metric $h$ that is positively (resp. negatively) curved on fibers, i.e., $\Theta_h(\cL)\vert_{\cX_s}>0$ (resp. $<0)$ on each fiber $\cX_s$.
This case is already studied by several authors (for the details, see~\cite{Schumacher2012,Naumann2021,Berndtsson_Paun_Wang2022}).
Here we discuss how to derive the known result from our computations.
Let a local holomorphic section of $R^qf_*\Omega^p_{\cX/S}(\cE)$ be given which is represented by a form $\psi$ in the sense of Lemma~\ref{L:representative}.
Then the Bochner-Kodaira-Nakano formula,
\begin{equation*}
\Box_{\dbar}-\Box_\pt
=
[\ii\Theta_h(\cL),\Lambda]
=
[\omega,\Lambda]
=
-(n-p-q)\id
\end{equation*}
for all $\cL$-valued $(p,q)$-forms, implies that $\psi$ is also $\pt$-harmonic on fibers, i.e., $\Box_\pt\psi=0$ on fibers.
Now Corollary~\ref{C:curvature_formula} yields that $\inner{L_{[v,\ol v]}\psi,\psi}=0$, so the curvature formula becomes
\begin{align*}
R\paren{\pt_s,\pt_{\ol s},\psi,\ol\psi}
&=
\inner{c(\omega)\psi,\psi}
-\norm{L_v'\psi}^2
+
\norm{A_s\cup\psi}^2
+
\norm{L_{\ol v}'\psi}^2
-
\norm{A_{\ol s}\cup\psi}^2.
\end{align*}
The second term and the forth term can be computed by Proposition~\ref{P:dbar_Lv}, Proposition~\ref{P:dbar*_Lvbar} and the Bochner-Kodaira-Nakano formula.
Naumann, Berndtsson-P\u aun-Wang proved the following theorems independently (\cite{Naumann2021,Berndtsson_Paun_Wang2022}).
\begin{itemize}
\item[(\romannumeral1)] $\Theta_h(\cL)$ is positive on fibers:
\begin{align*}
R(\pt_s,\pt_{\ol s},\psi,\psi)
&=
\inner{c(\omega)\psi,\psi}
+
\inner{
	(\Box_{\dbar}+1)^{-1}
	(A_s\cup\psi)
	,
	(A_s\cup\psi)
}
+
\inner{(\Box_{\dbar}-1)^{-1}(A_{\ol s}\cup\psi),A_{\ol s}\cup\psi}.
\end{align*}
\item[(\romannumeral2)] $\Theta_h(\cL)$ is negative on fibers:
\begin{align*}
R(\pt_s,\pt_{\ol s},\psi,\ol\psi)
&=
-
\inner{c(\omega)\psi,\psi}
-
\inner{
	(\Box_{\dbar}-1)^{-1}
	(A_s\cup\psi)
	,
	(A_s\cup\psi)
}
-
\inner{(\Box_{\dbar}+1)^{-1}(A_{\ol s}\cup\psi),A_{\ol s}\cup\psi}
\end{align*}
\end{itemize}

\subsection{Case : $q=0$}
Let $\psi$ be a local holomorphic section of $f_*\Omega_{\cX/S}^p(\cL)$ as represented in Lemma~\ref{L:representative}.
Considering the bidegree it is obvious that $\psi$ is primitive on fibers and $A_{\ol s}\cup\psi$ vanishes.
In particular, this implies that $\psi$ is $\pt^*$-closed due to the fundamental K\"ahler identity (see, for example \cite{Demailly(Book)}):
\begin{equation*}
	-\ii\pt^*=[\Lambda,\dbar].
\end{equation*}
Suppose that $\ii\Theta_h(\cL)$ is positive on fibers.
As before, we take the fiberwise Kähler-form $\omega=\ii\Theta_h(\cL)$.
Then a local section $\psi$ of $f_*\Omega_{\cX/S}^p(\cL)$ satisfies
\begin{equation*}
-\Box_\pt\psi=[\ii\Theta_h(\cL),\Lambda]\psi
=
[\omega_{\cX_s},\Lambda]=-(n-p)\psi
\end{equation*}
on fibers.

In case of hermitian line bundles, we can compute $\norm{L_v'\psi}^2$ and $\norm{L_{\ol v}'\psi}^2$ more precisely.
First we compute $\norm{L_v'\psi}^2$.
\begin{eqnarray*}
\norm{L_v'\psi}^2
&=&
\inner{L_v'\psi,L_v'\psi}
=
\inner{G_{\dbar}\Box_{\dbar}L_v'\psi,L_v'\psi}
=
\inner{G_{\dbar}\dbar^*\dbar L_v'\psi,L_v'\psi}
\\
&=&
\inner{G_{\dbar}\dbar L_v'\psi,\dbar L_v'\psi}
=
\inner{\Box_{\dbar}^{-1}\dbar L_v'\psi,\dbar L_v'\psi}
\\
&=&
\inner{
	(\Box_\pt-(n-p-1))^{-1}
	\paren{
		\pt(A_s\cup\psi)
		+
		A_s\cup\pt\psi
	},
	\paren{
		\pt(A_s\cup\psi)
		+
		A_s\cup\pt\psi
	}
}
\\
&=&
\inner{
	(\Box_\pt-(n-p-1))^{-1}
	\Box_\pt(A_s\cup\psi)
	,
	A_s\cup\psi
}
\\
&&
+
\inner{
	(\Box_\pt-(n-p-1))^{-1}
	(A_s\cup\psi)
	,
	\pt^*(A_s\cup\pt\psi)
}
\\
&&+
\inner{
	(\Box_\pt-(n-p-1))^{-1})
	\pt^*(A_s\cup\pt\psi)
	,
	(A_s\cup\psi)
}
\\
&&
+
\inner{
	(\Box_\pt-(n-p-1))^{-1}
	A_s\cup\pt\psi
	,A_s\cup\pt\psi
}
\\
&=:&
I_1+I_2+I_3+I_4.
\end{eqnarray*}
Proposition~\ref{P:d*_Acupd} implies that
\begin{align*}
	I_1&
	=
	\inner{
		(\Box_\pt-(n-p-1))^{-1}
		\Box_\pt(A_s\cup\psi)
		,
		A_s\cup\psi
	}
\\
&=
(n-p-1)\inner{(\Box_\pt-(n-p-1))^{-1}A_s\cup\psi,A_s\cup\psi}
+
\norm{A_s\cup\psi}^2
\\
&
=
(n-p-1)\inner{(\Box_{\dbar}+1)^{-1}A_s\cup\psi,A_s\cup\psi}
+
\norm{A_s\cup\psi}^2.
\end{align*}
We also have
\begin{equation*}
		I_4
		=
		\inner{
				(\Box_\pt-(n-p-1))^{-1}
				A_s\cup\pt\psi
				,A_s\cup\pt\psi
		}
		=
		\inner{
				\Box_\dbar^{-1}
				A_s\cup\pt\psi
				,A_s\cup\pt\psi
		}.
\end{equation*}
Proposition~\ref{P:d*_Acupd} states that $\pt^*(A_s\cup\pt\psi)=-(n-p)\paren{A_s\cup\psi}$, so it follows that
\begin{align*}
		I_2
		&=
		\inner{
					(\Box_\pt-(n-p-1))^{-1}
					(A_s\cup\psi)
					,
					\pt^*(A_s\cup\pt\psi)
		}
		=
		-(n-p)
		\inner{
					(\Box_\dbar+1)^{-1}
					(A_s\cup\psi)
					,
					A_s\cup\psi
		}
\end{align*}
and
\begin{align*}
I_3
=
\inner{
	(\Box_\pt-\ell+1)^{-1}
	\pt^*(A_s\cup\psi)
	,
	(A_s\cup\pt\psi)
}
&=
-(n-p)
\inner{
	(\Box_{\dbar}+1)^{-1}
	(A_s\cup\psi)
	,
	(A_s\cup\psi)
}.
\end{align*}
Alltogether, we have
\begin{equation*}
	\norm{L_v'\psi}^2
	=
	\norm{A_s\cup\psi}^2
	-(n-p+1)\inner{(\Box_{\dbar}+1)^{-1}A_s\cup\psi,A_s\cup\psi}
	+
			\inner{
				\Box_\dbar^{-1}
				A_s\cup\pt\psi
				,A_s\cup\pt\psi
		}.	
\end{equation*}
Here we used the following lemma.
\begin{lemma}\label{L:well_definedness1}
$\pt(A_s\cup\psi)$ and $A_s\cup\pt\psi$ are orthogonal to the harmonic space.
\end{lemma}
\begin{proof}
Let $\sum\rho_\nu$ be the decomposition of $A_s\cup\psi$ where the forms $\rho_\nu$ are eigenfunctions of $\Box_\dbar$ with  eigenvalues $\lambda_\nu$.
Then it follows from $\Box_\dbar\rho_\nu=\lambda_\nu\rho_\nu$ that
\begin{align*}
\lambda_\nu\pt\rho_\nu
=
\pt\Box_\dbar\rho_\nu
=
\pt\paren{\Box_\pt-\ell}\rho_\nu
=
\paren{\Box_\pt-\ell}\pt\rho_\nu
=
\paren{\Box_\dbar-1}\pt\rho_\nu,
\end{align*}
or
\begin{equation*}
\Box_\dbar\pt\rho_\nu
=
\paren{\lambda_\nu+1}\pt\rho_\nu
\end{equation*}
This means that all $\pt\rho_\nu$ are eigenfunctions with eigenvalues $\lambda_\nu+1$.
Since $\pt(A_s\cup\psi)=\sum\pt\rho_\nu$, the forms $\pt(A_s\cup\psi)$ do not have an harmonic part.
It also follows from Proposition~\ref{P:dbar_Lv} that $A_s\cup\pt\psi$ is orthogonal to the harmonic space.
\end{proof}

On the other hand,  by Proposition~\ref{P:Lie_derivative} and reasons of bidegree the Lie derivative $L_{\ol v}'\psi$ vanishes, in particular, we have $\norm{L_{\ol v}'\psi}^2=0$.
\begin{theorem}
\label{T:positive_p0}
Let $f:\cX\rightarrow S$ be a family of compact Kähler manifolds with a hermitian line bundle $(\cL,h)$ such that $\ii\Theta_h(\cL)$ is positive on fibers.
Then the curvature of $f_*\Omega^p_{\cX/S}(\cL)$ is given by
\begin{align*}
R(\pt_s,\pt_{\ol s},\psi,\ol\psi)
=
&
(n-p+1)
\inner{c(\omega)\psi,\psi}
-
\inner{c(\omega)\pt\psi,\pt\psi}
\\
&
+
(n-p+1)
\inner{
	(\Box_{\dbar}+1)^{-1}
	(A_s\cup\psi)
	,
	A_s\cup\psi
}
-
\inner{
	\Box_{\dbar}^{-1}
	(A_s\cup\pt\psi)
	,A_s\cup\pt\psi
}.
\end{align*}
\end{theorem}

Before going further, we prove the following proposition which says that the formula in Theorem \ref{T:positive_p0} is compatible with a previous result of Berndtsson for the  special case  of $p=n$ and $q=0$; we have the following theorem.
\begin{theorem}[\cite{Berndtsson2011,Naumann2021}]
The curvature of $f_*(K_{\cX/S})(\cL)$ is given as follows.
\begin{equation*}
R^H(\pt_s,\pt_\ol s,\psi,\ol\psi)
=
\inner{c(\omega)\cdot\psi , \psi}
+
\inner{
	\paren{
		\Box_{\dbar}+1}^{-1}\paren{A_s\cup\psi},
		A_s\cup\psi
	}
.
\end{equation*}
\end{theorem}

\begin{proof}
Note that  $f_*K_{\cX/S}(\cL)=f_*\Omega^n_{\cX/S}(\cL)$.
Since $\psi$ is $\dbar$-harmonic $(n,0)$-form on fibers, it is also $\pt$-harmonic on fibers by the Bochner-Kodaira-Nakano formula.
Hence, the conclusion readily follows from Theorem~\ref{T:positive_p0}.
\end{proof}

\subsection{Case : $p=n$}\label{SS:p=n}
Our last case concerns $R^qf_*\Omega^n_{\cX/S}(\cL)$.
Let a local holomorphic section of $R^qf_*\Omega^n_{\cX/S}(\cL)$ be represented by $\psi$ according to Lemma~\ref{L:representative}.
Then $\pt\psi$ and $A_{\ol s}\cup\psi$ vanish on fibers by reasons of bidegree.
Suppose that $\ii\Theta_h(\cL)$ is negative on fibers.
As before, we take the fiberwise Kähler-form $\omega$ by $-\ii\Theta_h(\cL)$.
Then for a local  section of $R^qf_*\Omega^n_{\cX/S}(\cL)$ given by $\psi$ as above we have
\begin{equation*}
-\Box_\pt\psi
=
[\ii\Theta_h(\cL),\Lambda]\psi
=
[-\omega_{\cX_s},\Lambda]=-q\psi.
\end{equation*}
on fibers.
Hence Corollary~\ref{C:curvature_formula} implies that
\begin{equation*}
	\inner{L_{[v,\ol v]} \chi,\psi}
	=
	q\inner{c(\omega)\,\chi,\psi}
	+
	\inner{c(\omega)\,\pt^*\chi,\pt^*\psi}.
\end{equation*}
Next we compute $\norm{L_v'\psi}^2$.
\begin{eqnarray*}
\norm{L_v'\psi}^2
&=&
\inner{L_v'\psi,L_v'\psi}
=
\inner{G_{\dbar}\Box_{\dbar}L_v'\psi,L_v'\psi}
=
\inner{G_{\dbar}\dbar^*\dbar L_v'\psi,L_v'\psi}
\\
&=&
\inner{G_{\dbar}\dbar L_v'\psi,\dbar L_v'\psi}
=
\inner{\Box_{\dbar}^{-1}\dbar L_v'\psi,\dbar L_v'\psi}
\\
&=&
\inner{
	\paren{\Box_\pt-(q+1)}^{-1}
	\paren{\pt(A_s\cup\psi)}
	,
	\pt(A_s\cup\psi)
}
\\
&=&
\inner{
	\paren{\Box_\pt-(q+1)}^{-1}
	\Box_\pt(A_s\cup\psi)
	,
	A_s\cup\psi
}
\\
&=&
(q+1)\inner{\paren{\Box_\pt-(q+1)}^{-1}A_s\cup\psi,A_s\cup\psi}
+
\norm{A_s\cup\psi}^2
\\
&=&
(q+1)\inner{\paren{\Box_\dbar-1}^{-1}A_s\cup\psi,A_s\cup\psi}
+
\norm{A_s\cup\psi}^2
\end{eqnarray*}
On the other hand, a similar argument using Proposition~\ref{P:dbar*_Lvbar} and Proposition~\ref{P:auxiliary_formulas} implies that
\begin{eqnarray*}
	\inner{L_{\ol v}'\psi,L_{\ol v}'\psi}
	&=&
	\inner{G_{\dbar}\dbar^*L_{\ol v}'\psi,\dbar^*L_{\ol v}'\psi}
	=
	\inner{
		\Box_\dbar^{-1}
		\paren{A_{\ol s}\cup\pt^*\psi}
		,
		\paren{A_{\ol s}\cup\pt^*\psi}
	}
\end{eqnarray*}
Hence we have
\begin{theorem}
\label{T:curvature_formula_(n,q)}
The curvature of $R^qf_*\Omega^n_{\cX/S}(\cL)$ is given as follows.
\begin{eqnarray*}
	R(\pt_s,\pt_{\ol s},\psi,\psi)
	&=&
	(q-1)
	\inner{c(\omega)\psi,\psi}
	+
	\inner{c(\omega)\pt^*\psi,\pt^*\psi}
	\\
	&&-
	(q+1)
	\inner{
		(\Box_{\dbar}-1)^{-1}
		(A_s\cup\psi)
		,
		A_s\cup\psi
	}
	+
	\inner{
		\Box_\dbar^{-1}
		\paren{A_{\ol s}\cup\pt^*\psi}
		,
		\paren{A_{\ol s}\cup\pt^*\psi}
	}.
\end{eqnarray*}
\end{theorem}

\begin{remark}
	By the previous argument using eigenfunctions of Laplacians, one has
	\begin{equation*}
		\inner{
			(\Box_{\dbar}-1)^{-1}
			(A_s\cup\psi)
			,
			A_s\cup\psi
		}
		=
		-
		\norm{H(A_s\cup\psi)}^2
		+
		\inner{
			(\Box_{\dbar}-1)^{-1}
			(A_s\cup\psi)^\perp
			,
			(A_s\cup\psi)^\perp
		},
	\end{equation*}
	where the second term is nonnegative.
\end{remark}

\section{Fiberwise hermitian flat case}
\label{S:flat}
In this section we will consider the following situation.
\begin{itemize}
\item $\ii\Theta_h(\cE)$ is hermitian flat on fibers.
\item There exists a fiberwise Kähler form $\omega$ on $\cX$.
\end{itemize}
Then the Bochner-Kodaira-Nakano formula says that we have
\begin{equation}\label{E:flat_BNK_formula}
\Box_\dbar-\Box_\pt=[\ii\Theta,\Lambda]=0
\end{equation}
on fibers.
Let $\psi$ be a local holomorphic section of $R^qf_*\Omega^p_{\cX/S}(\cE)$.
Then $\psi$ can be represented by a $\ol\pt$-closed $(p,q)$-form, which is $\dbar$-harmonic and also $\pt$-harmonic when restricted to fibers. This form will also be denoted by the same letter.
It follows from \eqref{E:curvature_formula} and Corollary~\ref{C:curvature_formula} that
\begin{align*}
R(\pt_s,\pt_{\ol s},\psi,\ol\psi)
&=
\inner{\Theta_E(v,\ol v)\psi,\psi}
-
\norm{L_v'\psi}^2
+
\norm{A_s\cup\psi}^2
+
\norm{L_{\ol v}'\psi}^2
-
\norm{A_{\ol s}\cup\psi}^2.
\end{align*}
In this case, we have
\begin{equation*}
		\dbar L_v'\psi
		=
		\pt(A_s\cup\psi)
		+
		\eta_s
			\we\psi
\end{equation*}
and
\begin{equation*}
		\dbar^* L_{\ol v}''\psi
		=
	(-1)^p\pt^*(A_{\ol s}\cup\psi)
	+
	[\Lambda,\ii\eta_{\ol s}]\psi.
\end{equation*}
where
\begin{equation*}
	\eta_s=-v\cup\Theta
	=
	-\Theta_{s\ol\beta}dz^{\ol\beta}
	\quad\text{and}\quad
	\eta_{\ol s}=-\ol v\cup\Theta
	=
	-\Theta_{\ol s\alpha}dz^\alpha
\end{equation*}
because of the fiberwise flatness of the hermitian structure.
Note that $\eta_s$ is harmonic by Proposition 4.8 and Proposition 4.9.
So we have
\begin{align*}
	\Box_\dbar \Theta(v,\ol v)
	=
	-g^{\ol\delta\gamma}
	\paren{
		\Theta_{s\ol s}
		+a\ind{s}{\alpha}{}\Theta_{\alpha\ol s}
		+a\ind{\ol s}{\ol\beta}{}\Theta_{s\ol\beta}
	}_{;\ol\delta\gamma}
	=:I_1+I_2+I_3.
\end{align*}
Indeed, the first term is already computed as $I_1=\ii\Lambda[\eta_s,\eta_{\ol s}]$ in~\cite{Geiger_Schumacher2017}.
The second term is computed as follows.
\begin{align*}
	I_2
	=
	-g^{\ol\delta\gamma}
	\paren{a\ind{s}{\alpha}{}\Theta_{\alpha\ol s}}_{;\ol\delta\gamma}
	=
	-g^{\ol\delta\gamma}
	\paren{a\ind{s}{\alpha}{;\ol\delta}\Theta_{\alpha\ol s}}_{;\gamma}
	=
	-g^{\ol\delta\gamma}
	\paren{A\ind{s}{\alpha}{\ol\delta}\Theta_{\alpha\ol s}}
	=
	\dbar^*\paren{A_s\cup\eta_{\ol s}}.
\end{align*}
Finally $I_3$ is computed as
\begin{align*}
	I_3
	&=
	-g^{\ol\delta\gamma}
	\paren{
		a\ind{\ol s}{\ol\beta}{}\Theta_{s\ol\beta}
	}_{;\ol\delta\gamma}
	=
	-g^{\ol\delta\gamma}
	\paren{
		a\ind{\ol s}{\ol\beta}{;\ol\delta}\Theta_{s\ol\beta}
		+
		a\ind{\ol s}{\ol\beta}{}\Theta_{s\ol\beta;\ol\delta}
	}_{;\gamma}
	\\
	&=
	-g^{\ol\delta\gamma}
	\paren{
		a\ind{\ol s}{\ol\beta}{;\ol\delta\gamma}\Theta_{s\ol\beta}
		+
		a\ind{\ol s}{\ol\beta}{;\gamma}\Theta_{s\ol\beta;\ol\delta}
		+
		a\ind{\ol s}{\ol\beta}{}\Theta_{s\ol\beta;\ol\delta\gamma}
	}
	\\
	&=
	-g^{\ol\delta\gamma}
	\paren{
		\Theta_{s\ol\beta}
		\paren{
			A\ind{\ol s}{\ol\beta}{\gamma;\ol\delta}
			-
			a\ind{\ol s}{\ol\tau}{}R\ind{\ol\tau}{\ol\beta}{\ol\delta\gamma}
		}
		+
		A\ind{\ol s}{\ol\beta}{\gamma}\Theta_{s\ol\beta;\ol\delta}
		+
		a\ind{\ol s}{\ol\beta}{}
		\paren{
			\Theta_{s\ol\beta;\gamma\ol\delta}
			+
			[\Theta_{s\ol\beta},\Theta_{\gamma\ol\delta}]
			+
			\Theta_{s\ol\tau}R\ind{\ol\beta}{\ol\tau}{\ol\delta\gamma}
		}
	}	
	\\
	&=
	-g^{\ol\delta\gamma}
	\paren{
		A\ind{\ol s}{\ol\beta}{\gamma}\Theta_{s\ol\beta}	
	}_{;\ol\delta}
	=
	-\pt^*\paren{A_{\ol s}\cup\eta_s}.
\end{align*}
Hence we have
\begin{equation*}
	\Theta(v,\ol v)
	=
	H(\Theta(v,\ol v))
	+
	\ii\Lambda[\eta_s,\eta_{\ol s}]
	+
	\dbar^*\paren{A_s\cup\eta_{\ol s}}
	-
	\pt^*\paren{A_{\ol s}\cup\eta_s}.	
\end{equation*}

For computing the terms $\norm{L_v'\psi}^2$ and $\norm{L_{\ol v}'\psi}^2$, we need the following proposition.
\begin{proposition}
	Suppose that $\eta_s$ is fiberwise parallel, in the sense that
	$$
	\nabla_{\cX/S}\eta_s
	:=
	\eta_{s\ol\beta;\alpha}
	dz^{\ol\beta}\otimes dz^{\alpha}
	+
	\eta_{s\ol\beta;\ol\delta}
	dz^{\ol\beta}\otimes dz^{\ol\delta}=0.
	$$
	Then $\eta_s\we\psi$ and $\ii[\Lambda,\eta_{\ol s}]\psi$ are harmonic.
\end{proposition}

Under the assumption that $\eta_s$ is parallel, the term $\norm{L_v'\psi}^2$ is computed as
\begin{align*}
\inner{L_v'\psi,L_v'\psi}
&=
\inner{G_{\dbar}\Box L_v'\psi,L_v'\psi}
=
\inner{\dbar^*G_{\dbar}\dbar L_v'\psi,L_v'\psi}
=
\inner{G_{\dbar}\dbar L_v'\psi,\dbar L_v'\psi}
\\
&=
\inner{G_{\dbar}\paren{\pt(A_s\cup\psi)+\eta_s\we\psi},\pt(A_s\cup\psi)+\eta_s\we\psi}
\\
&=
\inner{G_{\dbar}\paren{\pt(A_s\cup\psi),\pt(A_s\cup\psi)}}
=
\inner{G_{\pt}\pt^*\pt(A_s\cup\psi),A_s\cup\psi}
\\
&=
\inner{G_{\pt}\Box(A_s\cup\psi),A_s\cup\psi}
=
\norm{(A_s\cup\psi)^\perp}^2.
\end{align*}
Note that \eqref{E:flat_BNK_formula} implies that $G_\dbar$ and $G_\pt$ coincide.
Likewise, we have
\begin{align*}
\inner{L_{\ol v}'\psi,L_{\ol v}'\psi}
&=
\norm{(A_{\ol s}\cup\psi)^\perp}^2.
\end{align*}

\begin{theorem}\label{T:flat}
Let $\cX\rightarrow S$ be a family of compact Kähler manifolds and let $\cE\rightarrow\cX$ be a hermitian vector bundle on $\cX$ which is hermitian flat on fibers.
If $\eta_s$ is fiberwise parallel, then we have the following curvature formula for $R^qf_*\Omega^p_{\cX/S}(\cE)$.
\begin{align*}
	R(\pt_s,\pt_{\ol s},\psi,\ol\psi)
	&=
	\inner{H(\Theta(v,\ol v))\psi,\psi}
	+
	\inner{
	G_\dbar\paren{\ii\Lambda[\eta_s,\eta_{\ol s}]
	+
	\dbar^*\paren{A_s\cup\eta_{\ol s}}
	+
	\pt^*\paren{A_{\ol s}\cup\eta_s}}\psi,\psi}
		\\
		&+
		\norm{H(A_s\cup\psi)}^2
		-
		\norm{H(A_{\ol s}\cup\psi)}^2.
\end{align*}
\end{theorem}

\section{Families of Hermite-Einstein vector bundles}\label{S:HE}
In this section we will compute the curvature for holomorphic families of simple Hermite-Einstein vector bundles over a compact Kähler manifold; the case $p=0$ was treated in \cite{Geiger_Schumacher2017}.
More precisely, we consider the following setup:
\begin{itemize}
\item $(X,\omega)$ is a compact Kähler manifold
\item $(\cE,h)\rightarrow X\times S$ is a hermitian vector bundle.
\item The restriction $h_s$ of $h$ to $E_s:=\cE\vert_{X\times\set{s}}$ is a Hermite-Einstein metric for every $s\in S$.
\end{itemize}
Then the horizontal lift $v\in X\times S$ of $\pt_s\in S$ is a trivial lift, and the representative of Kodaira-Spencer class $A_s$ obviously vanishes.
Furthermore
$$
\eta_s=-\Theta_{s\ol\beta}dz^{\ol\beta} \; \text{ and } \; \eta_{\ol s}=-\Theta_{\alpha\ol s} dz^\alpha
$$
coincide with the harmonic \ks\ forms $\rho_s$ and their conjugates $\rho_{\ol s}$ resp.\ for the deformation of $E$ in the direction of $\pt_s$ (cf.\ also the section below).

In this sense the wedge product $\eta_s\we\text{\vartextvisiblespace}$ is equal to the cup product $\rho_s\cup\text{\vartextvisiblespace}$ from \cite{Geiger_Schumacher2017}, and
$$
[\Lambda,\ii\eta_{\ol s}]\psi_{A_p\ol B_q}=- g^{\ol\beta\alpha} \Theta_{\alpha\ol s}\psi_{A_p \ol\beta\ol\beta_2\ldots\ol\beta_q}
$$
is equal to the cap product $\rho^*_{\ol s} \cap\text{\textvisiblespace}$.
It is remarkable that $\rho^*_{\ol s} \cap \text{\textvisiblespace}$ is adjoint to $\rho_s\cup\text{\textvisiblespace}$.

So Equation~\eqref{E:curvature_formula} implies that
\begin{align*}
R(\pt_s,\pt_{\ol s},\psi,\ol\psi)
&=
\inner{\Theta(v,\ol v)\psi,\psi}
-
\norm{L_v'\psi}^2
+
\norm{L_{\ol v}'\psi}^2.
\end{align*}
By Proposition~\ref{P:dbar_Lv} and Proposition~\ref{P:dbar*_Lvbar} we have
\begin{align*}
R(\pt_s,\pt_{\ol s},\psi,\ol\psi)
&=
\inner{\Theta(\pt_s,\pt_{\ol s})\psi,\psi}
-
\inner{G_\dbar(\eta_s\we\psi),\eta_s\we\psi}
+
\inner{G_\dbar\paren{[\Lambda,\ii\eta_{\ol s}]\psi},[\Lambda,\ii\eta_{\ol s}]\psi}.
\end{align*}
It follows from Lemma~3.4 in \cite{Geiger_Schumacher2017} and \cite[(11)]{Schumacher_Toma1992} that $\Theta(\pt_s,\pt_{\ol s})=\Theta_{s\ol s}$ satisfies that
\begin{equation*}
\Box_\dbar\Theta_{s\ol s}
=
\ii\Lambda[\eta_s,\eta_{\ol s}].
\end{equation*}
This implies that the first term is computed as follows.
\begin{equation*}
\inner{\Theta_E(\pt_s,\pt_{\ol s})\psi,\psi}
=
\inner{H(\Theta_{s\ol s})\psi,\psi}
+
\inner{G_\dbar\paren{\ii\Lambda[\eta_s,\eta_{\ol s}]}\psi,\psi},
\end{equation*}
where $H(\Theta_{s\ol s})$ is the harmonic part of $\Theta_{s\ol s}$.


Altogether the following curvature formula holds:
\begin{theorem}\label{curvHE}
Let $\cE\rightarrow X\times S$ be a family of Hermite-Einstein vector bundles.
Then the curvature formula for $R^qf_*\Omega^p_{X\times S/S}(\cE)$ is
\begin{equation}\label{f:curvHE}
\begin{aligned}
R(\pt_s,\pt_{\ol s},\psi,\ol\psi)
=&
\inner{H(\Theta_{s\ol s})\psi,\psi}
+
\inner{G_\dbar(\ii\Lambda[\eta_s,\eta_{\ol s}])\psi,\psi}
\\
&-
\inner{G_\dbar(\eta_s\we\psi),\eta_s\we\psi}
+
\inner{G_\dbar\paren{[\Lambda,\ii\eta_{\ol s}]\psi},[\Lambda,\ii\eta_{\ol s}]\psi}.
\end{aligned}
\end{equation}
\end{theorem}


We will see below that in a moduli theoretic situation the first term above is not present.

It seems to be open, whether or not the term $G_\dbar(\ii\Lambda[\eta_s,\eta_{\ol s}])$ gives rise to a positive operator.

\section{Application to moduli spaces and \wp\ metrics}
For families of holomorphic vector bundles, as usual  tangent vectors of the parameter space are mapped to elements of $H^1(X, End(\cE_s))$ under the \ks\ map so that the curvature of $R^1f_*(End(\cE))$ is of interest.

Moduli spaces of holomorphic vector bundles were introduced under the assumption of simplicity which guaranteed that universal families exist, whose parameter spaces could be glued together. The \he\ condition was used to guarantee the Hausdorff property of the resulting moduli space.

Given a holomorphic family $\cE$ of simple \he\ bundles on $X$ parameterized by $S$ the \wp\ or $L^2$-(semi)norm of a tangent vector $\pt_s$ at a point $s\in S$ is defined as follows:
The \ks\ class of $\pt_s$ is represented by a harmonic form $\eta_s \in \cA^{0,1}(\cX_s,End(\cE_s))$ as above, and
$$
\|\pt_s\|^2_{WP}=\|\eta_s\|^2_{WP} = \int_{\cX_s} \ii \Lambda(\eta_s \we_h \eta_{\ol s})\, g \, dV=  \int_{\cX_s} \rm{tr}(\ii \Lambda(\eta_s \we \eta^*_{s}))\, g \, dV\, .
$$

When studying the geometry of moduli spaces of stable, holomorphic vector bundles on compact \ka\ manifolds, it is assumed that the respective determinant line bundles in a holomorphic family are constant. The line bundle contribution can be discussed separately. Note that
$$
{\rm tr}(\ii\Theta_\cE) = - \ddb \log \det(h_{i\ol\jmath}) = 2\pi c_1(\det \cE,\det(h_{i\ol\jmath}))
$$
so that infinitesimal deformations with constant determinant bundle are given by elements of $H^1(X,End^0(E_s))$.
For any form $\eta_s$ the form ${\rm tr} (\eta_s)$ is again harmonic, hence equal to zero. In this way in the curvature formula $End(\cE)$ can be replaced  by the space  $End^0(\cE)$ of endomorphisms with vanishing traces.

Given a family $\cE$ of holomorphic vector bundles over a compact \ka\ manifold $X$, parameterized by $S$, we will need the curvature of $R^pf_* End(\cE)$, in particular for $p=1$. Then the previous arguments can be adopted in the following way. The connection $\theta_\cE$ and curvature $\Theta_\cE$ induce connections on $End(\cE)= \cE \otimes \cE^*$ so that
$$
\Theta_{\cE\otimes \cE^*}= [\Theta_\cE,\vartextvisiblespace],
$$
in particular the harmonic forms $\eta_s$ act on $End(\cE)$ by $[\eta_s,\vartextvisiblespace]$.

\begin{proposition}
  The curvature of the \wp\ metric for families of simple \he\ bundles is given by the curvature for
  $$
  R^1f_*End(\cE)\,.
  $$
  When applying \eqref{f:curvHE} to sheaves $ R^qf_*End(\cE)$ the action of curvature terms of $\cE$ is the Lie commutator of bundle valued differential forms, in particular $\inner{H(\Theta_{s\ol s})\psi,\psi} =\inner{[H(\Theta_{s\ol s}),\psi],\psi}=0$ for all $\psi$.
  If the determinants $\det(\cE_s)$ are constant, then $End(\cE)$ is replaced by $End^0(\cE)$.
\end{proposition}

\subsection{Curvature of the Weil-Petersson metric for simple \he\ bundles}
We consider a universal family of simple, \he\ vector bundles on a compact \ka\ manifold $X$ with constant determinant line bundle over a parameter space $S$ such that (at any base point) for any tangent vector $\pt_s$ the term $H(\Theta_{s\ol s})$ vanishes.

The \wp\ form is known to be \ka. When computing its curvature the sheaf $R^1f_*End^0(\cE)$ on for $f: X\times S\to S$ for the original family $\cE$ is of interest. A tangent vector $\pt_s$ of the parameter space is again represented by $\eta_s$, but in this case its action is on endomorphisms rather than sections is of relevance. Hence the following theorem holds (\cite{Schumacher_Toma1992}).
\begin{theorem}\label{T:curvWP}
The holomorphic bisectional curvature of the \wp\ metric for tangent vectors given by $\eta_s$ and $\kappa$ is determined by
\begin{eqnarray*}
  R(\pt_s,\pt_{\ol s},\kappa,\ol{\kappa})&=&
\inner{G_\dbar(\ii\Lambda[\eta_s,\eta_{\ol s}]), \ii\Lambda[\kappa,\ol{\kappa}]}\\
& & \qquad - \inner{G_\dbar([\eta_s,\kappa]),[\eta_s,\kappa]}\\
& & \qquad +
\inner{G_\dbar(\ii \Lambda[\eta_{\ol s}, \kappa]),\ii \Lambda[\eta_{\ol s}, \kappa]}
\end{eqnarray*}
For $\kappa = \eta_s$ the formula determines the holomorphic sectional curvature. In this case
$$
  R(\pt_s,\pt_{\ol s},\eta_s,{\eta_{\ol s}})=
2 \inner{G_\dbar(\ii\Lambda[\eta_s,\eta_{\ol s}]), \ii\Lambda[\eta_s,{\eta_{\ol s}}]}
  - \inner{G_\dbar([\eta_s,\eta_s]),[\eta_s,\eta_s]},
$$
where the first term is non-negative, and the second is non-positive.

\end{theorem}
\begin{proof}
The result follows from \eqref{f:curvHE} applied to the endomorphism bundle: The first term is already known to vanish. Denote the endomorphism $G_\dbar(\ii\Lambda[\eta_s,\eta_{\ol s}])$ by $\chi$. Then the pointwise inner product for the second term equals
\begin{eqnarray*}
([\chi,\kappa],\kappa) &=& Tr(\ii\Lambda([\chi,\kappa] \we \kappa^* )) = Tr(\ii\Lambda(\chi\kappa - \kappa\chi)\we \kappa^*) \\ &=& Tr(\chi \ii\Lambda( (\kappa \we \kappa^* + \kappa^* \we \kappa))) = ( \chi , \ii\Lambda [\kappa,\ol\kappa]   ).
\end{eqnarray*}
\end{proof}

\subsection{\wp\ metric for fiberwise flat hermitian bundles}
In the situation of the previous section we assume that the fibers $\cE_s$ are flat, i.e.\
$$
\Theta|_{\cX_s} = \Theta_{\alpha\ol\beta}\, dz^\alpha\we dz^{\ol\beta}=0.
$$
Now the forms $\eta_s$ satisfy $\eta_s=-\Theta_{s\ol\beta} dz^{\ol\beta}$.
\begin{lemma}
	For  all directions $\pt_{s^i}$, $\pt_{s^k} $ the $End(\cE_s)$-valued $(0,2)$-forms $\frac{1}{2}[\eta_i, \eta_k]$ in particular the forms $\eta^2_s:=\frac{1}{2}[\eta_s,\eta_s]$ are harmonic.
\end{lemma}
\begin{proof}
	$$
	\ol\pt^*(\eta_i\we\eta_k)=-g^{\ol\delta\gamma}(\eta_{i\ol\beta}\we\eta_{k\ol\delta})_{; \gamma}dz^{\ol\beta} =g^{\ol\delta\gamma}(\Theta_{\gamma\ol\beta;i}\we\eta_{k\ol\delta})dz^{\ol\beta}=    0.
	$$
\end{proof}
The \ks\ map of order $p$ can be defined on the $q$-th symmetric power of the tangent space of $S$.
$$
\rho^q_{S}: S^q\cT_{S} \to R^qf_*End(\cE)
$$
by
$$
\rho^q_{S}({\pt_{i_1}\otimes\ldots \otimes\pt_{i_q}}) = \frac{1}{q!}\sum \eta_{s_{\nu_1}}\we\ldots\we\eta_{s_{\nu_q}},
$$
where the wedge product also involves the composition of endomorphisms (which implies commutators by skew-symmetricity) and the sum is taken over all permutations. In this sense $\eta^q_s= \eta_s\wedge\ldots\wedge \eta_s$ is defined and parallel.

Now we apply the curvature formula of Theorem~\ref{curvHE} to $End(\cE)$ with $\psi=\eta^q_s$.  Because of the harmonicity of $\eta^q_s$ the term that carries a negative sign vanishes.
\begin{theorem}\label{curvflat}
	Given a family of fiberwise flat hermitian bundles for $q=1,\ldots,n$ the curvature of $R^qf_*End(\cE)$ satisfies
\begin{eqnarray*}
   R(\pt_s,\pt_{\ol s},\eta^q_s,\eta^q_{\ol s})&=&
\inner{G_\dbar(\ii\Lambda[\eta_s,\eta_{\ol s}]) \eta^q_s  , \eta^q_s}\\
& & \qquad +
\inner{G_\dbar(\ii \Lambda[\eta_{\ol s}, \eta^q_s]),\ii \Lambda[\eta_{\ol s}, \eta^q_s]}
\end{eqnarray*}
\end{theorem}
Let $q=1$. Then the above results imply the following.
\begin{corollary}
	For families of fiberwise flat hermitian bundles, in particular on the moduli space of simple, flat bundles (with trivial determinant bundle) the holomorphic sectional curvature
	$$
   R(\pt_s,\pt_{\ol s},\eta_s,\eta_{\ol s})=
2\inner{G_\dbar(\ii \Lambda[\eta_{\ol s}, \eta_s]),\ii \Lambda[\eta_{\ol s}, \eta_s]}
	$$
	of the \wp\ metric is semi-positive.
\end{corollary}

\bibliographystyle{amsalpha}

\begin{thebibliography}{BPW22}

\bibitem[Ber09]{Berndtsson2009}
Bo~Berndtsson, \emph{Curvature of vector bundles associated to holomorphic fibrations}, Ann. of Math. (2) \textbf{169} (2009), no.~2, 531--560. \MR{2480611}

\bibitem[Ber11]{Berndtsson2011}
\bysame, \emph{Strict and nonstrict positivity of direct image bundles}, Math. Z. \textbf{269} (2011), no.~3-4, 1201--1218. \MR{2860284}

\bibitem[BP08]{Berndtsson_Paun2008}
Bo~Berndtsson and Mihai P\u{a}un, \emph{Bergman kernels and the pseudoeffectivity of relative canonical bundles}, Duke Math. J. \textbf{145} (2008), no.~2, 341--378. \MR{2449950}

\bibitem[BPW22]{Berndtsson_Paun_Wang2022}
Bo~Berndtsson, Mihai P\u{a}un, and Xu~Wang, \emph{Algebraic fiber spaces and curvature of higher direct images}, J. Inst. Math. Jussieu \textbf{21} (2022), no.~3, 973--1028. \MR{4404131}

\bibitem[Dem97]{Demailly(Book)}
Jean-Pierre Demailly, \emph{Complex analytic and differential geometry}, Citeseer, 1997.

\bibitem[Dem12]{Demailly(LN)}
\bysame, \emph{Analytic methods in algebraic geometry}, vol.~1, International Press Somerville, MA, 2012.

\bibitem[Gri69]{Griffiths1969}
Phillip~A. Griffiths, \emph{Some results on algebraic cycles on algebraic manifolds}, Algebraic {G}eometry ({I}nternat. {C}olloq., {T}ata {I}nst. {F}und. {R}es., {B}ombay, 1968), Tata Inst. Fundam. Res. Stud. Math., vol.~4, Tata Inst. Fund. Res., Bombay, 1969, pp.~93--191. \MR{257092}

\bibitem[Gri70]{Griffiths1970}
\bysame, \emph{Periods of integrals on algebraic manifolds. {III}. {S}ome global differential-geometric properties of the period mapping}, Inst. Hautes \'{E}tudes Sci. Publ. Math. (1970), no.~38, 125--180. \MR{282990}

\bibitem[GS17]{Geiger_Schumacher2017}
Thomas Geiger and Georg Schumacher, \emph{Curvature of higher direct image sheaves}, Higher dimensional algebraic geometry---in honour of {P}rofessor {Y}ujiro {K}awamata's sixtieth birthday, Adv. Stud. Pure Math., vol.~74, Math. Soc. Japan, Tokyo, 2017, pp.~171--184. \MR{3791213}

\bibitem[GT84]{Griffiths_Tu1984}
Phillip Griffiths and Loring Tu, \emph{Curvature properties of the {H}odge bundles}, Topics in transcendental algebraic geometry ({P}rinceton, {N}.{J}., 1981/1982), Ann. of Math. Stud., vol. 106, Princeton Univ. Press, Princeton, NJ, 1984, pp.~29--49. \MR{756844}

\bibitem[HH21]{Hu_Huang2021}
Zhi Hu and Pengfei Huang, \emph{Curvature formulas related to a family of stable {H}iggs bundles}, Comm. Math. Phys. \textbf{387} (2021), no.~3, 1491--1514. \MR{4324383}

\bibitem[LSY13]{Liu_Sun_Yang2013}
Kefeng Liu, Xiaofeng Sun, and Xiaokui Yang, \emph{Positivity and vanishing theorems for ample vector bundles}, J. Algebraic Geom. \textbf{22} (2013), no.~2, 303--331. \MR{3019451}

\bibitem[LY14]{Liu_Yang2014}
Kefeng Liu and Xiaokui Yang, \emph{Curvatures of direct image sheaves of vector bundles and applications}, J. Differential Geom. \textbf{98} (2014), no.~1, 117--145. \MR{3263516}

\bibitem[Mou97]{Mourougane1997}
Christophe Mourougane, \emph{Images directes de fibr\'{e}s en droites adjoints}, Publ. Res. Inst. Math. Sci. \textbf{33} (1997), no.~6, 893--916. \MR{1614576}

\bibitem[MT07]{Mourougane_Takayama2007}
Christophe Mourougane and Shigeharu Takayama, \emph{Hodge metrics and positivity of direct images}, J. Reine Angew. Math. \textbf{606} (2007), 167--178. \MR{2337646}

\bibitem[MT08]{Mourougane_Takayama2008}
\bysame, \emph{Hodge metrics and the curvature of higher direct images}, Ann. Sci. \'{E}c. Norm. Sup\'{e}r. (4) \textbf{41} (2008), no.~6, 905--924. \MR{2504108}

\bibitem[Nau21]{Naumann2021}
Philipp Naumann, \emph{Curvature of higher direct images}, Ann. Fac. Sci. Toulouse Math. (6) \textbf{30} (2021), no.~1, 171--201. \MR{4269140}

\bibitem[OT87]{Ohsawa_Takegoshi1987}
Takeo Ohsawa and Kensho Takegoshi, \emph{On the extension of {$L^2$} holomorphic functions}, Math. Z. \textbf{195} (1987), no.~2, 197--204. \MR{892051}

\bibitem[PT18]{Paun_Takayama2018}
Mihai P\u{a}un and Shigeharu Takayama, \emph{Positivity of twisted relative pluricanonical bundles and their direct images}, J. Algebraic Geom. \textbf{27} (2018), no.~2, 211--272. \MR{3764276}

\bibitem[Sch12]{Schumacher2012}
Georg Schumacher, \emph{Positivity of relative canonical bundles and applications}, Invent. Math. \textbf{190} (2012), no.~1, 1--56. \MR{2969273}

\bibitem[Siu82]{Siu1982}
Yum~Tong Siu, \emph{Complex-analyticity of harmonic maps, vanishing and {L}efschetz theorems}, J. Differential Geometry \textbf{17} (1982), no.~1, 55--138. \MR{658472}

\bibitem[Siu86]{Siu1986}
\bysame, \emph{Curvature of the {W}eil-{P}etersson metric in the moduli space of compact {K}\"{a}hler-{E}instein manifolds of negative first {C}hern class}, Contributions to several complex variables, Aspects Math., vol.~E9, Friedr. Vieweg, Braunschweig, 1986, pp.~261--298. \MR{859202}

\bibitem[ST92]{Schumacher_Toma1992}
Georg Schumacher and Matei Toma, \emph{On the {P}etersson-{W}eil metric for the moduli space of {H}ermite-{E}instein bundles and its curvature}, Math. Ann. \textbf{293} (1992), no.~1, 101--107. \MR{1162675}

\bibitem[Var10]{Varolin2010}
Dror Varolin, \emph{Three variations on a theme in complex analytic geometry}, Analytic and algebraic geometry, IAS/Park City Math. Ser., vol.~17, Amer. Math. Soc., Providence, RI, 2010, pp.~183--294. \MR{2743817}

\bibitem[Wan16]{Wang2016Ar}
Xu~Wang, \emph{Curvature of higher direct image sheaves and its application on negative-curvature criterion for the weil-petersson metric}, 2016.

\end{thebibliography}

\providecommand{\bysame}{\leavevmode\hbox to3em{\hrulefill}\thinspace}
\providecommand{\MR}{\relax\ifhmode\unskip\space\fi MR }
\providecommand{\MRhref}[2]{%
  \href{http://www.ams.org/mathscinet-getitem?mr=#1}{#2}
}
\providecommand{\href}[2]{#2}

\end{document}